\documentclass{article}

\usepackage{epsfig}
\usepackage{amssymb}
\usepackage[]{hyperref}

\usepackage[usenames,dvipsnames]{pstricks}
\usepackage{pst-grad} 
\usepackage{pst-plot} 

\usepackage{tikz}
\usepackage{ifthen}
\usepackage{float}

\usepackage{amsmath,amsthm,verbatim,amssymb,amsfonts,amscd, graphicx}
\usepackage{graphics}
\usepackage[all,cmtip,2cell]{xy}
\UseTwocells
\usepackage{xcolor}
\usepackage{tikz}
\numberwithin{figure}{section}
\usetikzlibrary{arrows}
\usepackage{float}
\usepackage{MnSymbol}
\usepackage{wasysym}

\usepackage{color}

\usepackage{tikz-cd}
\tikzset{%
	symbol/.style={%
		draw=none,
		every to/.append style={%
			edge node={node [sloped, allow upside down, auto=false]{$#1$}}}
	}
}

\def\Ee{\mathcal{E}}

\def\Alg{\mathrm{Alg}}

\def\Set{\mathrm{Set}}

\def\Poly{\mathbf{Poly}}

\def\Cat{\mathrm{Cat}}
\def\Int{\mathrm{Int}}
\def\sg{\sigma}

\def\LL{\mathbb{L}}
\def\hocolim{\mathrm{hocolim}}

\def\holim{\mathrm{holim}}

\def\LL{\mathbb{L}}

\theoremstyle{plain}
\newtheorem{theorem}[subsection]{Theorem}
\newtheorem{lemma}[subsection]{Lemma}
\newtheorem{proposition}[subsection]{Proposition}

\newtheorem{defin}[subsection]{Definition}

\newtheorem{lem}[subsection]{Lemma}
\newtheorem{corol}[subsection]{Corollary}
\theoremstyle{remark}
\newtheorem{remark}[subsection]{Remark}
\newtheorem{example}[subsection]{Example}

{\begin{itemize}\tt}%
{\end{itemize}}

\makeatletter
\newcommand{\subjclass}[2][1991]{%
	\let\@oldtitle\@title%
	\gdef\@title{\@oldtitle\footnotetext{#1 \emph{Mathematics subject classification.} #2}}%
}
\newcommand{\keywords}[1]{%
	\let\@@oldtitle\@title%
	\gdef\@title{\@@oldtitle\footnotetext{\emph{Key words and phrases.} #1.}}%
}
\makeatother

\newcommand{\F}{\mbox{$\mathcal F$}}

\begin{document}

\title{Polynomial monads and delooping of mapping spaces}

\date{\today}

\author{
	M. A.  Batanin\thanks{Macquarie University, North Ryde, 2109 Sydney, Australia}\\
	\texttt{michael.batanin@mq.edu.au}
	\and
	F. De Leger\thanks{Macquarie University, North Ryde, 2109 Sydney, Australia}\\
	\texttt{florian.deleger@mq.edu.au}
}

\maketitle

\renewcommand{\thefootnote}{\fnsymbol{footnote}} 
\renewcommand{\thefootnote}{\arabic{footnote}} 

\setcounter{tocdepth}{1}

\begin{abstract} We extend some classical results -- such as Quillen's Theorem A, the Grothendieck construction, Thomason's theorem and the characterisation of homotopically cofinal functors -- from the homotopy theory of small categories to polynomial monads and their algebras. 

As an application we give a  categorical proof of the Dwyer-Hess and Turchin results 
concerning the explicit double delooping of spaces of long knots. 
\end{abstract}

{\small \tableofcontents}

\part{Polynomial monads and homotopy theory}

\maketitle

\section{Introduction}

The homotopy theory of small categories  is a product of  decades of development by many prominent mathematicians; to name a few:  Quillen, Grothendieck and Thomason \cite{Gr,Q,Th}. 
More recent significant progress is the work of  Maltsiniotis and Cisinski \cite{Malt,cis06}.  
The theory provides a vital formalism for many applications in algebraic geometry, algebraic $K$-theory and algebraic topology. 

In this paper we show that some of the fundamental constructions and results from the homotopy theory of small categories are still valid in the larger context of the category of finitary polynomial monads. 
The theory of finitary polynomial monads (equivalently known as $\Sigma$-free coloured $Set$-operads) is a multivariable extension of the theory of small categories.
Indeed, a small  category $\mathrm{C}$  with the set of objects $I$ determines an endofunctor $C:\Set^I\to \Set^I,$ where $\Set^I$ is the category of $I$-indexed collections of sets:
$$C(X)(i) = \coprod_{j}\mathrm{C}(j,i)\times X(j),$$ 
where $\mathrm{C}(j,i)$ is the set of morphisms in $\mathrm{C}$ from $j$ to $i.$ It is easy to see that the functor $C$ preserves connected limits.
The category structure  of $\mathrm{C}$  amounts then to a structure of cartesian monad on the functor $C.$ 
The last conditions mean that the unit and multiplication of this monad are cartesian natural transformations; that is, all naturality squares are pullbacks.
The category of algebras for this monad is isomorphic to the category of covariant presheaves  $\Set^{\mathrm{C}}.$

Finitary polynomial monads can be defined as cartesian monads whose underlying functor is a coproduct of sets which involve multivariable summands like $B(j_1,\ldots, j_k; i)\times X(j_1)\times\ldots\times X(j_k)$ with finite{ly} many factors. 
The category of algebras of such a monad is the category of $I$-collections equipped with the operations 
$$B(j_1,\ldots, j_k; i)\times X(j_1)\times\ldots\times X(j_k)\to X(i)$$
which satisfy appropriate  associativity and unitarity conditions. One can also give the structure of an  
algebra by specifying a family of maps  
$$b: X(j_1)\times\ldots\times X(j_k)\to X(i) \ , \          b\in B(j_1,\ldots, j_k; i).$$
An important fact is: the category of algebras of a finitary polynomial monad can be defined in any symmetric monoidal category $\Ee.$ Indeed, it suffices to replace the cartesian product of sets in the definition of algebra by the tensor product of objects in $\Ee.$  

It was observed by the first author in \cite{EHBat} that the algebras of a polynomial monad in the symmetric monoidal category $(\Cat,\times, 1)$ of small categories (called {\it categorical algebras}) play a special role. 
For such algebras the theory can be internalised; that is, one can consider a kind of algebra (called internal algebra) inside a categorical algebra of a polynomial monad. 
Formally, an internal $T$-algebra inside a categorical $T$-algebra $A$ is defined as a lax $T$-algebra map from the terminal $T$-algebra $1$ to $A.$

A good example to have in mind is the category of monoids in any monoidal category. 
We consider a monoidal category as a categorical (pseudo)algebra of a finitary polynomial monad 
(the free monoid monad) 
$M$ given by the geometric series $M(X) = \coprod_{n\ge 0} X^n.$ 
It is well known then that   
a monoid in a monoidal category $A$ is the same thing as a lax-monoidal functor from the terminal monoidal category $1$ to $A$; that is, a lax $M$-algebras map $1\to A.$ 
\begin{remark} More precisely, the categorical algebras of $M$ are strict monoidal categories. 
However, this difference between strict monoidal categories and general monoidal categories 
does not play much role in our theory due to Mac Lane's coherence theorem. 
\end{remark}
A classical observation of Lawvere is that the theory of monoids is represented by the monoidal category of finite ordinals $\Delta_{alg}$ in the sense that this monoidal category is freely generated by a monoid (the terminal ordinal) inside it. This means that a monoid in a monoidal category $A$ is the same as a strict monoidal functor from $\Delta_{alg}$ to $A.$

It was shown in \cite{EHBat} that this observation of Lawvere has a far reaching generalisation:  for any cartesian map between cartesian monads $f:S\to T$,  there exists a  categorical $T$-algebra $T^S$ with a nice universal property: it is freely generated by an internal $S$-algebra. We called this algebra {\it the classifier of internal $S$-algebras inside categorical $T$-algebras}. 

The theory of classifiers provides a link between Grothendieck's homotopy theory and polynomial monads.  
For example, if $f:S\to T$ is a functor between small categories (interpreted as a map between the corresponding cartesian monads) then $T^S$ is the {covariant $\Cat$-valued presheaf 
on $T$} which associates to an object $i\in T$ the comma (or slice) category $f/i.$ 
This slice category construction is one of the main tools of the homotopy theory of small categories.  

The theory of internal algebras classifiers for polynomial monads and its applications  was developed further  by the first author and Clemens Berger in \cite{BB}. An application of this theory to the Baez and Dolan stabilization hypothesis for higher dimensional categories was found in \cite{BatStab}. It was observed that the homotopy type of the classifier for $f:S\to T$ can tell us a lot about the homotopy behaviour of the adjoint pair of functors between simplicial algebras of $T$ and $S$ induced by $f,$ very much like in the homotopy theory of small categories where the homotopy type of slices of a functor provides important information about homotopy Kan extensions along this functor.   

In this paper we take this analogy seriously and develop a formalism extending that of the homotopy theory of small categories. The main ingredients of our new formalism are 
:
\begin{enumerate}
\item An analogue of the Grothendieck construction for a polynomial monad and interpretation of the classifier construction as its left adjoint;
\item An analogue of Quillen's Theorem A for polynomial monads;
\item The analogue of the characterisation of homotopically cofinal functors in terms of preservation of homotopy limits. 
\item A generalisation of Thomason's theorem about the homotopy colimit of  nerves of a  diagram of small categories.
\end{enumerate}

It turns out that this extended formalism provides some extra flexibility which is not achievable in the classical setting of small categories.  For example, one can add constants to the theory, which turns out to be very useful in the study of homotopy mapping spaces between simplicial algebras.  

As an illustration of the power of this formalism we give a new proof of the Dwyer-Hess-Turchin result on explicit double delooping of the space of long knots \cite{DH,T}. Our work was, in fact, inspired by the paper of Turchin \cite{T}.

The space
$\overline{Emb}(\mathbb{R}^1,\mathbb{R}^m)$ 
of {\it long knots modulo immersions} is a homotopy fiber of the map 
$$Emb(\mathbb{R}^1,\mathbb{R}^m) \to Imm(\mathbb{R}^1,\mathbb{R}^m)$$
where $Emb(-,-)$ is the space of embeddings and $Imm(-,-)$ is the space of immersions with compact support (that is, it is equal to the standard embedding outside of a compact subspace).  
	
Dwyer-Hess and independently Turchin proved the following statement: for $m>3$, there is a weak equivalence of spaces
	\begin{equation}\label{DHconjecture1}
	\overline{Emb}(\mathbb{R}^1,\mathbb{R}^m) \sim \Omega^2 Map_{\mathrm{Op}} (\mathcal{D}_1,\mathcal{D}_m)	\end{equation}	
where $\mathcal{D}_k$ is a $E_k$-operad (that is, any operad homotopy equivalent to the little $k$-disks operad), $ Map_{\mathrm{Op}}(-,-)$
is the homotopy mapping space in the category of symmetric operads and $\Omega$ is the loop space functor. 
			
		Both proofs use an earlier result of Sinha \cite{Sinha} about the weak equivalence: 
	\begin{equation}\label{fiberembeddingspace}
	\overline{Emb}(\mathbb{R}^1,\mathbb{R}^m) \sim \widetilde{Tot}(\mathcal{K})
	\end{equation}
	where $\widetilde{Tot}(\mathcal{K})$ is the cosimplicial totalization of the Kontsevich operad $\mathcal{K}.$ This construction is possible because $\mathcal{K}$ is not only an $E_m$-operad but it is also multiplicative; that is, it is equipped with a map of non-symmetric operads $\mathcal{A}ss\to d_1\mathcal{K},$ where $\mathcal{A}ss$ is the terminal non-symmetric operad and $d_1(-)$ is the functor of dessymmetrisation (that is, the functor forgetting the symmetric groups actions).

In fact, Dwyer-Hess and Turchin established  the following more general delooping result:
\begin{theorem} For any multiplicative reduced non-symmetric operad $\mathcal{O}$, there are two 
weak equivalences of mapping spaces:
$$Map_{\mathrm{Bimod}}(\mathcal{A}ss,\mathcal{O}) \sim \Omega Map_{\mathrm{NOp}} (\mathcal{A}ss,\mathcal{O})$$
and 
$$Map_{\mathrm{WBimod}}(\mathcal{A}ss,\mathcal{O}) \sim \Omega Map_{\mathrm{Bimod}} (\mathcal{A}ss,\mathcal{O}),$$	
where mapping spaces are taken in the model categories of non-symmetric operads $\mathrm{NOp}$, $\mathcal{A}ss$-bimodules $\mathrm{Bimod}$ and weak $\mathcal{A}ss$-bimodules $\mathrm{WBimod}.$  Reducedness means that $\mathcal{O}_0$  and  $\mathcal{O}_1$ are both contractible spaces. \end{theorem}
This theorem is applicable to delooping the space of long knots because (as several people observed) the category $\mathrm{WBimod}$ is isomorphic to the category of cosimplicial objects and, hence, the (homotopy) cosimplicial totalization functor is the homotopy mapping space from the terminal weak bimodule.     
		
	The proofs of this theorem by Dwyer-Hess and Turchin are both based on homotopy theory {but of different flavours. Turchin uses some very explicit cofibrant resolutions for operads, bimodules and weak bimodules and then constructs all necessary higher homotopies by hand. Dwyer and Hess use model theoretical argument related to moduli spaces (in fact they prove a more general statement about relations between homotopy mapping spaces of monoids and bimodules). Unfortunately, both proofs are very technical and do not provide a clear conceptual explanation of the result. Consequently both proofs are hard to generalise to other situations where we want to study the delooping of mapping spaces; for example, for delooping the higher dimensional spaces of embeddings.

We approach this question by observing that all three categories  $\mathrm{NOp}$, $\mathrm{Bimod}$ and $\mathrm{WBimod}$	are categories of algebras of appropriate polynomial monads. The mapping spaces like $Map_{\mathrm{NOp}} (\mathcal{A}ss,\mathcal{O})$ are then the `derived versions' of the category of internal algebras. So the statement of the theorem can be conceptually understood in the setting of internal algebras. 

For example, the first delooping statement can be understood at the outset through the following `baby' case.  Suppose we are given two maps of non-symmetric operads $\mathcal{A}ss\to \mathcal{O}$ in $\Cat.$  One can then construct an $\mathcal{A}ss$-bimodule out of $\mathcal{O}$ using the first map to define a left action of $\mathcal{A}ss$  and the second map to define a right action. 
 Now, suppose that $\mathcal{O}$ is a groupoid in each degree and $\mathcal{O}_1$ is a contractible groupoid.   Then one can prove by hand that the process above has an inverse; that is, any bimodule structure on the collection of categories $\mathcal{O}$ is obtained from two maps of operads $\mathcal{A}ss\to \mathcal{O}.$  
In general, Turchin-Dwyer-Hess delooping is essentially the above statement where $\mathcal{O}$ is now an operad in $\omega$-groupoids. 
Of course, in this case the inverse functor reconstructs two operadic maps as well as an operadic structure on $\mathcal{O}$, but only up to higher homotopies. 
In our formalism this statement is equivalent to the statement that the map of polynomial monads $$Bimod_+\to NOp_{**}$$ is homotopically cofinal (Theorem \ref{bullet}).  
Here, $Bimod_+$ is the polynomial monad for $\mathcal{A}ss$-bimodules with an additional distinguished point in the degree $1$ component and $NOp_{**}$ is the polynomial monad for double multiplicative operads; that is, non-symmetric operads equipped with two maps from $\mathcal{A}ss.$ 

We believe that the possibility of using the above kind of reasoning to conceptually understand a problem, and then apply formal high level homotopy theory language to finish the proof, illustrates an important powerful feature of our approach.  
We expect the technique to be very useful in future applications. 	 
	
Here is the plan of the paper. Part 1 is devoted to the formalism of the homotopy theory of polynomial monads. We define the category of Internal algebras of polynomial monads in Section 3. We then show that there is an analogue of the Grothendieck construction for polynomial monads which produce a polynomial monad map out of a categorical algebra of a polynomial monad. 
We then interpret the category of internal algebras as a category of sections of Grothendieck constructions. In this sense the category of internal algebras can be thought as the nonabelian cohomology of the polynomial monad. The category of relative internal algebras also admits an interpretations in terms of a category of liftings. 
In Section 3 we recall the definition of internal algebra classifiers and their construction in terms of  a codescent object. 
Our new result here is that the classifier construction is the left $2$-adjoint to the Grothendieck construction. 	

In Section 5 we relate internal algebra classifiers with the homotopy theory of simplicial algebras over a polynomial monad. Here we prove Quillen's Theorem A for polynomial monads. 
We also introduce the notion of homotopically cofinal maps between polynomial monads and show that these maps can be characterised in terms of maps between mapping spaces between algebras very much as homotopically cofinal functor can be characterised as functors restriction along which preserves homotopy colimits. 

In Section 6 we develop yet another version of Grothendieck construction which we call twisted Boardman-Vogt tensor product. It is interesting that these two versions of Grothendieck constructions coincide in the $2$-category of small categories.  We prove then a generalisation to polynomial monads of Thomason's theorem about the colimit of the diagram of nerves of small categories. In Section 7 we apply this criteria to introduce a notion of homotopically cofinal square of polynomial monads.  These are exactly the commutative squares of polynomial monads which induce the homotopy pushouts of nerves of classifiers over any fixed polynomial monad.

In Section 8 we prove some useful results about formal delooping of homotopy mapping spaces between pointed algebras of polynomial monads. These results will play main role in our approach to the proof of the Dwyer-Hess-Turchin theorems in Part 2 of our paper. 
We must add that most of our results about homotopy mapping spaces can be proved for (semi) model categories of algebras of polynomial monads in a monoidal model category satisfying some very moderate assumptions. We do not do it in this paper just because it would make the proofs more technical and longer at the expense of clarity of exposition of the main new ideas. 

In Section 9 we briefly remind the reader about multiplicative operads, bimodules and weak bimodules, and show that there are polynomial monads for all these categories.     	

In Section 10 we prove a result about homotopy cofinality of a certain map between polynomial monads (Theorem \ref{bullet}). This result is a vast generalisation of the Dwyer-Hess-Turchin result where we compare mapping spaces between cospans of operads and bimodules over different operads. This is exactly this theorem which provides a conceptual explanation of the existence of Dwyer-Hess and Turchin delooping.  
An analogous result holds for mapping spaces between cospans of bimodules and two sided weak bimodules. 

In Section 11 we show how cofinality of the maps  $Bimod_+\to NOp_{**}$ and 
$WBimod_+\to Bimod_{**}$ follows from Theorem \ref{bullet}. Finally, in Section 12 we provide a proof of the Dwyer-Hess-Turchin theorem which is now a relatively simple consequence of our formal delooping theorems and the second cofinality Theorem \ref{cofinalmaps}.

{\it \underline{Latest development}.}	{A significant progress in explicit delooping of the embedding spaces  was made recently in \cite{Boavida,D,DCT}. The approach of Boavida and Weiss \cite{Boavida} is more topological, whereas our approach is categorical and combinatorial and is closer to the Dwyer-Hess  and Ducoulombier-Turchin approaches in \cite{D,DCT}. As it is noticed in \cite{Boavida} 
: `. . . the Dwyer-Hess result is a theorem about fairly general operads and as such it has a different scope and applicability from our result'. 
There are, however, very interesting connections between all these approaches and we are going to address it in a future paper.}  

\section{Polynomial monads}

\begin{defin}[\cite{KJBM,GK,BB}] A finitary polynomial $P$ is a diagram in $\Set$ of the form 
\[
			\xymatrix{
				J && E \ar[ll]_-s \ar[rr]^-p && B \ar[rr]^-t && I
			}
			\]
where $p^{-1}(b)$ is a finite set for any $b\in B.$ \end{defin} 
Each polynomial $P$ generates a functor called {\it polynomial functor} between functor categories $$\underline{P}:\Set^J \to \Set^I$$ which is defined as the composite functor
\[
			\xymatrix{
				\Set^J \ar[rr]^-{s^*}&& \Set^E  \ar[rr]^-{p_*} && \Set^B \ar[rr]^-{t_!} && \Set^I
			}
			\]
where we consider the sets $J,E,B, I$ as discrete categories and $s^*$ is the restriction functor and $p_*$ and $t_!$ are the right and left Kan extensions correspondingly.
Explicitly the functor $\underline{P}$ is given by the formula
\begin{equation}\label{PPP}  \underline{P}(X)_i = \coprod_{b\in t^{-1}(i)} \prod_{e\in p^{-1}(b)} X_{s(e)}, \end{equation}
which explains the name `polynomial': its values are sums of products of formal variables.

{\it A cartesian morphism between polynomial functors} is a natural transformation between the functors such that each naturality square is a pullback. 
One can prove that such a cartesian morphism is determined by a commutative diagram in $\Set$  
\[
			\xymatrix{
				J' \ar[dd] && E' 
				 \ar[ll]_-{s'} \ar[rr]^-{p'} \ar[dd] && B' \ar[rr]^-{t'} \ar[dd] && I' \ar[dd] \\
				\\
				J&& E \ar[ll]_-s \ar[rr]^-p && B \ar[rr]^-t && I
			}
			\]
			such that  the middle square is a pullback.
 
Composition of finitary polynomial functors is again a finitary polynomial functor. 
{Sets, finitary polynomial functors and their cartesian morphisms form a $2$-category  $\Poly$.} 
\begin{defin}A finitary polynomial monad is a monad in the $2$-category $\Poly.$ \end{defin}
\begin{remark} Finitary polynomial functors preserve filtered colimits and pullbacks. 
Polynomial monads are cartesian;  that is, their underlying functors preserve pullbacks and their units and multiplications are cartesian natural transformations. 
\end{remark}
\begin{remark} One can consider more general polynomial functors of non-finitary type. Since in this paper we shall not need these more general functors, we use the term polynomial monad for finitary polynomial monad.
\end{remark}
For a polynomial monad { $T$ }
 \begin{equation}\label{polymonad}
			\xymatrix{
				I && E \ar[ll]_-s \ar[rr]^-p && B \ar[rr]^-t && I
			}
\end{equation} 
we will call the set $I$ {\it the set of colours of $T$}, the set $B$ {\it the set of operations of $T$}, the set $E$ {\it the set of marked operations of $T$}, the map $t$ {\it the target map} and the map $s$ {\it the source map}. The map $p$ will be called {\it the middle map of $T$}. 

Explicitly, the structure of polynomial monad is given by a family of {  elements (units)} $1_i\in B$  for all $i\in I$ such that $t(1_i) = i$, $s(p^{-1}(1_i)) = \{i\},$ and a composite 
$\mu_T(b;b_1,\ldots,b_k)$ for each $b\in B,$ and each list of elements $b_1,\ldots.b_k\in B$ together with a bijection  $\eta:\{1,\ldots,k\} \to s(p^{-1}(b))$  such that $t(b_m) = \eta(m).$ 
These data should satisfy unitarity, associativity and equivariancy conditions. 
Polynomial monads and their cartesian maps form a category $\mathrm{PMon}.$

\begin{example} One can consider a small category $C$ with the set of objects $I$ and set of morphisms $B$ as a polynomial monad 
\[
			\xymatrix{
				I && B \ar[ll]_-s \ar[rr]^-{id} && B \ar[rr]^-t && I
			}
			\]
where $s$ and $t$ are the usual source and target maps. This gives us a  full embedding of categories $\mathrm{Cat}\to \mathrm{PMon}.$ This embedding has a right adjoint  which for a polynomial monad $T$ returns its submonad of unary operations.    						
\end{example}
\begin{example}[\cite{BB}] The { free} monoid monad is a polynomial monad represented by the diagram
\[
			\xymatrix{
				1 && Ltr^{*} \ar[ll]_-s \ar[rr]^-{p} && Ltr \ar[rr]^-t && 1
			}
			\] 
where $Ltr$ is the set of isomorphism classes of linear trees (or equally the set of all ordinals $\{0< 1<\ldots<n\}$), the set $Ltr^*$ is the set of linear trees with one vertex marked (equivalently the set  of all ordinals $\{0< 1<\ldots<k^*< \ldots<n\}$), the set of colours is the one object set. The middle map forgets the marking. The multiplication in the monad is generated by insertion of a linear tree to the marked vertex of another tree.

\end{example}

\begin{example} \label{example3}  Recall that a \emph{non-symmetric operad} $\mathcal{O}$ in a symmetric monoidal category $(\Ee,\otimes,e)$ is given by
		\begin{itemize}
			\item an object $\mathcal{O}_n$ in $\Ee$ for all integers $n \geq 0$
			\item a morphism $\epsilon: e \to \mathcal{O}_1$ called unit
			\item morphisms
			\[
			m: \mathcal{O}_k \otimes \mathcal{O}_{n_1} \otimes \ldots \otimes \mathcal{O}_{n_k} \to \mathcal{O}_{n_1 + \ldots + n_k}
			\]
			called multiplication
		\end{itemize}
		such that the usual associativity and unitarity conditions are satisfied.  	
		
The polynomial monad $NOp$ for non-symmetric operads was described in {\cite{BB,KJBM}}. The corresponding polynomial is :
 \[
			\xymatrix{
				\mathbb{N} && Ptr^{*} \ar[ll]_-s \ar[rr]^-{p} && Ptr \ar[rr]^-t && \mathbb{N}
			}
			\] 
Here, $Ptr, Ptr^*$ are the sets of isomorphism classes of planar trees, and planar trees with a marked vertex respectively. The middle map forgets the marked point, the source map is given by the number of incoming edges for the marked point and the target map associates to a tree its number of leaves.   			
The multiplication in this monad is generated by insertion of a tree inside a marked point. 
\end{example}

Let $\Ee$ be a cocomplete symmetric monoidal category and $T$ be a polynomial functor. One can construct a functor $\underline{T}^{\mathcal{E}}:\Ee^I \to \Ee^I$ given by a formula similar to (\ref{PPP}):
$$ \underline{T}^{\Ee}(X)_i = \coprod_{b\in t^{-1}(i)} \bigotimes_{e\in p^{-1}(b)} X_{s(e)}.$$
If $I = J$ and $T$ was given the structure of a polynomial monad then $\underline{T}^{\Ee}$ would acquire a structure of monad on $\Ee^I.$ This last category will be called the category of $I$-collections in $\Ee.$ The category of $\Set$-collections often will be called simply the category of $I$-collections and the category of $\Cat$-collection will be called the category of categorical $I$-collections. 
\begin{defin} The category of algebras of a polynomial monad $T$ 
in a cocomplete symmetric monoidal category $\Ee$ is the category of algebras of the monad $\underline{T}^{\Ee}.$
\end{defin} 

Explicitly, an $\Ee$-algebra $A$ of a polynomial monad $T$ is given by a collection $A_i\in\Ee,i\in I,$ equipped with the following structure maps:
$$m_{(b,\sigma)}: A_{s(\sigma(1))}\otimes\ldots\otimes A_{s(\sigma(k))} \to A_{t(b)}$$
for each $b\in B,$ and each bijection $\sigma:\{1,\ldots,k\} \to p^{-1}(b)$ which satisfy some appropriate associativity, unitarity and the following equivariancy condition \cite{BB}. If $\sigma':\{1,\ldots,k\} \to p^{-1}(b)$
is a bijection then the following triangle commutes:

{ \[
	\xymatrix{
		A_{s(\sigma(1))}\otimes\ldots\otimes A_{s(\sigma(k))} \ar[rr]^\tau \ar[rd]_{m_{(b,\sigma)}} && A_{s(\sigma'(1))}\otimes\ldots\otimes A_{s(\sigma'(k))} \ar[ld]^{m_{(b,\sigma')}} \\
		& A_{t(b)}
	}
\]}
where $\tau$ is the action of the permutation $(\sigma')^{-1}\circ \sigma.$

\begin{remark} The presence of the linear ordering of $p^{-1}(b)$ in the formula above is necessary to fix an order of tensor products. The equivariancy condition assures that the structure maps do not depend on the ordering \cite{BB}. This is closely related to the fact that the category of polynomial monads is equivalent to the category of $\Sigma$-free symmetric operads.
The category of algebras of a polynomial monad is isomorphic to the category of algebras of the corresponding $\Sigma$-free operad \cite{BB}.   
\end{remark}


\section{Internal algebras and  Grothendieck construction} 
 
 Algebras of a polynomial monad $T$ in the symmetric monoidal category of small categories $(\Cat,\times , 1)$ will play a special role. We will call them {\it categorical algebras} of $T.$  The category of categorical algebras of $T$ is isomorphic to the category of internal categories in the category of $T$-algebras of $T$ (in $\Set$).   
 The category of categorical $T$-algebras is naturally a $2$-category. We will use this fact but preserve the notation $\Alg_T(\Cat)$ for this $2$-category. 
 
 A terminal internal category has a unique $T$-algebra structure for any polynomial monad $T$; the latter promotes it to a terminal categorical  $T$-algebra. From now on {\it all terminal objects will be denoted $1$} hoping that this will cause no confusion.
 
 The following definitions are taken from \cite{EHBat} and \cite{BB}.
 \begin{defin}\label{intalg1}Let $A$ be a categorical $T$-algebra for a polynomial monad $T$.

An \emph{internal $T$-algebra in $A$} is a \emph{lax morphism} of categorical $T$-algebras from the \emph{terminal} categorical $T$-algebra to $A$.

Internal $T$-algebras in $A$ and $T$-natural transformations form a category $\Int_T(A)$ and this construction extends to a $2$-functor $\Int_T:\Alg_T(\Cat)\to\Cat.$\end{defin}
 
An internal $T$-algebra in a categorical $T$-algebra $A$ can be explicitly given by a collection of objects $a_i\in A_i$ together with a morphism$$\mu_{(b,\sg)}:m_{(b,\sg)}(a_{s(\sigma(1))},\ldots, a_{s(\sigma(k))})\rightarrow a_{t(b)},$$ for each operation $(b,\sigma),$ which satisfies obvious associativity, unitarity and equivariancy conditions. Here, $m_{(b,\sigma)}$ is the structure functor of $A.$

Given a cartesian map of polynomial monads $f: S \to T$ we have a restriction $2$-functor $f^*: \Alg_T(\Cat)\to \Alg_S(\Cat).$

 \begin{defin}\label{intalg2}Let $A$ be a categorical $T$-algebra for a polynomial monad $T$.

An \emph{internal $S$-algebra in $A$} is a \emph{lax morphism} of categorical $S$-algebras from the \emph{terminal} categorical $S$-algebra to $f^*(A)$.

Internal $S$-algebras in $A$ and $S$-natural transformations form a category $\Int_S(A)$ and this construction extends to a $2$-functor 
\begin{equation}\label{Int} \Int_S:\Alg_T(\Cat)\to\Cat.\end{equation} \end{defin}

 We will also need the following generalisation of the classical Grothendieck construction for categories to polynomial monads. 
 Let $T$ be a polynomial monad and let $A$ be a categorical algebra for it. 
 We construct a new polynomial monad 
$\int A$ as follows. 
The set of colours of $\int A$ is the set of pairs $(i,a)$ where $a$ is an object of the category $A_i.$  
An operation consists of:  
\begin{enumerate} 
\item An element 
 $b\in B;$ 
 \item For each element $e\in p^{-1}(b)$ an object $a_e\in A_{s(e)};$
 \item An object $y\in A_{t(b)};$
 \item A morphism $f_{(b,\sigma)}: m_{(b,\sigma)}(a_{{\sigma(1)}},\ldots,a_{\sigma(k)})\to  y$ in $A_{t(b)}$ for each $\sigma:\{1,\ldots,k\}\to p^{-1}(b),$ which satisfies the following equivariancy condition.
 If $\sigma':\{1,\ldots,k\}\to p^{-1}(b),$ and $\tau(a_{{\sigma(1)}},\ldots,a_{\sigma(k)}) = (a_{{\sigma'(1)}},\ldots,a_{\sigma'(k)})$ then 
 $$f_{(b,\sigma)} = f_{(b,\sigma')}.$$  
 \end{enumerate} 
Obviously, to specify an operation it suffices to know the morphism $f_{(b,\sigma)}.$ 
A   marked operation in $\int A$ is an  operation in which one of the elements $e\in p^{-1}(b)$ is marked. As usual, the middle map forgets about marking. 
The target map of the monad $\int A$ is the pair $(t(b),y)$ and the source map is $(s(e),a_e)$ where $e$ is the  marked element. 
 
To describe composition suppose we are given a list of  operations $ f_{(b_1,\sigma_1)},\ldots, f_{(b_k,\sigma_k)}$  with targets $(y_1,\ldots,y_k)\in A_{t(b_1)}\times\ldots\times A_{t(b_k)}$ in $\int A$ and an operation $g_{(c,\sigma)}$ with compatible sources.  Due to the equivariancy condition, we always can choose $\sigma$ in a way that these compatibility condition mean that $m_{(c.\sigma)}(y_{\sigma(1)},\ldots,y_{\sigma(k)})$ is the source of $g_{(c,\sigma)}.$ Hence, we define the composite operation in $\int A$ as the operation $h_{(d,\pi)}$ where 
 $d$ is an operation in $T$ obtained as a composite of $c$ and $b_1,\ldots,b_k,$ the underlying morphism is   the composite of two morphisms
 $$g_{(c,\sigma)}\circ m_{(c,\sigma)}( f_{(b_1,\sigma_1)},\ldots, f_{(b_k,\sigma_k)}),$$
  and 
  $$\pi = \sigma\circ (\sigma_1\times \ldots \times \sigma_k).$$
  
  The unit of the monad $\int A$ sends an operation in the identity monad $(i,a)$ to the operation $id_{(e_i,1)}$ where $e_i = \eta(i) \in B$  and $1$ is the unique function from $1\to p^{-1}(e_i).$  
  
 The polynomial monad $\int A$ comes equipped with a cartesian map of monads $\Gamma: \int A \to T.$  A section of $\Gamma$ is a map of polynomial monads $T\to \int A$ such that its composite with $\Gamma$ is the identity.  It is quite obvious that there is a bijection between sections and internal $T$-algebras in $A.$  
 
 We need to enhance this bijection to a functor. For this we first extend the category $\mathrm{PMon}$ to a $2$-category $\mathbf{PMon}.$ Let $f,f':S\to T$ be two cartesian maps between polynomial monads given by the diagram 

\[
			\xymatrix{
				J \ar@<-.5ex>[dd]_{\phi} \ar@<.5ex>[dd]^{\phi'} && D 
				 \ar[ll]_-{v} \ar[rr]^-{q} \ar@<-.5ex>[dd]_{\pi} \ar@<.5ex>[dd]^{\pi'} && C \ar[rr]^-{u} \ar@<-.5ex>[dd]_{\psi} \ar@<.5ex>[dd]^{\psi'} && J \ar@<-.5ex>[dd]_{\phi} \ar@<.5ex>[dd]^{\phi'} \\
				\\
				I&& E \ar[ll]_-s \ar[rr]^-p && B \ar[rr]^-t && I
			}
			\]
			
A natural transformation $\xi:f\to f'$ consists of map $\sigma:J\to B,$ such that for any $j\in J$ the set $p^{-1}(\sigma(j))$ has only one element,
$t(\sigma(j)) = \phi'(j),$ $s (p^{-1}(\sigma(j)))= \phi(j)$, and for any $c\in C$ there is an equality   
$$\mu_T(\sigma(u(c));\psi(c)) = \mu_T(\psi'(c);\sigma(v(c_1)),\ldots,\sigma(v(c_k))),$$ where $\{c_1,\ldots,c_k\}=q^{-1}(c).$
If $f$ and $f'$ are functors between small categories this definition amounts to the definition of a natural transformation between $f$ and $f'.$

 \begin{proposition}\label{sectionP} The category of internal $T$-algebras in $A$ is isomorphic to the category of sections of $\Gamma$ and natural transformations between them such that their composite with $\Gamma$ is the identity natural transformation of the identity map of $T.$  
 \end{proposition}
 
 \begin{proof} By direct calculations. 
  \end{proof} 
  
  \begin{remark} If $T$ is a polynomial monad with an identity middle map (that is a small category) our polynomial Grothendick construction coincides with the classical Grothendieck construction of a functor $A:T\to \Cat.$ The category of internal algebras $\Int_T(A)$ is, therefore, the lax-limit of this functor and $\int A$ is its lax-colimit.
  
  \end{remark}
  
Let $\mathbf{PMon}/{R}$ be the $2$-category of polynomial monads over $R.$ The objects of this $2$-category are cartesian polynomial monad morphisms:
{$$g:T\to R,$$}
the morphisms are commutative triangles:

{
\[
	\xymatrix{
		S \ar[rr]^f \ar[rd]_h && T \ar[ld]^g \\
		&  R
	}
\]}
and $2$-morphisms are natural transformations $f\to f'$ such that the whiskering with $g$ is an identity transformation of $h.$ 
With this notation, the category of sections above is the hom-category $(\mathbf{PMon}/{{ T}})(T,\int A).$

  Now let $f:S\to T$ be a cartesian map of polynomial monads given by a commutative diagram:
\begin{equation}\label{injective}
			\xymatrix{
				J\ar[dd]_{\phi} && D 
				 \ar[ll]_-{v} \ar[rr]^-{q} \ar[dd]_{\pi} && C \ar[rr]^-{u} \ar[dd]_{\psi} && J \ar[dd]_{\phi} \\
				\\
				I&& E \ar[ll]_-s \ar[rr]^-p && B \ar[rr]^-t && I
			}
			\end{equation}
If $A$ is a categorical algebra of $T$, the algebra $f^*(A)$ has the following explicit description. The underlying collection of $f^*(A)$ is given by the collection $f^*(A)_j = A_{\phi_j}.$ The structure functor
$$m_{(c,\sigma)}: f^*(A)_{v(\sigma(1))}\times\ldots\times f^*(A)_{v(\sigma(k))} \to f^*(A)_{u(c)}$$
is given by the functor $m_{(\phi(c),\sigma')}$, where $\sigma'$ is the composite
$$\{1,\ldots,k\} \stackrel{\sigma}{\to} q^{-1}(c)\stackrel{\pi'}{\to}\pi^{-1}(\psi(c)).$$    
 In the last display $\pi'$ is the bijection induced by $\pi$ on fibers due to  the fact that the middle square is a pullback. 
  
 \begin{proposition} Let $A$ be a categorical algebra of a polynomial monad $T.$ Also let  
$f:S\to T$ be a map of polynomial monads. Then there is a cartesian map of polynomial monads $\int f: \int f^*(A) \to \int A$ making the following diagram commutative 
\begin{equation}\label{Grothendieck}
			\xymatrix{
				\int f^*(A) \ar[rr]^{} \ar[dd]_{} && \int A \ar[dd]^{\Gamma} \\
				\\
				S \ar[rr]^{f} && T
			}
			\end{equation}
Moreover, this diagram is a pullback of polynomial monads.			
 \end{proposition}  
 \begin{proof}   The colours of $\int f^*(A)$ are pairs $(j,a)$ where $a\in f^*(A)_j = A_{\phi(j)}.$
 The operations of $\int f^*(A)$ are morphisms
 $f_{\sigma}: m_{(c,\sigma)}(a_{{\sigma(1)}},\ldots,a_{\sigma(k)})\to  y$ in $f^*(A)_{u(c)} = A_{t(\psi(c))},$ 
 where $a_{\sigma(i)} \in f^*(A)_{v(\sigma(i))} = A_{\phi(v(\sigma(i)))} = A_{s(\psi(\sigma(i)))} = A_{s(\sigma'(i))}.$
We define  $\int \phi (j,a) = (\phi(j),a)$ on colours. We observe that an operation $f_{\sigma}$ as above can be interpreted as an operation 
$f_{\sigma'}: m_{\psi(c),\sigma'}(a_{\sigma'(1)},\ldots, a_{\sigma'(k)})\to y$
and, hence we define $\int \phi (f_{\sigma}) = f_{\sigma'}.$ The definition of $\int f$ on marked operations is obvious. 
 
 It is now a simple exercise to check that the square of polynomial monads is a pullback.

 \end{proof}
 
We obtain the following generalisation of  Proposition \ref{sectionP}:
 \begin{corol}\label{sections} 
 The category of internal algebras $\Int_S(A)$ of $S$ in $A$ is isomorphic to the category $(\mathbf{PMon}/{{ T}})(S,\int A)$ of cartesian maps between polynomial monads $S\to \int A$ over $T$ and their natural transformations such that their composite with $\Gamma$ is the identity transformation of $f.$ 
 \end{corol}
 
 \section{Classifiers for maps between polynomial monads}

For any cartesian morphism of cartesian monads $f:S\to T$ one can associate a  categorical $T$-algebra $T^S$ with certain  universal property \cite{EHBat,BB,W}. Namely,
this is the object representing the $2$-functor
(\ref{Int}).  This categorical $T$-algebra  is called {\it the classifier of internal $S$-algebras inside categorical $T$-algebras} and is denoted $T^S$. 

In particular, if $f= Id$ the $T$-algebra $T^T$ is called {\it an absolute classifier of $T$}. It was proved in \cite{EHBat,BB} that an absolute classifier of $T$ can be computed as a truncated simplicial $T$-algebra:
 
\[
			\xymatrix{
				\F_T1 \ar[rr]|-{T\eta_1} && \F_T(T 1) \ar@<-1ex>[ll]_-{\mu_1} \ar@<1ex>[ll]^-{T\tau} && \F_T(T^2 1) \ar[ll]|-{\mu_{T1}} \ar@<-1ex>[ll]_-{T \mu_1} \ar@<1ex>[ll]^-{T^2 \tau}
			}
\]
\noindent where $1$ is the terminal $I$-collection, $\tau:T(1)\to 1$ is the unique map, $\F_T$ is the free $T$-algebra functor. This simplicial object satisfies Segal's condition because $T$ is a cartesian monad and, hence, represents an internal category in the category $\Alg_T(\Set).$ The last category is equivalent to $\Alg_T(\Cat)$ again due to cartesianness of $T.$ 

It is important to understand that $T^T$ is  a family of categories indexed by $i\in I.$ It has a universal internal $T$-algebra $1\to T^T$ which generates it. A component of this internal algebra $1_i\in (T^T)_i$ is a terminal object in this category.

\begin{example} If $T$ is a small category, the categorical $T$-algebra $T^T$ is the presheaf of categories on $T$ given by comma-categories $ T^T(i) = T/i.$ The universal internal algebra is given by objects $i\stackrel{id}{\to} i$ for each $i.$  

\end{example}

     \begin{example}\cite{BB} For the { free} monoid monad $Mon$ the absolute classifier $Mon^{Mon}$ is the monoidal category of all finite ordinals $\Delta_{alg}.$ 
The universal internal algebra is given by the terminal ordinal $[0].$ 
\end{example}
     
     \begin{example}\cite{BB} For the free nonsymmetric operad monad $NOp$, the absolute classifier $NOp^{NOp}$ is the non-symmetric categorical operad of planar trees. The morphisms are generated by contractions of internal edges and introducing a single vertex on an edge. The canonical internal operad is given by the sequence of corollas. 
    
     \end{example}

There are analogous formulas in the non absolute case. Namely,
 given  a cartesian map between polynomial monads  $f:S\to T$ as in (\ref{injective}), we have the following commutative square of adjunctions:  
{
	\[
		\xymatrix{
			\Alg_S \ar@<-.5ex>[rr]_-{f_!} \ar@<-.5ex>[dd]_-{\scriptstyle \mathcal{U}_S} && \Alg_T \ar@<-.5ex>[ll]_-{f^*} \ar@<-.5ex>[dd]_-{\scriptstyle \mathcal{U}_T} \\
			\\
			\Set^J \ar@<-.5ex>[rr]_-{\phi_!} \ar@<-.5ex>[uu]_-{\scriptstyle \mathcal{F}_S} && \Set^I \ar@<-.5ex>[ll]_-{\phi^*} \ar@<-.5ex>[uu]_-{\scriptstyle \mathcal{F}_T}
		}
	\]
}

Here $\phi^*$ is the  restriction functor $\Set^I\to \Set^J$ induced by  $\phi: J\to I$ and $\phi_!:\Set^J\to \Set^I$ is its left adjoint given by coproducts over fibers of $\phi.$  

The $T$-categorical algebra is given then by an internal categorical object similar to the absolute case:

\[
			\xymatrix{
				\F_T(\phi_!(1)) \ar[rr]|-{T\eta_1} && \F_T(\phi_!(S 1)) \ar@<-1ex>[ll]_-{\mu_1} \ar@<1ex>[ll]^-{T!} && \F_T(\phi_!(S^2 1)) \ar[ll]|-{\mu_{T1}} \ar@<-1ex>[ll]_-{T \mu_1} \ar@<1ex>[ll]^-{T^2 !}						}
			\]  
where $1$ is now the terminal $J$-collection.

The classifier construction provides a $2$-functor
$$T^{(-)}:\mathbf{PMon}/T \to Alg_T(\Cat).$$

\begin{proposition} Let $T$ be a polynomial monad. The classifier $2$-functor $$T^{(-)}:\mathbf{PMon}/T \to Alg_T(\Cat).$$
is the ($\Cat$-enriched) left adjoint to the Grothendieck construction $2$-functor $$\int (-): Alg_T(\Cat) \to \mathbf{PMon}/T.$$
\end{proposition} 

\begin{proof} Let $A$ be a categorical algebra of $T$ and let $f:S\to T$ be  a polynomial monad over $T.$ Then 
$$Alg_T(\Cat)(T^S, A) \cong \Int_S(A) \cong \mathbf{PMon}/T(S,\int A)$$
by Corollary \ref{sections}.

\end{proof} 
\begin{corol} The classifier functor commutes with colimits. 

In particular, given a pushout diagram of polynomial monads 
\begin{equation}\label{ABCD}
			\xymatrix{
				A\ar[rr]^{f} \ar[dd]_{g} && B\ar[dd]^{F} \\
				\\
				C \ar[rr]^{G} && D \ ,
			}
			\end{equation}

we obtain a pushout of categorical $D$-algebras :
\begin{equation}\label{pushoutdd}
			\xymatrix{
				D^A\ar[rr]^{D^f} \ar[dd]_{D^g} && D^B\ar[dd]^{D^F} \\
				\\
				D^C \ar[rr]^{D^G} && D^D
			}
\end{equation}
\end{corol} 

The following functorial properties of classifiers will be very useful for us:

\begin{proposition}\label{f_!(S^S)} Let $f:S\to T$ be a map of polynomial monads. Let $f^*$ be the restriction functor on categorical algebras and let $f_!$ be its left adjoint. Then
$$f_!(S^S) \cong T^S.$$
\end{proposition}
\begin{proof} The proof is a simple exercise in universal properties of adjoints and classifiers.
\end{proof} 
This implies another functorial property of classifiers. 
\begin{proposition}\label{square}    
Any commutative square of maps of polynomial monads 
\begin{equation}\label{abcd}
			\xymatrix{
				A\ar[rr]^{f} \ar[dd]_{g} && B\ar[dd]^{F} \\
				\\
				C \ar[rr]^{G} && D
			}
			\end{equation}
induces a map of classifiers
$$G^f:C^A \to G^*( D^B)$$
functorial with respect to horizontal pasting of squares. 
\end{proposition}
\begin{proof} Observe that we have a natural map of classifiers $C^A\to G^* (D^A).$ Indeed, by adjunction such a map corresponds to a map $G_!(C^A) \to D^A.$ 
But $G_!(C^A) \cong G_!(g_!(A^A) )\cong D^A$ as proved in Proposition \ref{f_!(S^S)}. 

We can now take the composite
$$C^A\to G^*( D^A) \to G^* (D^B),$$
where the last map is induced by the  upper commutative triangle of the square. 

\end{proof}

More generally given a commutative square (\ref{abcd}) let $X$ and $Y$  be categorical algebras of $C$ and $D$ correspondingly and let  $\xi:X\to G^*(Y)$ be its morphism. And let $x:1\to g^*(X)$ be an internal $A$-algebra in $C$ and $y:1\to F^*(Y)$ be an internal $B$-algebra in $Y.$  
\begin{defin}\label{inmor} A morphism of internal algebras $x\to y$ is a lax transformation $\phi:x\to f^*(y)$ of the internal algebras as  displayed on the following square of categorical $A$-algebras and their morphisms:
\begin{equation}\label{a'b'c'd'}
			\xymatrix{
				1\ar[rr]^{x} \ar[dd]_{id}  \xtwocell[rrdd]{}<>{\phi} && g^*(X)\ar[dd]^{g^*(\xi)} \\
				\\
				f^*(1) \ar[rr]^{f^*(y)} && g^*(G^*(Y))= f^*F^*(Y)
			}
			\end{equation}
\end{defin}


The following Proposition establishes a universal property of the morphism $G^f:$

\begin{proposition}\label{Gf} The morphism $G^f$ from Proposition \ref{square} induces a  morphism of canonical internal algebras $a$ in $C^A$ and $b$ in $D^B$ in the sense of the Definition \ref{inmor} such that $\phi:a\to f^*(b)$ is an identity  and is determined by this property.  
\end{proposition}
\begin{proof} The proof is by checking universal property. \end{proof}

\section{Homotopy theory of algebras and classifiers}

Let $(\Ee,\otimes, e)$ be a monoidal model category and let $T$ be a polynomial monad.  The category of $I$-collections  $\Ee^I$   has a projective model structure. For a polynomial monad $T$  we can try to transfer this model structure on collections     to the category of $T$-algebras along  the forgetful functor $\mathcal{U}_T:\Alg_T(\Ee)\to \Ee^I .$ This process requires some conditions on $\Ee$ and $T$ \cite{BB}. If we  need only a semi-model model structure on $\Alg_T(\Ee)$   it suffices for $\Ee$ to be cofibrantly generated \cite{WY}. In many cases of interest (for simplicial sets, topological spaces or chain complexes in characteristic $0,$ for example)  we do have a full model structure.   

Classifiers enter the scene because of the following theorem proved in \cite{BB}[Theorem 8.2]. 

\begin{theorem}\label{hocolim}Let $\Ee$ be a monoidal model category with a ``good'' realisation functor  for simplicial objects, and let $f:S\rightarrow T$ be a cartesian monad morphism between polynomial monads. Let $X$ be an $S$-algebra in $\Ee$ whose underlying $J$-collection is pointwise cofibrant. Then the $I$-collection underlying the left derived Quillen functor $\,\LL f_{!}(X)$ can be calculated as  the homotopy colimit over $T^S$ of the functor $\tilde{X}:T^S\to\Ee$ representing the $S$-algebra $X$.\end{theorem}
The Theorem \ref{hocolim} has an important corollary which allows an interpretation of the nerve of a relative classifier of a map $f:S\to T$ as the value of the left derived  functor of $f_!$ on the terminal  $S$-algebra. 
\begin{corol}[\cite{BB}]\label{Nerve}The  nerve $N(T^S)$  is a cofibrant simplicial $T$-algebra. In fact,
$$N(T^S) \cong N (f_!(T^T)) \cong f_! (N(T^T)) = \LL f_{!}(1)$$ where  $1$ is the terminal simplicial $S$-algebra.\end{corol}

A ``good'' realisation functor always exists if the category $\Ee$ is a simplicial model category. 
For a simplicial model category $\mathcal{M}$, we will denote by $\underline{\mathcal{M}}(X,Y)$ the simplicial hom functor between $X$ and $Y$. 
If $\Ee$ is a simplicial category then $\Alg_T(\Ee)$ is also a simplicial category for any polynomial monad $T.$  
If the transferred model category structure on $\Alg_T(\Ee)$ exists then it is also a simplicial model structure and an adjunction between categories of algebras generated by a map of polynomial monads is also a simplicial adjunction. In particular, all this is true for simplicial algebras of polynomial monads.

Recall also, that in a simplicial model category $\mathcal M$ the mapping space $Map_{\mathcal{M}}(X, Y)$ can be computed as the simplicial hom $\underline{\mathcal{M}}(cof(X),fib(Y))$ where $cof$ and $fib$ are cofibrant and fibrant replacement respectively \cite{Hirschhorn}.

\begin{theorem}[Quillen Theorem A for Polynomial monads]\label{QTAPoly}

 For any commutative diagram of maps of polynomial monads

{
	\[
		\xymatrix{
			S \ar[rr]^f \ar[rd]_h && T \ar[ld]^g \\
			& R \ar[d]_\phi \\
			& P
		}
	\]
}if $N(R^S)\to N(R^T)$ is a weak equivalence then
$N(P^S)\to N(P^T)$ is a weak equivalence. 

\end{theorem}

\begin{proof} We have
$$N(P^T) \cong (g\circ \phi)_!(N(T^T)) \cong \phi_!(N(g_!(T^T)) \cong  \phi_!(N(R^T)).$$ 
If $X$ is a fibrant simplicial $P$-algebra then the induced morphism of simplicial homs
$$\underline{\Alg_P}(N(P^S),X)\leftarrow \underline{\Alg_P}(N(P^T),X)$$
is isomorphic to 
$$\underline{\Alg_P}(\phi_!(N(R^S)),X)\leftarrow \underline{\Alg_P}(\phi_!(N(R^T)),X)$$
and by adjunction to
$$\underline{\Alg_R}(N(R^S)),\phi^*X)\leftarrow \underline{\Alg_R}(N(R^T),\phi^*X).$$ 
Since $\phi^*(X)$ is fibrant and $N(R^S)$ and $N(R^T)$ are cofibrant $R$-algebra, the last map is a weak equivalence.  So, $N(P^S)\to N(P^T)$ is a weak equivalence as well.
\end{proof} 

\begin{corol}[Classical Quillen Theorem A]  If in a commutative triangle of small categories

{
	\[
	\xymatrix{
		S \ar[rr]^f \ar[rd]_h && T \ar[ld]^g \\
		& R
	}
	\]
}$f$ induces a weak equivalence $N(h/r)\to N(g/r)$ for any object $r\in R$ then $N(f):N(S)\to N(T)$ is a weak equivalence.

\end{corol}

\begin{proof} Consider this commutative triangle as commutative triangle of morphisms between polynomial monads. The maps of comma categories $h/r\to g/r$ are the components of map of classifiers $R^S\to R^T.$

 Take $P=1,$  the terminal category,  and $\phi:R\to P$ the unique functor.    Then we can apply   Theorem \ref{QTAPoly}. $N(f)$ is exactly the nerve of the map between classifiers $1^S \to 1^T.$

\end{proof}

\begin{corol}\label{QuillenA} Let $f:S\to T$ be a map of polynomial monads. The following statements are equivalent:
\begin{enumerate} 

\item {$N(T^S)$} is contractible;
\item For any commutative triangle of maps of polynomial monads

	\begin{equation}\label{triangle}
	\xymatrix{
		S \ar[rr]^f \ar[rd]_h && T \ar[ld]^g \\
		& R}
	\end{equation}

the morphism $N(R^f):N(R^S)\to N(R^T)$ is a  weak equivalence.

\end{enumerate}

\end{corol}

\begin{proof} 

\noindent (1) $\to$ (2).  

The triangle above can be rewritten as a commutative diagram

{
	\[
	\xymatrix{
		S \ar[rr]^f \ar[rd]_f && T \ar[ld]^{id} \\
		& T \ar[d]^g \\
		& R
	}
	\]
}We have $N(T^f): N(T^S)\to N(T^T)$ is a weak equivalence since $T^T$ has a terminal object. 
By Theorem \ref{QTAPoly},  $N(R^S)\to N(R^T)$ is a weak equivalence.

\noindent (2) $\to$ (1).  Take $R=T$ and $g = id:T\to T.$ Then $N(T^f): N(T^S)\to N(T^T)$ is a weak equivalence implies that $N(T^S)$ is contractible.
\end{proof}

This result justifies the following definition.

\begin{defin} A cartesian map $f: S\to T$  between polynomial monads is called homotopically cofinal if $N(T^{S})$ is contractible.

\end{defin} 

A well known classical characterisation of  cofinal functors asserts that these are exactly the functors restriction along which  preserves  limits. We are going to provide a similar characterisation of homotopy cofinal maps between polynomial monads.

\begin{theorem} 
For a commutative triangle of polynomial monads (\ref{triangle}) the following statements are equivalent:

 \begin{enumerate}
 \item The map $N(R^f):N(R^S)\to N(R^T)$ is a weak equivalence.
 \item For  any simplicial $R$-algebra $X$
the morphism $f$ induces a weak equivalence of homotopy mapping spaces
$$Map_{Alg_S}(1,h^*X) \to Map_{Alg_T}(1,g^*X).$$ 
Here $1$ means the terminal simplicial algebra.

\end{enumerate}

\end{theorem}

\begin{proof}

 From the beginning we can assume that $X$ is a fibrant $R$-algebra. We can compute the mapping space 
$Map_{\Alg_S}(1,h^*X)$ as the simplicial hom 
$\underline{Alg_S}(cof(1),h^*X)$ where $cof(1)$ is a cofibrant replacement for the terminal algebra $1.$ By adjunction  this space is isomorphic to $ \underline{Alg_R}(\LL h_! (1), X).$  
By Corollary \ref{Nerve}  $\LL h_! (1) \cong N(R^S).$ Then $$\underline{\Alg_R}(\LL h_! (1), X)
\sim \underline{\Alg_R}(N(R^S), X)$$ 
Similarly 
$$\underline{\Alg_R}(g_!(cof(1)), X) \sim \underline{\Alg_R}({N(R^T)},X).$$ 
Hence,  $N(R^f)$ induces a weak equivalence between these simplicial sets.

It is obvious that we can reverse these calculations and so prove that $N(R^f)$ is a weak equivalence provided it induces a weak equivalence between mapping spaces for all $R$-algebra $X.$ 
\end{proof}

\begin{corol}\label{weakequivalencemappingspaces} A cartesian map between polynomial monads $f:S\to T$ is homotopically cofinal if and only if for any simplicial $T$-algebra $X$ it induces a weak equivalence between mapping spaces:
$$Map_{\Alg_S}(1,f^*X)\to Map_{\Alg_T}(1, X).$$

\end{corol}
Recall that {\it homotopy left cofinal functor} is defined in \cite{Hirschhorn}[Definition 19.6.1] as a functor 
between small categories $f:S\to T$ such that  $N(f/i)$ is contractible for all objects $i\in I.$    
It coincides with our notion of homotopically cofinal map between polynomial monads when the latter is specialised to small categories. 

\begin{corol}[Theorem 19.6.13(b) \cite{Hirschhorn}] A functor $f:S\to T$ between small categories is homotopically left cofinal if and only if
for any functor $X:T\to SSet$ it induces a weak equivalence
$$\holim_S f^*X \to \holim_T X .$$ 
\end{corol}
\begin{proof} The homotopy limit of a functor  $X$ can be computed as a mapping space from terminal functor to $X.$
\end{proof}
\begin{remark} This theorem (and its dual) is proved in \cite{Hirschhorn} in a slightly more general setting. We are not going to pursue this generality in this paper but it is not hard to get Hirschhorn's Theorem from 
the corollary above and the model theoretic argument from \cite{BB}. 
\end{remark} 

\begin{remark} If $\mathcal{W}$ is a fundamental localizer of Grothendieck \cite{cis06} then all the results of this section can be localized with respect to $\mathcal{W}.$ In particular, if we take $\mathcal{W} = \mathcal{W}_0$ (so the weak equivalences between small categories become functors which induce isomorphisms on $\pi_0$) the notion of $\mathcal{W}_0$-homotopically cofinal functor coincides with the classical categorical notion of cofinal functor. 
\end{remark}
\begin{remark} In \cite{BatStab} the first author used the fact that the map of polynomial monads $Op_n \to SOp$ is a $\mathcal{W}_{n-2}$-homotopically cofinal functor to prove the Baez-Dolan stabilization hypothesis for Rezk's weak $n$-categories.  Here, $Op_n$ is the polynomial monads for $n$-operads, $SOp$ is the polynomial monad for symmetric operads and $\mathcal{W}_{n-2}$ is the fundamental localizer for $(n-2)$-truncated homotopy type.   
\end{remark}

\section{Twisted Boadrman-Vogt tensor product~  \\  and Thomason's theorem}\label{BVP}
 
 Since polynomial monads form a $2$-category we can ask about lax-colimits of a diagram in $\mathbf{PMon}.$ Such a lax-colimit can be given explicitly as another version of Grothendieck construction. 
 
 Let $A$ be a small category and $F:A\to \mathbf{PMon}$ be a strict $2$-functor. We then define a new polynomial monad $\oint F$ as follows. The set of its colours is the set of pairs $(a,i)$ where $a\in A$ and $i\in F(a)_0$ is a colour of the polynomial monad $F(a).$ An operation consists of the following data:
 \begin{enumerate}
 \item An operation $\beta\in F(b)$
 \item A family of  morphisms $\alpha_d: a_d\to b$ in $A,$ where $d$ runs over the set $p_b^{-1}(\beta)$ and where $p_b$ is the middle map of the polynomial monad $F(b)$ ; 
 \item A family of colours of  $j_d\in F(a_d), d\in p^{-1}(\beta)$ such that   
$F \alpha_d(j_d) = s_b(d), d\in  p_b^{-1}(\beta),$ where $s_b$ is the source map of the monad $F(b).$  
 \end{enumerate} 
A marked operation is the operation $\beta$ with a marked element $d\in p^{-1}(\beta)$ and the middle map forgets the marking. The source map assigns to a marked operation  the pair $(a_d,j_d),$ where $d$ is the marked element.  The target of an operation $\beta$ as above is the pair $(b,t(\beta)).$
The unit and composition operation of the polynomial monad $\oint{F}$ are obvious now.
 In order to distinguish this construction from Grothendieck construction of a categorical algrebra we will call the polynomial monad $\oint F$ {\it twisted Boardman-Vogt tensor product.}

A  relation of twisted Boardman-Vogt product with the classical  Grothendieck construction is given by the following:  
 \begin{proposition} Let $\Ee$ be a  symmetric monoidal category. The category of algebras of $\oint F$ is isomorphic to the category of sections of the Grothendieck construction of the functor
$$Alg_F(\Ee) : A^{op} \to CAT$$
which associates to an $a\in A$ the category $Alg_{F(a)}(\Ee)$ and to a morphism $f:a\to b$  the functor
$F(f)^*:Alg_{F(b)}(\Ee)\to Alg_{F(a)}(\Ee).$ 

That is an algebra $X$ of $\oint F$ consists of a family of $F(a)$-algebras $X(a), a\in A$ together with family of $F(a)$-morphisms $\phi(a):X(a)\to F(f)^* X(b)$  for each $f:a\to b$ 
which satisfies obvious functoriality conditions. 

\end{proposition}  

\begin{proof} By direct inspection.
\end{proof}

The following corollary justifies our terminology.

\begin{corol} Let $D$ be a polynomial monad and let $F:A\to \mathbf{PMon}$ be the constant functor $F(a) = D$ then
$$\oint F \cong D\otimes_{\scriptscriptstyle BV} A,$$
where $\otimes_{\scriptscriptstyle BV}$ is the Boardman-Vogt tensor product of symmetric operads.  
\end{corol}
\begin{proof} Indeed, algebras of $\oint F$ in this case are just presheaves of algebras of $D,$ which is a defining property of the Boardman-Vogt tensor product.
\end{proof}
\begin{remark}
The Boardman-Vogt product was defined for symmetric operads not for polynomial monads. It can not be restricted under equivalence between $\Sigma$-free symmetric operads and polynomial monads in general because the result of this product may not be $\Sigma$-free. Famous example is, of course, the isomorphism $Ass\otimes_{\scriptscriptstyle BV} Ass \cong Com,$ which is an incarnation of the Eckmann-Hilton argument. 

But, it is not hard to check that the $BV$-product of a $\Sigma$-free operad and a category (symmetric operad with unary operations only) is again $\Sigma$-free, so our sentence makes sense.     
\end{remark}

\begin{corol}\label{intoint} Let $X$ be a categorical algebra of $\oint F.$ Then an internal $\oint F$-algebra $x$ in $X$ consists of a family of  internal $F(a)$-algebras $x(a)$ in $X(a)$ together with a family of morphisms of internal algebras (in the sense of Definition \ref{inmor}) $\phi(a):x(a)\to f^*(x(b)),$  which satisfies the usual lax-compatibility conditions.   

\end{corol}
  
It is obvious that in the case of presheaves in $\mathbf{Cat}$  twisted $BV$-product and  Grothendieck constructions are equivalent.  Indeed, given $F:A\to Cat$ one can form $\oint F.$ But we also can consider $F$ as a categorical  $A$-algebra. Then $\int F$ is isomorphic to $\oint F.$    
 But in general, there can not be even a cartesian polynomial monad map $\oint F \to A$ because if such a map exists the pullback of the middle maps would force $\oint F$ to be a category.
Instead of this we have an obvious functoriality of the twisted $BV$-product
 in the sense that every natural transformation $\psi:F\to G$ of presheaves of polynomial monads over $A$ induces a map of polynomial monads $$\oint \psi:\oint F \to \oint G.$$  If $F$ is a presheaf  of small categories and $1$ is the constant terminal presheaf of small categories then the unique morphism $e:F\to 1$ induces $\oint e:\oint F \cong \int F \to \oint 1 = A,$ which is exactly the projection from classical Grothendieck construction. 
 
Given $\psi: F\to G$ we can now construct the classifier $(\oint G)^{\oint F}.$ 
 \begin{proposition}\label{GFclassifier}
 The classifier $(\oint G)^{\oint F}$ is the categorical $\oint G$-algebra freely generated by the internal $F(a)$ algebras $x(a)\in G(a)^{F(a)} , a \in A$ together with morphisms
 $\phi(a): x(a)\to f^*(x(b))$ for each $f:a\to b$ in $A$ which satisfy obvious lax-functorial property.    
 \end{proposition}
 \begin{proof} This follows from general description of universal properties of the classifiers and  Lemma \ref{intoint}. 
 \end{proof}
 
Given $\psi:F\to G$ we also can construct yet another canonical categorical  $\oint G$-algebra as follows. For each $a\in A$ we take the classifier $G(a)^{F(a)}$ of $\psi(a):F(a)\to G(a).$  By Proposition \ref{square} these objects form an algebra in $Cat$ of $\oint G.$ By slightly abusing notations we will call  it $G^F$ and call it {\it local classifier algebra} of $\psi.$  

\begin{theorem}\label{cofibrantreplacement} There is a canonical morphism 
$$\Psi: (\oint G)^{\oint F} \to G^F$$ of categorical $\oint G$-algebras which is a  weak equivalence of simplicial $\oint G$-algebras after application of nerve functor. 

\end{theorem}

\begin{proof} We can get the map $\Psi$ from universal property of classifiers as follows. It will be enough to observe that $G^F$ contains a canonical internal $\oint F$-algebra. But this follows from Proposition \ref{Gf} and Corollary \ref{intoint} where all morphisms $\phi(a)$ are identities. 

Also observe that according to Proposition \ref{GFclassifier} each $G(a)$-algebra $(\oint G)^{\oint F}(a)$ contains an internal $F(a)$-algebra which immediately provides us with a $G(a)$-algebra morphism $p(a):G^F(a)\to(\oint G)^{\oint F}(a).$  By universal property the functor $p(a)$ is a section of $\Psi(a)$ but not necessary a $(\oint G)^{\oint F}$-algebras morphism. 
But we claim $p(a)$ is the left adjoint  to $\Psi(a).$ Indeed the counit $p(a)(\psi(a)(f^*(b))\to f^*(x(b))$ on the generating internal $F(b)$-algebra $x(b)$  in in $G(b)$ is given by canonical morphism       
  $\phi(a): x(a)\to f^*(x(b))$ since $p(a)(\psi(a)(f^*(b)) = x(a)$ for any $f:a\to b$ (since the morphism $\Psi$  maps each $\phi(a)$ to the identity).     
\end{proof}

\begin{remark} This theorem allows to see  more relations between twisted $BV$-product  and Grothendieck construction.
There is a canonical factorization of the map of polynomial monads $\oint \psi : \oint F\to \oint G:$
{
	\[
	\xymatrix{
		\oint F \ar[rr]^{} \ar[rd]_{\oint \psi} && \int G^F \ar[ld]^{} \\
		& \oint G
	}
	\]
}  
This is just the mate under the adjuction between Grothendieck construction and classifier functor of the canonical map $\Psi$ from the classifier $(\oint G)^{\oint F}$ to the local classifier algebra $G^F.$ 

 In the case of a map of presheaves in $Cat$ the canonical morphism $\Psi$ can be also understood as the counit of the adjunction between Grothendieck construction and classifier functor. 
\end{remark}

 \begin{corol} The simplicial $\oint G$-algebra $N((\oint G)^{\oint F})$ is a cofibrant replacement of the simplicial $\oint G$-algebra $N(G^F).$
 \end{corol}    

Now let $\psi:F\to G$ be a morphism of presheaves of polynomial monads over  a polynomial monad $D.$  By universal property of lax-colimit we then have   a  commutative  triangle of maps of polynomial monads:
{
	\[
	\xymatrix{
		\oint F \ar[rr]^{t} \ar[rd]_{e} && \oint G \ar[ld]^{d} \\
		& D
	}
	\]
}  
where $t = \oint \psi.$ The map $d$ induces a left Quillen functor $d_!: Alg_{\oint G}(SSet)\to Alg_D(SSet).$

\begin{proposition} 
There exists an isomorphism in the homotopy category of  simplicial $D$-algebras:
$$\mathbb{L}d_!(N(G^F)) \cong N(D^{\oint F}).$$ 

\end{proposition}

\begin{proof} Compute
$$e_!(N((\oint F)^{\oint F}) = N(e_!((\oint F)^{\oint F}) = N(D^{\oint F}).$$
On the other hand this object is equal to
$$d_!t_!(N((\oint F)^{\oint F}) =  d_!(N t_!((\oint F)^{\oint F})) = d_!(N((\oint G)^{\oint F}).$$ 
Since $N((\oint G)^{\oint F}$ is a cofibrant replacement for the nerve of $G^F$ we have
that the last object is isomorphic in $Ho(Alg_D(SSet)$ to the result of application of the derived functor of $d_!$ to $N(G^F).$
   
\end{proof}
This Proposition allows us to prove the following generalisation  of a classical Thomason's theorem  about homotopy colimits of small categories  \cite{Thlax}. 
\begin{theorem}\label{thomason} Let $F$ be a presheaf of polynomial monads over a polynomial monad $D.$ Then  $N(D^{\oint F})$ is weakly equivalent to $ \hocolim_A N(D^{F(a)}).$ 
\end{theorem}
\begin{proof} We take $G$ to be the constant functor $G(a) = D.$  Then we have a natural transformation of presheaves $\psi:F\to G$ whose components are derived from the components of cocone of $F$ over $D.$ 
The left adjoint $d_!$ is then the colimit over $A$ in the category of $D$-algebras.

\end{proof}

\begin{corol}[Thomason] For a presheaf $F$ of small categories over $A$ the nerve of its Grothendieck construction is homotopy equivalent to the homotopy colimit of the  presheaf of nerves. 

\end{corol}

\begin{proof} It is sufficient to take $D = 1.$  The classifier $1^{\oint F}$ is just the Grothendieck construction $\int F$ and the local classifier algebra $1^F$ is just $F$ itself.
\end{proof}

\begin{theorem}\label{recognition}  Let $F$ be a presheaf of polynomial monads over a polynomial monad $D.$ 
The following conditions are equivalent \begin{enumerate}
\item $\oint F \to  D$ is a homotopically cofinal functor;
\item $N(D^{\oint F})$ is contractible;
\item For any map of polynomial monads $D\to R$ the natural map $$\hocolim_A N(R^{F(a)}) \to N(R^D)$$
is a weak equivalence.
\end{enumerate}
\end{theorem}
\begin{proof} The equivalence of (1) and (2) is the definition of homotopically cofinal functor.
Now (2) is equivalent to (3) by Theorem \ref{thomason} and Corollary \ref{QuillenA}.  
\end{proof}

\section{Homotopy pushouts of classifiers}

For our application purpose  we will be mostly interested  in homotopy pushouts of classifiers. In other words we are going to consider presheaves of polynomial monads over the category $\Lambda$   with three objects and two nontrivial arrows:
$$2\longleftarrow 0 \longrightarrow 1.$$
Main question for us is then given a commutative square of polynomial monads (\ref{abcd}) over a polynomial monad $R$ when the square 
\begin{equation}\label{hpushouted}
			\xymatrix{
				N(R^A)\ar[rr]^{N(R^f)} \ar[dd]_{N(R^g)} && N(R^B)\ar[dd]^{N(R^F)} \\
				\\
				N(R^C) \ar[rr]^{N(R^G)} && N(R^D)
			}
\end{equation}
is a homotopy pushout square in the category of $R$-algebras?

\begin{example} 
In the square above let $F:\Lambda \to  \mathbf{PMon}$ is such that 
$F(0)=A,F(1) = B,F(2)=C$ are monoids and $D$ be the pushout of this diagram.  Let $R=1.$ Then
the map  $\hocolim_{\Lambda} N(R^{F(a)}) \to N(R^D)$ is the map from homotopy pushout of classifying spaces of this diagram of monoids to the classifying space of the pushout. This situation was considered by Fiedorowicz in \cite[Theorem 4.1]{F}\begin{footnote}{First author is grateful to Andrey Lazarev for pointing out to this Fiedorowicz's theorem.}\end{footnote}. It was proved that a sufficient condition for this map to be a weak equivalence is that $\mathbb{Z}[B]$ and $\mathbb{Z}[C]$ are flat $\mathbb{Z}[A]$-modules, where $\mathbb{Z}[M]$ is the monoid ring of a monoid $M.$ 
\end{example}

An example away from categories is one of the main topics in \cite{BB}:

\begin{example}
	Let $T$ be a polynomial monad with the set of colors $C$ and let $F(1) = T$ and $F(0) = F(2) = I$, where $I$ is the identity polynomial monad with the same set of colours $C$. The colimit of such diagram is $T = D$ again. The monad $\oint F$ in this case is the monad $T^{f,g}$ from \cite{BB}. Let $R = T$ then the colimit over the classifier $T^{T^{f,g}}$ `computes' the pushouts of algebras of $T$ along a free map of $T$-algebras. The homotopy type of $N(T^{T^{f,g}})$ can be  very nontrival.
\end{example}

\begin{defin} We will call a commutative square of polynomial monads (\ref{abcd}) homotopically cofinal
 if the equivalent conditions of Theorem \ref{recognition} are satisfied.  
\end{defin}
Thus for a homotopically cofinal square any square like (\ref{hpushouted}) is a homotopy pushout of $R$-algebras.

\begin{proposition} \label{dddd} \begin{enumerate}

\item For any polynomial monad $D$ the constant square 
\begin{equation*}\label{DDDD}
			\xymatrix{
				D\ar[r]^{id} \ar[d]_{id} & D\ar[d]^{id} \\
				D \ar[r]^{id} & D 
			}
			\end{equation*}
is homotopically cofinal.

\item If in a commutative diagram of polynomial monads
{
\begin{equation*}\label{DDDDDD}
			\xymatrix{
				A\ar[r]^{} \ar[d]_{} &C \ar[d]_{} \ar[r]^{}& E\ar[d]^{} 
				\\
				B \ar[r]^{} &D \ar[r]^{}  & F 
			}
			\end{equation*}}
the left square is homotopically cofinal, then the outer square is homotopically cofinal if and only if the right square is homotopically cofinal.

\item In a commutative cube of polynomial monads
 \[
\xymatrix{
A' \ar[dd] \ar[rr] \ar[rd] && B' \ar[dd] |!{[dl];[dr]}\hole \ar[rd] \\
& A \ar[dd] \ar[rr] && B \ar[dd] \\
C' \ar[rr] |!{[ru];[rd]}\hole \ar[rd] && D' \ar[rd] \\
& C \ar[rr] && D \ 
} 
\] 
\noindent let the back square be homotopically cofinal and let $A' \to A$, $B' \to B$ and $C' \to C$ homotopically cofinal morphisms. Then the front square is a homotopically cofinal square if and only if $D' \to D$ is homotopy cofinal.

\end{enumerate}
\end{proposition}
\begin{proof} 
For the first statement we observe similarly that  for any $D\to R$ we get a constant square of nerves of classifiers which is a homotopy pushout in the category of simplicial $R$-algebras.

To prove the second statement it is enough to consider an arbitrary map of polynomial monads $F\to R$. Then we obtain a commutative diagram of classifiers over $R$. From standard properties of homotopy pushouts we get that after application of nerve the outer square is a homotopy pushout of simplicial $R$-algebras if and only if the right square is a homotopy pushout. So Theorem \ref{recognition} implies the result.

Finally, for the third statement let $F:\Lambda \to  \mathbf{PMon}$ be the functor which constitutes the left corner of the front square. Let $F':\Lambda \to  \mathbf{PMon}$ be the corresponding functor for the back square. Then we have a commutative square of polynomial monads maps:
\begin{equation*}\label{ooDD}
			\xymatrix{
				\oint F'\ar[r]^{} \ar[d]_{} & \oint F\ar[d]^{} \\
				D' \ar[r]^{} & D 
			}
			\end{equation*}

The nerve of the classifier of the map $\oint F' \to \oint F$ is the cofibrant replacement of the nerve of the local classifier $F^{F'}$ which is contractible because $F(i)\to F'(i)$ is homotopically cofinal for each $i = 0,1,2.$  So, by Corollary \ref{QuillenA} the map $$N(D^{\oint F'})\to N(D^{\oint F})$$ is a weak equivalence. On the other hand, since the back square is homotopically cofinal  $\oint F'\to D'$ is a homotopically  cofinal map by definition. Therefore, the map $$N(D^{\oint F'})\to N(D^{D'})$$ is also a weak equivalence. The composite of two homotopically cofinal maps is homotopically cofinal map again and so if one of the two maps $D' \to D$ or $\oint F \to D$ is homotopically cofinal,  $\oint F' \to D$ is homotopically cofinal and the nerve of its classifier is contractible which implies that the other map must also be homotopically cofinal.  

\end{proof}

\begin{remark}\label{desc}  It is instructive to give a description of the classifier $D^{\oint F}, $ for $F:\Lambda \to Poly$ as in the square (\ref{abcd}).  
According to general theory of classifiers from \cite[Section 6.3]{BB} an object of this categorical $D$-algebra are given by 
specifying of \begin{enumerate} 
\item an operation  $\beta$ of the polynomial monad $D;$ 
 \item a labelling  of the sources of $\beta$ by numbers $0,1,2.$  
 \end{enumerate}
Morphisms are generated by morphisms in $D^A,$ when we multiply operations in $D$ with $0$-labelled sources and, similarly, by morphisms in $D^B$ and $D^C$ when we multiply operations with $1$ and $2$ labelled sources respectively. We also have two other type of generators: one can replace a source labelled by $0$ by a source labelled by $1$ or $2.$ There are relations on this type of morphisms which comes from the requirement that  $f$ and $g$ in (\ref{abcd}) are polynomial monad maps.  
\end{remark}

\section{Mapping spaces between pointed algebras of polynomial monads}  

Let $T$ be a polynomial monad and let $T_*$  be the monad whose category of algebras is the category of pointed $T$-algebras; that is, the comma-category
$1/\Alg_T.$ 
There is a map of monads $$u:T\to T_*$$ such that the restriction functor $u^*:\Alg_{T_*}\to\Alg_{T}$
`forgets the point'. Analogously, let $T_{**}$ be the category of double pointed algebras; that is, the category $1\coprod 1/ \Alg_T.$  We have a pushout of monads:
\begin{equation}\label{double}
			\xymatrix{
				T\ar[rr]^{u} \ar[dd]_{u} && T_*\ar[dd]^{} \\
				\\
				T_* \ar[rr]^{} && T_{**}
			}
\end{equation}
We can consider this pushout as a pushout of monads over $T_*$ because the identity $id:T_*\to T_*$ induces a map of monads $$U:T_{**}\to T_*.$$ such that the restriction functor $\Alg_{T_*}\to\Alg_{T_{**}}$ `doubles the point'.

\begin{theorem}[Formal delooping Theorem]\label{pointedalgebras}  Assume $T_*$ is a polynomial monad and the square (\ref{double}) is homotopically cofinal. Then, for any pointed simplicial $T$-algebra $X$, there is a weak equivalence of simplicial sets:
$$\Omega Map_{\Alg_T}(1,u^*X) \sim Map_{\Alg_{T_{**}}}(1,U^*X)$$
where $\Omega  Map_{\Alg_T}(1,u^*X)$ is the loop space with the base point given by the point $1\to X$ in the $T$-algebra $X.$
\end{theorem}   
\begin{proof} By assumption the square (\ref{double}) satisfies the conditions of Theorem \ref{recognition}. Therefore we have a homotopy pushout of nerves of classifiers over $T_*.$ For a fibrant replacement $fib(X)$ of a $T_*$-algebra $X$, we then have a homotopy pullback 
\[
		\xymatrix{
			\underline{\Alg_{T_*}}(N(T_*^{{T_{**}}}),{fib(X)})  \ar[rr] \ar[dd] && \underline{\Alg_{T_*}}(N(T_*^{{T_*}}),{fib(X)}) \ar[dd] \\
			\\
			\underline{\Alg_{T_*}}(N(T_*^{{T_*}}),{fib(X)}) \ar[rr]&& \underline{\Alg_{T_*}}(N(T_*^{{T}}),{fib(X)})
		}
\]
 By adjunctions $$\underline{\Alg_{T_*}}(N(T_*^{{T_{**}}}),{fib(X)}) \sim Map_{\Alg_{T_{**}}}(1,U^*X),$$ 
 $$\underline{\Alg_{T_*}}(N(T_*^{{T}}),{fib(X)}) \sim  Map_{\Alg_T}(1,u^*X)  $$ and 
 $$\underline{\Alg_{T_*}}(N(T_*^{{T_*}}),{fib(X)}) \sim  Map_{\Alg_{T_*}}(1,X) . $$
  But, in the category of pointed $T$-algebras, the terminal algebra is also the initial object so the space $Map_{\Alg_{T_*}}(1,X) $ is contractible. This completes the proof. 
\end{proof}

\begin{remark} The condition that the monad of pointed algebras is polynomial is not trivial. This is true for any tame monad in the sense of Batanin and Berger \cite{BB}. Yet, for example, it does not hold for the monad for symmetric operads.
\end{remark} 

Let $T$ be a polynomial monad with set of colours $I$ and let $i\in I.$ Let $Id$ be the identity polynomial monad
on $\Set$, and let $Id_I$ be the identity polynomial monad on $\Set^I.$ 
There is a cartesian map of polynomial monads
$$i: Id\to Id_I$$
which sends the unique colour to the element $i$ and the unique operation to the  identity on $i.$ 
(Both monads are small categories, the map $i$ corresponds to the functor from the terminal category which picks up the object $i.$)

We also have a one coloured polynomial monad $Id_+$  which `adds a point' to each set $X.$ 
Explicitly,  $Id_+$ is given by the following polynomial  
\[\label{I+}
			\xymatrix{
				1 && 1 \ar[ll]_-s \ar[rr]^-p &&2 \ar[rr]^-t && 1
			}
\]
where $2 $ is the set with two elements  $\{0,1\}$ and $p$ sends $1$ to $0.$ 
The algebras of $Id_+$ are pointed sets. Let now $T_+$ be the pushout of polynomial monads
\begin{equation} \label{ieta}
			\xymatrix{
				Id\ar[rr]^{i} \ar[dd]_{\epsilon} &&Id_I \ar[rr]^{\eta} && T\ar[dd]^{} \\
				\\
				Id_+ \ar[rrrr]^{} &&  && T_{+} \ .
			}
		\end{equation}

The algebras of $T_+$ are, therefore, the  algebras of $T$ equipped with a marked point in degree $i\in I.$ We called them {\it $i$-pointed $T$-algebras.}

Now let $$f:T\to S$$ be a map of polynomial monads. 
Suppose that the composite $i\circ \eta\circ f: Id\to S$ can be factorised through the unit of $Id_+.$ 
We then have a map of polynomial monads $$F:T_+\to S.$$ 
This factorisation condition says that the $\phi(i)$-th component $X_{\phi(i)}$ of any 
$S$-algebra $X$ has a `marked' point $1\to X_i$ and the restriction functor $f^*$ preserves this canonical point. 
\begin{theorem}[Formal Fibration sequence Theorem]\label{+point} If the square (\ref{ieta}) is homotopically cofinal then
for any simplicial $S$-algebra $X$, there is a fibration sequence
$$Map_{\Alg_{T_+}}(1, F^*X)\to Map_{\Alg_T}(1, f^*X) \to fib(X_i),$$
where  $fib(X_i)$ is a fibrant replacement for the simplicial set $X_i.
$\end{theorem}
 \begin{proof} 
Let $fib(X)$ be the fibrant replacement of the $S$-algebra $X.$ 
Observe then, that the $i$-th component $fib(X)_i$ is also a fibrant replacement $fib(X_i)$  for the simplicial set $X_i.$
 We have the following homotopy pushout of nerves of classifiers
 \[
			\xymatrix{
				N(S^{Id})\ar[rr]^{} \ar[dd]_{} && N(S^T)\ar[dd]^{} \\
				\\
				N(S^{Id_+}) \ar[rr]^{} && N(S^{T_+}) \ .
			}
			\]
We then have a homotopy pullback of simplicial sets    
\begin{equation}\label{homotopypb}
		\xymatrix{
			\underline{\Alg_S}(N(S^{{T_+}}),fib{X})  \ar[rr] \ar[dd] && \underline{\Alg_S}(N(S^{{T}}),fib{X}) \ar[dd] \\
			\\
			\underline{\Alg_S}(N(S^{{Id_+}}),fib{X}) \ar[rr]&& \underline{\Alg_S}(N(S^{{Id}}),fib{X}) \ .
		}
		\end{equation} 
  By adjunction { we have a simplicial set isomorphism}
   { $$\underline{\Alg_S}(N(S^{{Id}}),fib{X}) \cong \underline{\Alg_{Id}}(N(Id^{Id}), i^* \eta^* f^*(fib{X})) \ .$$}
   
   The space 
   { $i^* \eta^* f^*(fib{X})$} is just the simplicial set $fib(X_i)$ and we have  
  $$\underline{\Alg_{Id}}(N(Id^{Id}), 
  { i^* \eta^* f^*(fib{X})})\cong Map_{\Alg_{Id}}(1, fib(X_i)) \sim fib(X_i). $$
  Analogously:
  $$\underline{\Alg_S}(N(S^{{Id_+}}),fib{X}) \cong   \underline{\Alg_{Id_+}}(N(Id_+^{Id_+}), \alpha^*(fib{X}))$$ 
  where $\alpha:Id_+\to S.$ 
  It can be seen by a direct calculation or by the universal property that the classifier $Id_+^{Id_+}$  is just the pointed category with two objects $0$ (a point) and $1$, {and} one nontrivial arrow $0\to 1.$ The nerve of this category is a pointed simplicial interval. Hence,
 $ \underline{\Alg_{Id_+}}(N(Id_+^{Id_+}), \alpha^*(fib{X}))$ is the path space over $fib(X_i).$

These calculations show that the bottom arrow in the homotopy pullback (\ref{homotopypb}) is the classical path fibration over $fib(X_i).$ 
The rest of the proof follows from the usual adjunction argument.  \end{proof}

 \begin{remark} A special case of this  Theorem  is when $F$ is an identity map. Then, for any $i$-pointed $T$-algebra, we have a fibration sequence
 $$Map_{\Alg_{T_+}}(1, X)\to Map_{\Alg_T}(1, X) \to fib(X_i),$$ 
 where we skip notation for the forgetful functor from $i$-pointed $T$-algebras to $T$-algebras. 
 This is a conceptual explanation of Theorem \ref{+point}.
  \end{remark}
  

 \part{Applications}
 
\section{Multiplicative operads, bimodules and weak bimodules}

The non-symmetric operad which is equal to the unit $e$ in each degree is called {\it associativity operad} and is denoted by $Ass.$ 
If $\Ee$ is a cartesian category then $Ass = 1$ is the terminal object in the category of non-symmetric operads $\mathrm{NOp}.$  
\begin{remark}
		In the literature, $Ass$ is often used to denote the symmetrised version of our associative operad.
	\end{remark}
\begin{defin}
		A \emph{multiplicative non-symmetric operad} is a non-symmetric operad $\mathcal{O}$ together with an operadic morphism
		\[
		Ass \to \mathcal{O} \ .
		\]
	\end{defin}
The category of multiplicative non-symmetric operads  $  { \mathrm{NOp}_*}$ is the category $Ass/\mathrm{NOp} $. 	
We have a forgetful functor $$u^*:  { \mathrm{NOp}_*}\to \mathrm{NOp} \ .$$	

\begin{defin}
			Let $\mathcal{A}$ and $\mathcal{C}$ be two non-symmetric operads. An $\mathcal{A}\text{-}\mathcal{C}$-\emph{bimodule}  in a symmetric monoidal category $\Ee$ is given by
			\begin{itemize}
				\item an object $\mathcal{B}_n$ for all $n \geq 0$
				\item morphisms
				\[
				\mathcal{A}_k \otimes \mathcal{B}_{n_1} \otimes \ldots \otimes \mathcal{B}_{n_k} \to \mathcal{B}_{n_1 + \ldots + n_k}
				\]
				called left actions
				\item morphisms
				\[
				\mathcal{B}_k \otimes \mathcal{C}_{n_1} \otimes \ldots \otimes \mathcal{C}_{n_k} \to \mathcal{B}_{n_1 + \ldots + n_k}
				\]
				called right actions
			\end{itemize}
			satisfying the obvious analogue of the axioms for non-symmetric operad and a compatibility condition between left and right actions.
		\end{defin}
	\begin{remark}\label{0composition}  Notice that in the definition of left action one can take $k=0$ and so an empty product of $\mathcal{B}$s; that is, the tensor unit $e\in \Ee $. 
	On the right hand side we will have then an empty sum of natural numbers that is $0.$ 
	So we have a map $\mathcal{A}_0\to \mathcal{B}_0$ as one of the structure operations 
	for the left $\mathcal{A}$-module. 
	\end{remark}
	
\begin{remark}\label{assbimod} It is not hard to see that the left  $Ass$-module on a family $\mathcal{B}_n\in \Ee$ is given by an associative pairing:
$$\mathcal{B}_p\times \mathcal{B}_q\to \mathcal{B}_{p+q}$$ 
which is unital with respect to the unit $Ass_0\to \mathcal{B}_0.$ 

The right $Ass$-module structure is the same as a structure of a functor  $\mathcal{B}_n = \mathcal{B}([n])$ on the subcategory  $\Delta_{surj}\subset \Delta $ of  order-preserving surjections. For the $Ass\text{-}Ass$-bimodule these structures are required to be compatible in the obvious sense. 

\end{remark}

Recall that non-symmetric operads can be   defined in terms of $\circ_i$-operations \cite[Definition 11]{Markl}. The $\circ_i$-operations are obtained as the composites :
		\begin{equation}\label{circ}
			\mathcal{O}_k \otimes \mathcal{O}_n = \mathcal{O}_k \otimes e \otimes \ldots \otimes \mathcal{O}_n \otimes \ldots \otimes e \xrightarrow{1 \otimes \epsilon \otimes \ldots \otimes 1 \otimes \ldots \otimes \epsilon} \mathcal{O}_k \otimes \mathcal{O}_1 \otimes \ldots \otimes \mathcal{O}_n \otimes \ldots \otimes \mathcal{O}_1 \xrightarrow{m} \mathcal{O}_{k+n-1}
		\end{equation}

	\begin{defin}
		Let $\mathcal{A}$ and $\mathcal{C}$ be two non-symmetric operads.  A \emph{weak $\mathcal{A}-\mathcal{C}$-bimodule} $\mathcal{W}$  in a symmetric monoidal category $\Ee$ is given by
		\begin{itemize}
			\item an object $\mathcal{W}_n$ in $\Ee$ for all $n \geq 0$
			\item for $i = 1,\ldots,k$, morphisms
			\[
			\circ_i: \mathcal{A}_k \otimes \mathcal{W}_n \to \mathcal{W}_{k+n-1}
			\]
			called left action
			\item for $i = 1,\ldots,k$, morphisms
			\[
			\bullet_i: \mathcal{W}_k \otimes \mathcal{C}_n \to \mathcal{W}_{k+n-1}
			\]
			called right action
		\end{itemize}
		satisfying the analogue of the axioms for non-symmetric operads in terms of $\circ_i$-operations and a compatibility condition again.
	\end{defin}

\begin{remark}\label{weakbimod}It is easy to prove that a weak $Ass\text{-}Ass$-bimodule is the same as a functor $\Delta \to \Ee$; that is, a cosimplicial object in $\Ee.$ 
\end{remark}	
	
We use  the notations $\mathrm{Bimod}$ and $\mathrm{Wbimod}$ for the categories of $Ass\text{-}Ass$-bimodules and weak $Ass\text{-}Ass$-bimodules respectively.

There is a  forgetful functor $$v^*:{\mathrm{NOp}_*} \to \mathrm{Bimod}.$$ 
The category of pointed bimodules $\mathrm{Bimod}_*$ is the category $v^*Ass/\mathrm{Bimod}$. 	We have two forgetful functors $${b}^*: \mathrm{Bimod}_*\to \mathrm{Bimod}$$ and $${w}^*:\mathrm{Bimod}_* \to \mathrm{Wbimod}.$$

\begin{proposition} \begin{enumerate} \item There are polynomial monads $NOp,Bimod$ and $WBimod$ whose categories of algebras are isomorphic to the categories
$\mathrm  NOp,Bimod$ and $\mathrm WBimod$ respectively. 
Moreover, the functors ${{ v}}^*$ and ${ w}^*$ are isomorphic to the restriction functors along the maps of polynomial monads:
$${{ v} }:Bimod\to NOp_*.$$
and 
$${{w}}:WBimod \to Bimod_*.$$ 

\item The monads $NOp$ and $Bimod$ satisfy the conditions of Theorem \ref{pointedalgebras}.  
The functors $u^*$ and ${ b}^*$ are induced by the maps of polynomial monads
$$u: NOp \to NOp_*$$ and $${ b}:Bimod\to Bimod_* \ .$$ 
 
  \end{enumerate} 
  
\end{proposition} 
\begin{proof} For the description of the monad $NOp$ see Example \ref{example3}. 

The monads $Bimod$ and $WBimod$ indeed have been described by Turchin in \cite{T} without explicitly saying it. For this reason, and because in Section 10 we will describe a closely related construction, we give only a brief description now. For $Bimod$, the corresponding polynomial is given by 
   \[
			\xymatrix{
				\mathbb{N} && Btr^{*} \ar[ll]_-s \ar[rr]^-{p} && Btr \ar[rr]^-t && \mathbb{N} \ .
			}
			\] 			
Here $Btr$ is the set of isomorphism classes of certain planar trees with black and white vertices (called {\it beads} in \cite{T}[p.14]). The restrictions on the class of trees are \cite{T}: 	
\begin{enumerate} \item An edge cannot connect two black vertices; \item If a tree has a white vertex then the path between any leaf or stump (valency one vertex) and the root should pass through one and only one white vertex;
 \item There is exactly one tree without white vertices. This special tree has one black vertex and no leaves (this tree is here to accommodate a special operation described in the Remark \ref{0composition}.  
\end{enumerate} 
As Turchin notices, the {second} condition means that one can draw any such tree T in a way that all white vertices lie on the same horizontal line.	
\begin{remark} We use a more general type of planar tree here in comparison with \cite{T}[Part 1] because we allow vertices of valencies $1$ and $2.$ This kind of tree is considered in Part 2 of \cite{T}. 
\end{remark}

The rest of the data for the polynomial monad $Bimod$ is very similar to $NOp$. 
The set $Btr^*$ is the set $Btr$ with one of the white vertices marked and the middle map, source and target having the same description as for $NOp.$ 
The multiplication is induced by insertion of a tree to a marked vertex and contraction of all edges connecting black vertices.  

The monad $WBimod$ is given by the polynomial
\[
			\xymatrix{
				\mathbb{N} && WBtr^{*} \ar[ll]_-s \ar[rr]^-{p} && WBtr \ar[rr]^-t && \mathbb{N}
			}
			\] 
where the class of planar trees $WBtr$ consists of planar trees with black and exactly one white vertex (see \cite{T}[p.11]) and the rest of the structure is analogous to the previous.

 The polynomial monad $NOp_*$ has been described many times in the literature in terms of a $\Sigma$-free $\mathbb{N}$-colored operad of planar trees with white and black vertices (see \cite{BataninBergerLattice} for a description in terms of trees and lattice paths of complexity $2$ as well as for earlier references).   
The monad $Bimod_*$ can be given a similar description using the corresponding class of trees. We skip its description here because we will describe a very similar monad in Section \ref{cofinalresults}. 

{ The maps of polynomial monads ${{ v} }$, ${w}$, $u$ and ${ b} $ are easy to guess from the restrictions functors ${ v}^*$, ${w}^*$, $u^*$ and ${b}^*$.}

\end{proof}

\section{First cofinality theorem}\label{cofinalresults}

Many constructions of this and subsequent sections are instances of twisted Boardman-Vogt tensor product from Section \ref{BVP}. We use some visible special notations for corresponding polynomial monads for a convenience of the reader.

Consider the following polynomial monads: \begin{enumerate}
\item The monad $NOp_{\bullet \to \bullet \leftarrow \bullet}$ whose algebras are cospans
$$\mathcal{A}\to \mathcal{C} \leftarrow \mathcal{B}$$
of non-symmetric operads.
\item The polynomial monad $Bimod_{\bullet +\bullet}$ whose algebras are triples $(\mathcal{A},X, \mathcal{B})$ where $\mathcal{A},\mathcal{B}$ are non-symmetric operads and $X$ is a $1$-pointed $\mathcal{A}-\mathcal{B}$ bimodule.  
\end{enumerate}
 There is map of polynomial monads 
 \begin{equation}\label{map1} f: Bimod_{\bullet + \bullet} \to NOp_{\bullet\to \bullet \leftarrow \bullet}\end{equation}
  such that the restriction functor along it forgets operadic structure on $\mathcal{C}$ and remembers only $\mathcal{A}\text{-}\mathcal{B}$-bimodule structures induced by two operadic maps as well as the base point in $\mathcal{C}$ given by the unit of the operad $\mathcal{C}.$  



Analogously, let $Bimod_{\bullet \to \bullet \leftarrow \bullet}$ be the polynomial monad whose algebras are cospans
		$$\mathcal{A} \to \mathcal{C} \leftarrow \mathcal{B}$$
		of $Ass-Ass$-bimodules.
	
	Also let $Wbimod_{\bullet+\bullet}$ be the polynomial monad whose algebras are triples $(\mathcal{A},X,\mathcal{B})$ where $\mathcal{A},\mathcal{B}$ are $Ass\text{-}Ass$-bimodules and $X$ is a 
		$0$-pointed $\mathcal{A}\text{-}\mathcal{B}$-bimodule. This last object is given by a 
		family of sets ${X}_n, n\ge 0$ together with 
		\begin{enumerate}
			\item a point $\star \in X_0$
			\item left $\mathcal{A}$-action
			\[
			\mathcal{A}_{p} \times  X_{q}  \to X_{p+q}
			\]
			\item right $\mathcal{B}$-action
			\[
		X_{p} \times \mathcal{B}_{q}  \to X_{p+q}
			\]			
			\item { extension of $X_n = X([n])$ to a functor} on the subcategory  $\Delta_{surj} $ of $\Delta$ with order-preserving surjections as morphisms.
		\end{enumerate}
%
%
Both actions are required to be associative and unital with respect to the pairing defined in Remark \ref{assbimod} as well as compatible with each other in the usual  sense. They  also have to be natural with respect to the morphisms in $\Delta_{surj}.$

\begin{remark}
If in the definition of algebra of $Wbimod_{\bullet+\bullet}$ the bimodules  $\mathcal{A}=\mathcal{B}= {\color{black} 1}$, {\color{black} are both the terminal $Ass\text{-}Ass$-bimodule} then $X$ is a $0$-pointed weak $Ass\text{-}Ass$-bimodule.
	\end{remark}

There is a map of polynomial monads 
	\begin{equation}\label{map1} g: Wbimod_{\bullet+\bullet} \to Bimod_{\bullet \to \bullet \leftarrow \bullet}\end{equation}
	such that the restriction functor along it forgets the bimodule structure on $\mathcal{C}.$

\begin{theorem}\label{bullet} The maps 
 $$f: Bimod_{\bullet + \bullet} \to NOp_{\bullet\to \bullet \leftarrow \bullet}$$
 and 
 $$f: Wbimod_{\bullet+\bullet} \to Bimod_{\bullet \to \bullet \leftarrow \bullet}$$ 
are homotopically cofinal.
\end{theorem} 


\begin{proof}

We have to prove that the classifier
 $NOp_{\bullet\to \bullet \leftarrow \bullet}^{Bimod_{\bullet+\bullet}}$
 is contractible.
 
 First of all we need explicit descriptions of the polynomial monads $NOp_{\bullet\to \bullet \leftarrow \bullet}$ and $Bimod_{\bullet+\bullet}.$ 
 
 The monad 
  $NOp_{\bullet\to \bullet \leftarrow \bullet}$ is given by the polynomial
\begin{equation}\label{polyforspan}
			\xymatrix{
				\mathbb{N}\coprod \mathbb{N}\coprod\mathbb{N}  && Ptr_{1\mathrm{O2}}^{*} \ar[ll]_-s \ar[rr]^-{p} && Ptr_{1\mathrm{O}2} \ar[rr]^-t && \mathbb{N}\coprod \mathbb{N}\coprod\mathbb{N}
			}
			\end{equation}
 Here, $Ptr_{1\mathrm{O}2}$ is the set of isomorphism classes of planar trees whose vertices can have three colours: white or black of two types $1$ or $2.$  We also associate one of these colours to each tree (called its target colour). The only condition is that this colour is white type if the tree contains only white vertices or any mixed type vertices. If all vertices of the tree are black of the same type, the target colour can be of the same type as the vertices colour or white. 
 So, for such a tree, we necessary have two copies in the set of operations: one with the corresponding black colour as target and one with the white colour.  
 Also each copy of $\mathbb{N}$ has its own colour (again, white or black, $1$ or $2.$)  
  \begin{figure}[H]
		\caption{Typical tree from $Ptr_{1\mathrm{O}2}$}
		\[
			\begin{tikzpicture}[scale = .7]
				\draw (0,-1) -- (0,0);
				
				\draw (0,0) -- (-2,1);
				\draw (0,0) -- (2,1);
				
				\draw (-2,1) -- (-3.5,2);
				\draw (-2,1) -- (-2,2);
				\draw (-2,1) -- (-.5,2);
				
				\draw (2,1) -- (1,2);
				\draw (2,1) -- (3,2);
				
				\draw (-3.5,2) -- (-4.5,3);
				\draw (-3.5,2) -- (-3.5,3);
				\draw (-3.5,2) -- (-2.5,3);
				
				\draw (-.5,2) -- (-2,3);
				\draw (-.5,2) -- (-1,3);
				\draw (-.5,2) -- (0,3);
				\draw (-.5,2) -- (1,3);
				
				\draw (3,2) -- (2.3,3);
				\draw (3,2) -- (3.7,3);
				
				\draw[fill] (0,0) circle (3pt) node[right]{$2$};
				\draw[fill] (-2,1) circle (3pt) node[right]{$2$};
				\draw[fill=white] (2,1) circle (3pt);
				\draw[fill=white] (-3.5,2) circle (3pt);
				\draw[fill] (-.5,2) circle (3pt) node[right]{$1$};
				\draw[fill] (3,2) circle (3pt) node[right]{$1$};
			\end{tikzpicture}
		\]
	\end{figure}
As usual $Ptr^*_{1\mathrm{O}2}$ is the set $Ptr_{1\mathrm{O}2}$ with one vertex marked.  The source  map produces a natural number of the corresponding colour (white or black $1$ or $2$) depending on what kind of vertex is marked. The target map returns the number of leaves as before,  which is placed to a copy of $\mathbb{N}$  of target  colour of the tree.

The polynomial of the monad
  $Bimod_{\bullet+\bullet}$ is 
   \[
			\xymatrix{
				\mathbb{N}\coprod \mathbb{N}\coprod\mathbb{N}  && Ptr_{1\mathrm{B}2}^{*} \ar[ll]_-s \ar[rr]^-{p} && Ptr_{1\mathrm{B}2} \ar[rr]^-t && \mathbb{N}\coprod \mathbb{N}\coprod\mathbb{N}
			}
			\] 
   The set $Ptr_{1\mathrm{B}2}$ is a subset of $Ptr_{1\mathrm{O}2}$ whose elements subject to the following restrictions : 	
\begin{enumerate}
\item If a tree has all vertices black they must have the same type and the target type of such a tree is also the same;
 \item If there is a white vertex in a tree, the path between any leaf or stump and the root should pass through one and only one white vertex (second of Turchin's restrictions), the target type of the tree must be white in this case; 
\item Black vertices along a path like the one above have type $1$ before the path meets a white vertex and type $2$ after the white vertex;
\item There are exactly three copies of the tree without vertices (a free living edge) each of them having   $1$ as target but of different types: black $1,$  black $2$ or white.  
\end{enumerate} 

{ This last requirement corresponds to the constants in the theory. Each free living edge represents a nullary operation. Operations with black target  
represent units of each operad, whereas  the free livng edge with white target represents the base point in
the $1$-pointed bimodule.} 

\begin{figure}[H]
		\caption{Typical tree from $Ptr_{1\mathrm{B}2}$}
		\[
			\begin{tikzpicture}[scale = .7]
				\draw (0,-1) -- (0,0);
				
				\draw (0,0) -- (-2,1);
				\draw (0,0) -- (2,1);
				
				\draw (-2,1) -- (-3.5,2);
				\draw (-2,1) -- (-2,2);
				\draw (-2,1) -- (-.5,2);
				
				\draw (2,1) -- (1,2);
				\draw (2,1) -- (3,2);
				
				\draw (-3.5,2) -- (-4.5,3);
				\draw (-3.5,2) -- (-3.5,3);
				\draw (-3.5,2) -- (-2.5,3);
				
				\draw (-.5,2) -- (-2,3);
				\draw (-.5,2) -- (-1,3);
				\draw (-.5,2) -- (0,3);
				\draw (-.5,2) -- (1,3);
				
				\draw (3,2) -- (2.3,3);
				\draw (3,2) -- (3.7,3);
				
				\draw[fill=white] (-2,2) circle (3pt);
								\draw[fill] (0,0) circle (3pt) node[right]{$2$};
				\draw[fill] (-2,1) circle (3pt) node[right]{$2$};
				\draw[fill=white] (2,1) circle (3pt);
				\draw[fill=white] (-3.5,2) circle (3pt);
				\draw[fill=white] (-.5,2) circle (3pt) node[right]{$ $};
				\draw[fill] (3,2) circle (3pt) node[right]{$1$};
			\end{tikzpicture}
		\]
	\end{figure}

The map $f: Bimod_{\bullet + \bullet} \to NOp_{\bullet\to \bullet \leftarrow \bullet}$ is now obvious. It is the identity on colours and is the natural inclusion on other sets.

We are now ready to describe the classifier $NOp_{\bullet\to \bullet \leftarrow \bullet}^{Bimod_{\bullet+\bullet}}.$ 

By definition this is a family of categories
indexed by natural numbers of three types. If we consider restriction to any of the black colours $1$ and $2$ the corresponding classifier is isomorphic to the absolute classifier for non-symmetric operads, therefore the category $NOp_{\bullet\to \bullet \leftarrow \bullet}^{Bimod_{\bullet+\bullet}}(n_b),$ for any $n_b$ a `black' natural number, is contractible.    
Let us concentrate on the categories $NOp_{\bullet\to \bullet \leftarrow \bullet}^{Bimod_{\bullet+\bullet}}(n_w)$ indexed by white natural numbers. 
So, we fixed one of them. By general machinery of Batanin and Berger \cite{BB}:    
\begin{proposition}\label{genrel}
\begin{enumerate}
	\item the objects of this category are trees from $P_{1\mathrm{O}2}$ with exactly $n_w$ leaves. 
	\item the morphisms are generated by
	\begin{enumerate}
		\item \label{type1} contractions to a white vertex of edges where the upper vertices are black of the first type and the lower vertex is white
		
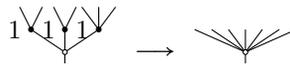
\[
		\begin{tikzpicture}[scale=0.3]
		\draw (0,-.5) -- (0,0);
		\draw (0,0) -- (-1.5,1);
		\draw (0,0) -- (0,1);
		\draw (0,0) -- (1.5,1);
		
		\draw (-1.5,1) -- (-2,2);
		\draw (-1.5,1) -- (-1,2);
		
		\draw (0,1) -- (-.5,2);
		\draw (0,1) -- (.5,2);
		
		\draw (1.5,1) -- (.75,2);
		\draw (1.5,1) -- (1.5,2);
		\draw (1.5,1) -- (2.25,2);
		
		\begin{scope}[shift={(8,0)}]
		\draw (0,0) -- (0,-.5);
		\draw (0,0) -- (-2.25,1);
		\draw (0,0) -- (-1.5,1);
		\draw (0,0) -- (-.75,1);
		\draw (0,0) -- (0,1);
		\draw (0,0) -- (.75,1);
		\draw (0,0) -- (1.5,1);
		\draw (0,0) -- (2.25,1);
		\draw[fill=white] (0,0) circle (3pt);
		\end{scope}
		
		\draw[fill=white] (0,0) circle (3pt);
		\draw[fill] (-1.5,1) circle (3pt) node[left]{$1$};
		\draw[fill] (0,1) circle (3pt) node[left]{$1$};
		\draw[fill] (1.5,1) circle (3pt) node[left]{$1$};
		\draw (4,0) node{$\longrightarrow$};
		\end{tikzpicture}
		\]
		\item \label{type2} contractions to a white vertex of edges where the upper vertices are white and the lower vertex is black of the second type
		
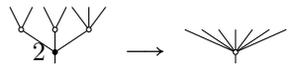
\[
		\begin{tikzpicture}[scale=0.3]
		\draw (0,-.5) -- (0,0);
		\draw (0,0) -- (-1.5,1);
		\draw (0,0) -- (0,1);
		\draw (0,0) -- (1.5,1);
		
		\draw (-1.5,1) -- (-2,2);
		\draw (-1.5,1) -- (-1,2);
		
		\draw (0,1) -- (-.5,2);
		\draw (0,1) -- (.5,2);
		
		\draw (1.5,1) -- (.75,2);
		\draw (1.5,1) -- (1.5,2);
		\draw (1.5,1) -- (2.25,2);
		
		\draw[fill] (0,0) circle (3pt) node[left]{$2$};
		\draw[fill=white] (-1.5,1) circle (3pt);
		\draw[fill=white] (0,1) circle (3pt);
		\draw[fill=white] (1.5,1) circle (3pt);
		
		\draw (4,0) node{$\longrightarrow$};
		
		\begin{scope}[shift={(8,0)}]
		\draw (0,0) -- (0,-.5);
		\draw (0,0) -- (-2.25,1);
		\draw (0,0) -- (-1.5,1);
		\draw (0,0) -- (-.75,1);
		\draw (0,0) -- (0,1);
		\draw (0,0) -- (.75,1);
		\draw (0,0) -- (1.5,1);
		\draw (0,0) -- (2.25,1);
		\draw[fill=white] (0,0) circle (3pt);
		\end{scope}
		\end{tikzpicture}
		\]
		
		\item  contractions of edges with black vertices of the same type;
		\item insertion of a vertex of valency $2$ of any of the types on an edge. For example:
		
		
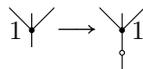
\[
			\begin{tikzpicture}[scale=0.3]
			\draw (0,-.75) -- (0,.75);
			\draw (0,0) -- (-1,1);
			\draw (0,0) -- (1,1);
			
			\draw (4,-1.75) -- (4,.75);
			\draw (4,0) -- (3,1);
			\draw (4,0) -- (5,1);
			
			
			\draw[fill] (0,0) circle (3pt) node[left]{$1$};
			\draw[fill=white] (4,-1) circle (3pt);
			\draw[fill] (4,0) circle (3pt) node[right]{$1$};
			\draw (2,0) node{$\longrightarrow$};
			\end{tikzpicture}
			\]

	\end{enumerate}
	\item the relations are generated by the usual bimodules relations, as well as the operadic relations. So, for example, the squares below commute :
	 
	 
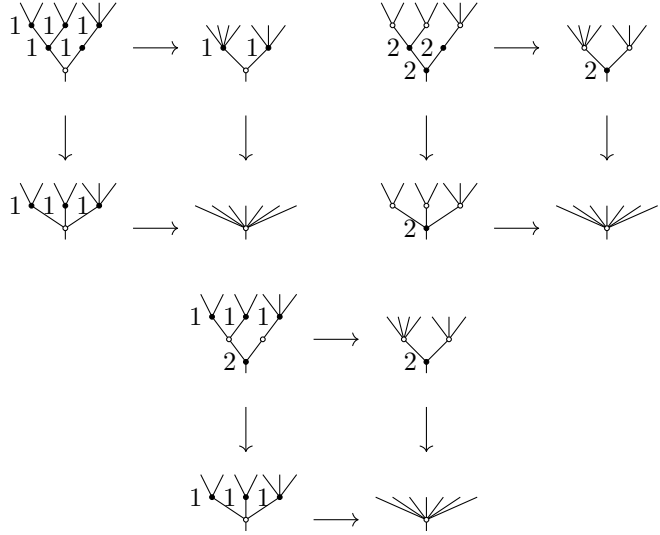
\[
			\begin{tikzpicture}[scale=0.3]
			\draw (0,-1.5) -- (0,-1);
			
			\draw (0,-1) -- (-.75,0);
			\draw (0,-1) -- (.75,0);
			
			\draw (-1.5,1) -- (-2,2);
			\draw (-1.5,1) -- (-1,2);
			
			\draw (.75,0) -- (1.5,1);
			
			\draw (0,1) -- (-.5,2);
			\draw (0,1) -- (.5,2);
			
			\draw (-.75,0) -- (-1.5,1);
			\draw (-.75,0) -- (0,1);
			
			\draw (1.5,1) -- (.75,2);
			\draw (1.5,1) -- (1.5,2);
			\draw (1.5,1) -- (2.25,2);
			
			\draw[fill=white] (0,-1) circle (3pt);
			\draw[fill] (-.75,0) circle (3pt) node[left]{$1$};
			\draw[fill] (.75,0) circle (3pt) node[left]{$1$};
			\draw[fill] (-1.5,1) circle (3pt) node[left]{$1$};
			\draw[fill] (0,1) circle (3pt) node[left]{$1$};
			\draw[fill] (1.5,1) circle (3pt) node[left]{$1$};
			
			\begin{scope}[shift={(8,0)}]
			\draw (0,-1.5) -- (0,-1);
			
			\draw (0,-1) -- (-1,0);
			\draw (0,-1) -- (1,0);
			
			\draw (-1,0) -- (-1.75,1);
			\draw (-1,0) -- (-1.25,1);
			\draw (-1,0) -- (-.75,1);
			\draw (-1,0) -- (-.25,1);
			
			\draw (1,0) -- (.25,1);
			\draw (1,0) -- (1,1);
			\draw (1,0) -- (1.75,1);
			
			%
			%
			
			\draw[fill=white] (0,-1) circle (3pt);
			\draw[fill] (-1,0) circle (3pt) node[left]{$1$};
			\draw[fill] (1,0) circle (3pt) node[left]{$1$};
			\end{scope}
			
			\begin{scope}[shift={(0,-8)}]
			\draw (0,-.5) -- (0,0);
			\draw (0,0) -- (-1.5,1);
			\draw (0,0) -- (0,1);
			\draw (0,0) -- (1.5,1);
			
			\draw (-1.5,1) -- (-2,2);
			\draw (-1.5,1) -- (-1,2);
			
			\draw (0,1) -- (-.5,2);
			\draw (0,1) -- (.5,2);
			
			\draw (1.5,1) -- (.75,2);
			\draw (1.5,1) -- (1.5,2);
			\draw (1.5,1) -- (2.25,2);
			
			\draw[fill=white] (0,0) circle (3pt);
			\draw[fill] (-1.5,1) circle (3pt) node[left]{$1$};
			\draw[fill] (0,1) circle (3pt) node[left]{$1$};
			\draw[fill] (1.5,1) circle (3pt) node[left]{$1$};
			\end{scope}
			
			\begin{scope}[shift={(8,-8)}]
			\draw (0,0) -- (0,-.5);
			\draw (0,0) -- (-2.25,1);
			\draw (0,0) -- (-1.5,1);
			\draw (0,0) -- (-.75,1);
			\draw (0,0) -- (0,1);
			\draw (0,0) -- (.75,1);
			\draw (0,0) -- (1.5,1);
			\draw (0,0) -- (2.25,1);
			\draw[fill=white] (0,0) circle (3pt);
			\end{scope}
			
			%
			
			\draw[->] (3,0) -- (5,0);
			\draw[->] (3,-8) -- (5,-8);
			\draw[->] (0,-3) -- (0,-5);
			\draw[->] (8,-3) -- (8,-5);
			\begin{scope}[shift={(16,0)}]
			\draw (0,-1.5) -- (0,-1);
			
			\draw (0,-1) -- (-.75,0);
			\draw (0,-1) -- (.75,0);
			
			\draw (-1.5,1) -- (-2,2);
			\draw (-1.5,1) -- (-1,2);
			
			\draw (.75,0) -- (1.5,1);
			
			\draw (0,1) -- (-.5,2);
			\draw (0,1) -- (.5,2);
			
			\draw (-.75,0) -- (-1.5,1);
			\draw (-.75,0) -- (0,1);
			
			\draw (1.5,1) -- (.75,2);
			\draw (1.5,1) -- (1.5,2);
			\draw (1.5,1) -- (2.25,2);
			
			\draw[fill] (0,-1) circle (3pt) node[left]{$2$};
			\draw[fill] (-.75,0) circle (3pt) node[left]{$2$};
			\draw[fill] (.75,0) circle (3pt) node[left]{$2$};
			\draw[fill=white] (-1.5,1) circle (3pt);
			\draw[fill=white] (0,1) circle (3pt);
			\draw[fill=white] (1.5,1) circle (3pt);
			
			\begin{scope}[shift={(8,0)}]
			\draw (0,-1.5) -- (0,-1);
			
			\draw (0,-1) -- (-1,0);
			\draw (0,-1) -- (1,0);
			
			\draw (-1,0) -- (-1.75,1);
			\draw (-1,0) -- (-1.25,1);
			\draw (-1,0) -- (-.75,1);
			\draw (-1,0) -- (-.25,1);
			
			\draw (1,0) -- (.25,1);
			\draw (1,0) -- (1,1);
			\draw (1,0) -- (1.75,1);
			
			%
			%
			
			\draw[fill] (0,-1) circle (3pt) node[left]{$2$};
			\draw[fill=white] (-1,0) circle (3pt);
			\draw[fill=white] (1,0) circle (3pt);
			\end{scope}
			
			\begin{scope}[shift={(0,-8)}]
			\draw (0,-.5) -- (0,0);
			\draw (0,0) -- (-1.5,1);
			\draw (0,0) -- (0,1);
			\draw (0,0) -- (1.5,1);
			
			\draw (-1.5,1) -- (-2,2);
			\draw (-1.5,1) -- (-1,2);
			
			\draw (0,1) -- (-.5,2);
			\draw (0,1) -- (.5,2);
			
			\draw (1.5,1) -- (.75,2);
			\draw (1.5,1) -- (1.5,2);
			\draw (1.5,1) -- (2.25,2);
			
			\draw[fill] (0,0) circle (3pt) node[left]{$2$};
			\draw[fill=white] (-1.5,1) circle (3pt);
			\draw[fill=white] (0,1) circle (3pt);
			\draw[fill=white] (1.5,1) circle (3pt);
			\end{scope}
			
			\begin{scope}[shift={(8,-8)}]
			\draw (0,0) -- (0,-.5);
			\draw (0,0) -- (-2.25,1);
			\draw (0,0) -- (-1.5,1);
			\draw (0,0) -- (-.75,1);
			\draw (0,0) -- (0,1);
			\draw (0,0) -- (.75,1);
			\draw (0,0) -- (1.5,1);
			\draw (0,0) -- (2.25,1);
			\draw[fill=white] (0,0) circle (3pt);
			\end{scope}
			
			%
			
			\draw[->] (3,0) -- (5,0);
			\draw[->] (3,-8) -- (5,-8);
			\draw[->] (0,-3) -- (0,-5);
			\draw[->] (8,-3) -- (8,-5);
			\end{scope}
			\end{tikzpicture}
			\]

			\[
			\begin{tikzpicture}[scale=0.3]
			\draw (0,-1.5) -- (0,-1);
			
			\draw (0,-1) -- (-.75,0);
			\draw (0,-1) -- (.75,0);
			
			\draw (-1.5,1) -- (-2,2);
			\draw (-1.5,1) -- (-1,2);
			
			\draw (.75,0) -- (1.5,1);
			
			\draw (0,1) -- (-.5,2);
			\draw (0,1) -- (.5,2);
			
			\draw (-.75,0) -- (-1.5,1);
			\draw (-.75,0) -- (0,1);
			
			\draw (1.5,1) -- (.75,2);
			\draw (1.5,1) -- (1.5,2);
			\draw (1.5,1) -- (2.25,2);
			
			\draw[fill] (0,-1) circle (3pt) node[left]{$2$};
			\draw[fill=white] (-.75,0) circle (3pt);
			\draw[fill=white] (.75,0) circle (3pt);
			\draw[fill] (-1.5,1) circle (3pt) node[left]{$1$};
			\draw[fill] (0,1) circle (3pt) node[left]{$1$};
			\draw[fill] (1.5,1) circle (3pt) node[left]{$1$};
			
			\begin{scope}[shift={(8,0)}]
			\draw (0,-1.5) -- (0,-1);
			
			\draw (0,-1) -- (-1,0);
			\draw (0,-1) -- (1,0);
			
			\draw (-1,0) -- (-1.75,1);
			\draw (-1,0) -- (-1.25,1);
			\draw (-1,0) -- (-.75,1);
			\draw (-1,0) -- (-.25,1);
			
			\draw (1,0) -- (.25,1);
			\draw (1,0) -- (1,1);
			\draw (1,0) -- (1.75,1);
			
			%
			%
			
			\draw[fill] (0,-1) circle (3pt) node[left]{$2$};
			\draw[fill=white] (-1,0) circle (3pt);
			\draw[fill=white] (1,0) circle (3pt);
			\end{scope}
			
			\begin{scope}[shift={(0,-8)}]
			\draw (0,-.5) -- (0,0);
			\draw (0,0) -- (-1.5,1);
			\draw (0,0) -- (0,1);
			\draw (0,0) -- (1.5,1);
			
			\draw (-1.5,1) -- (-2,2);
			\draw (-1.5,1) -- (-1,2);
			
			\draw (0,1) -- (-.5,2);
			\draw (0,1) -- (.5,2);
			
			\draw (1.5,1) -- (.75,2);
			\draw (1.5,1) -- (1.5,2);
			\draw (1.5,1) -- (2.25,2);
			
			\draw[fill=white] (0,0) circle (3pt);
			\draw[fill] (-1.5,1) circle (3pt) node[left]{$1$};
			\draw[fill] (0,1) circle (3pt) node[left]{$1$};
			\draw[fill] (1.5,1) circle (3pt) node[left]{$1$};
			\end{scope}
			
			\begin{scope}[shift={(8,-8)}]
			\draw (0,0) -- (0,-.5);
			\draw (0,0) -- (-2.25,1);
			\draw (0,0) -- (-1.5,1);
			\draw (0,0) -- (-.75,1);
			\draw (0,0) -- (0,1);
			\draw (0,0) -- (.75,1);
			\draw (0,0) -- (1.5,1);
			\draw (0,0) -- (2.25,1);
			\draw[fill=white] (0,0) circle (3pt);
			\end{scope}
			
			%
			
			\draw[->] (3,0) -- (5,0);
			\draw[->] (3,-8) -- (5,-8);
			\draw[->] (0,-3) -- (0,-5);
			\draw[->] (8,-3) -- (8,-5);
			\end{tikzpicture}
			\]
	\end{enumerate}
	\end{proposition}
	
\begin{remark} \label{replacement} The role of the morphism of insertion of a white vertex deserves a separate comment.  Combining this morphism with the morphism of bimodule contractions, we obtain two important families of morphisms in the classifier $NOp_{\bullet\to \bullet \leftarrow \bullet}^{Bimod_{\bullet+\bullet}}:$ 
\begin{enumerate}\item Morphisms which replace a black vertex of type $1$ by a white vertex:
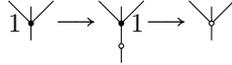
\[
			\begin{tikzpicture}[scale=0.3]
			\draw (0,-.75) -- (0,.75);
			\draw (0,0) -- (-1,1);
			\draw (0,0) -- (1,1);
			
			\draw (4,-1.75) -- (4,.75);
			\draw (4,0) -- (3,1);
			\draw (4,0) -- (5,1);
			
			\draw (8,-.75) -- (8,.75);
			\draw (8,0) -- (7,1);
			\draw (8,0) -- (9,1);
			
			\draw[fill] (0,0) circle (3pt) node[left]{$1$};
			\draw[fill=white] (4,-1) circle (3pt);
			\draw[fill] (4,0) circle (3pt) node[right]{$1$};
			\draw (2,0) node{$\longrightarrow$};
			\draw (6,0) node{$\longrightarrow$};
			\draw[fill=white] (8,0) circle (3pt);
			\end{tikzpicture}
			\]

\item Morphisms which replace a black vertex of type $2$ by a white vertex
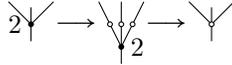
\[
			\begin{tikzpicture}[scale=0.3]
			\draw (0,-.75) -- (0,.75);
			\draw (0,0) -- (-1,1);
			\draw (0,0) -- (1,1);
			
			\draw (4,-1.75) -- (4,1);
			\draw (4,-1) -- (3,1);
			\draw (4,-1) -- (5,1);
			
			\draw (8,-.75) -- (8,.75);
			\draw (8,0) -- (7,1);
			\draw (8,0) -- (9,1);
			
			\draw[fill] (0,0) circle (3pt) node[left]{$2$};
			\draw[fill=white] (4,0) circle (3pt);
			\draw[fill=white] (3.5,0) circle (3pt);	
			\draw[fill=white] (4.5,0) circle (3pt);			
			\draw[fill] (4,-1) circle (3pt) node[right]{$2$};
			\draw (2,0) node{$\longrightarrow$};
			\draw (6,0) node{$\longrightarrow$};
			\draw[fill=white] (8,0) circle (3pt);
			\end{tikzpicture}
			\]

\end{enumerate}

There is one exception to the rule of replacement of a black vertex of type $2$, 
namely, if such a vertex has valency $1.$
In this case, however, we have a special operation in the left bimodule (see Remark \ref{0composition}) which provides a morphism
\[
			\begin{tikzpicture}[scale=0.3]
			\draw (0,-.75) -- (0,0);
			
			\draw (4,-.75) -- (4,0);

			\draw[fill] (0,0) circle (3pt) node[left]{$2$};
			\draw[fill=white] (4,0) circle (3pt);
			\draw (2,0) node{$\longrightarrow$};
			\end{tikzpicture}
			\]

These morphisms make the classifier $NOp_{\bullet\to \bullet \leftarrow \bullet}^{Bimod_{\bullet+\bullet}}$ look very similar to the absolute classifier
$NOp_{\bullet\to \bullet \leftarrow \bullet}^{NOp_{\bullet\to \bullet \leftarrow \bullet}}.$ Indeed, as categories they have the same objects  and in both categories there are morphisms of  replacing black vertices by white vertices. In the absolute classifier $NOp_{\bullet\to \bullet \leftarrow \bullet}^{NOp_{\bullet\to \bullet \leftarrow \bullet}}$ these morphisms correspond to two internal operad morphisms, which is a part of the $NOp_{\bullet\to \bullet \leftarrow \bullet}$-algebra structure. 
The difference between these categories is that in $NOp_{\bullet\to \bullet \leftarrow \bullet}^{Bimod_{\bullet+\bullet}}$ there are no contractions of edges connecting white vertices, which reflects the fact that we just have an internal bimodule, not an operad.   \end{remark}
	
\begin{remark} \label{RNOP}One can develop a very similar theory in the case of reduced operads and bimodules (this is the case of Turchin's paper \cite{T}[Part 1]). Reduced means that operads and bimodules we consider are such that $\mathcal{X}(0) = \emptyset$ and  $\mathcal{X}(1) = 1.$ All $\mathcal{A}\text{-}\mathcal{B}$-bimodules are, therefore, canonically $1$-pointed. The corresponding classifiers do not contains vertices of valencies $1$ and $2$ and, hence, are finite. Indeed they are finite posets.  The following is a picture of the category
$RNOp_{\bullet\to \bullet \leftarrow \bullet}^{RBimod_{\bullet+\bullet}}(3_w):$ 

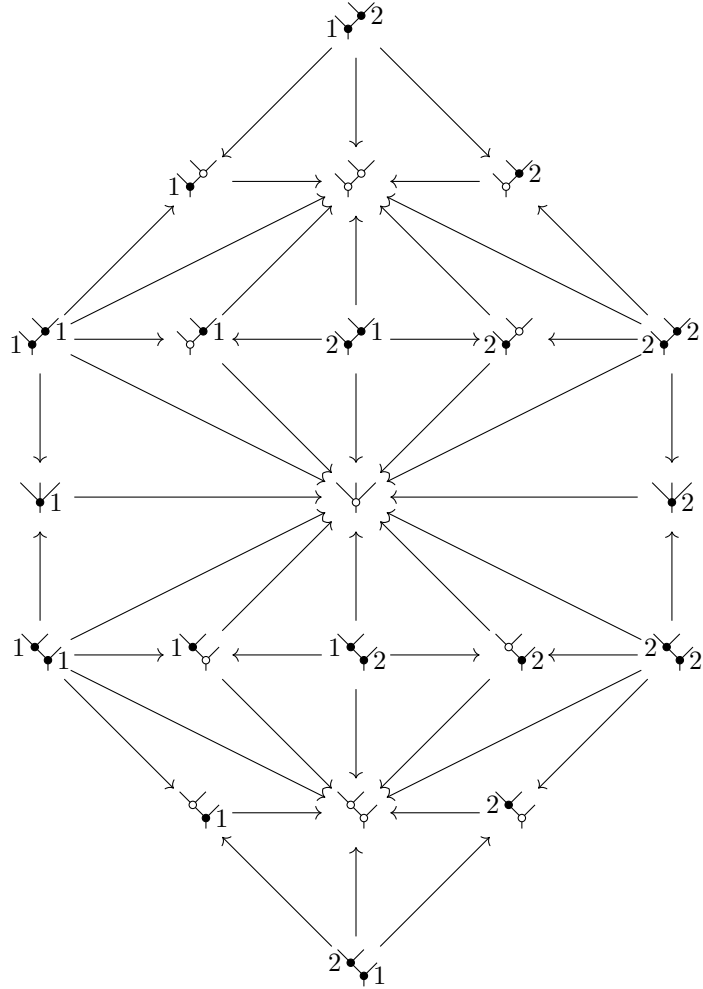
			\def\lefttree#1#2#3{
				\begin{scope}[shift={#1}]
					\begin{scope}[shift={(.15,0)}]
						\draw (0,0) -- (-.5,.5);
						\draw (0,0) -- (.25,.25);
						\draw (-.25,.25) -- (0,.5);
						\draw (0,0) -- (0,-.2);
						\ifthenelse
						{
							#2 = 1
						}
						{
							\draw[fill] (0,0) circle (2pt) node[right]{$1$};
						}
						{
							\ifthenelse
							{
								#2 = 2
							}
							{
								\draw[fill] (0,0) circle (2pt) node[right]{$2$};
							}
							{
								\draw[fill=white] (0,0) circle (2pt);
							}
						}
						\ifthenelse
						{
							#3 = 1
						}
						{
							\draw[fill] (-.25,.25) circle (2pt) node[left]{$1$};
						}
						{
							\ifthenelse
							{
								#3 = 2
							}
							{
								\draw[fill] (-.25,.25) circle (2pt) node[left]{$2$};
							}
							{
								\draw[fill=white] (-.25,.25) circle (2pt);
							}
						}
					\end{scope}
				\end{scope}
			}
			
			\def\righttree#1#2#3{
				\begin{scope}[shift={#1}]
					\begin{scope}[shift={(-.15,0)}]
						\draw (0,0) -- (.5,.5);
						\draw (0,0) -- (-.25,.25);
						\draw (.25,.25) -- (0,.5);
						\draw (0,0) -- (0,-.2);
						\ifthenelse
						{
							#2 = 1
						}
						{
							\draw[fill] (0,0) circle (2pt) node[left]{$1$};
						}
						{
							\ifthenelse
							{
								#2 = 2
							}
							{
								\draw[fill] (0,0) circle (2pt) node[left]{$2$};
							}
							{
								\draw[fill=white] (0,0) circle (2pt);
							}
						}
						\ifthenelse
						{
							#3 = 1
						}
						{
							\draw[fill] (.25,.25) circle (2pt) node[right]{$1$};
						}
						{
							\ifthenelse
							{
								#3 = 2
							}
							{
								\draw[fill] (.25,.25) circle (2pt) node[right]{$2$};
							}
							{
								\draw[fill=white] (.25,.25) circle (2pt);
							}
						}
					\end{scope}
				\end{scope}
			}
			
			\def\middletree#1#2{
				\begin{scope}[shift={#1}]
					\draw (0,0) -- (-.375,.375);
					\draw (0,0) -- (0,.375);
					\draw (0,0) -- (.375,.375);
					\draw (0,0) -- (0,-.2);
					\ifthenelse
					{
						#2 = 1
					}
					{
						\draw[fill] (0,0) circle (2pt) node[right]{$1$};
					}
					{
						\ifthenelse
						{
							#2 = 2
						}
						{
							\draw[fill] (0,0) circle (2pt) node[right]{$2$};
						}
						{
							\draw[fill=white] (0,0) circle (2pt);
						}
					}
				\end{scope}
			}
			
			\def\line#1#2#3#4{
				\draw[->] ({#1 + 1.3*(#3-#1)/sqrt(((#3-#1)*(#3-#1))+((#4-#2)*(#4-#2)))/2},{#2 + 1.3*(#4-#2)/sqrt(((#3-#1)*(#3-#1))+((#4-#2)*(#4-#2)))/2}) -- ({#3 - 1.3*(#3-#1)/sqrt(((#3-#1)*(#3-#1))+((#4-#2)*(#4-#2)))/2},{#4 - 1.3*(#4-#2)/sqrt(((#3-#1)*(#3-#1))+((#4-#2)*(#4-#2)))/2});
			}

			\begin{figure}[H]\label{fullcategory}
				\[
				\begin{tikzpicture}[scale = 0.7]
				\begin{scope}[shift={(0,.1)}]
				\line{-6}{-3}{0}{0};
				\line{-3}{-3}{0}{0};
				\line{0}{-3}{0}{0};
				\line{3}{-3}{0}{0};
				\line{6}{-3}{0}{0};
				
				\line{-6}{-3}{-3}{-3};
				\line{0}{-3}{-3}{-3};
				\line{0}{-3}{3}{-3};
				\line{6}{-3}{3}{-3};
				
				\line{-6}{0}{0}{0};
				\line{6}{0}{0}{0};
				
				\line{-6}{3}{0}{0};
				\line{-3}{3}{0}{0};
				\line{0}{3}{0}{0};
				\line{3}{3}{0}{0};
				\line{6}{3}{0}{0};
				
				\line{-6}{3}{-3}{3};
				\line{0}{3}{-3}{3};
				\line{0}{3}{3}{3};
				\line{6}{3}{3}{3};
				
				\line{-6}{-3}{0}{-6};
				\line{-3}{-3}{0}{-6};
				\line{0}{-3}{0}{-6};
				\line{3}{-3}{0}{-6};
				\line{6}{-3}{0}{-6};
				\line{-3}{-6}{0}{-6};
				\line{0}{-9}{0}{-6};
				\line{3}{-6}{0}{-6};
				
				\line{-6}{3}{0}{6};
				\line{-3}{3}{0}{6};
				\line{0}{3}{0}{6};
				\line{3}{3}{0}{6};
				\line{6}{3}{0}{6};
				\line{-3}{6}{0}{6};
				\line{0}{9}{0}{6};
				\line{3}{6}{0}{6};
				
				\line{0}{-9}{-3}{-6};
				\line{0}{-9}{3}{-6};
				\line{-6}{-3}{-3}{-6};
				\line{6}{-3}{3}{-6};
				\line{-6}{-3}{-6}{0};
				\line{6}{-3}{6}{0};
				
				\line{0}{9}{-3}{6};
				\line{0}{9}{3}{6};
				\line{-6}{3}{-3}{6};
				\line{6}{3}{3}{6};
				\line{-6}{3}{-6}{0};
				\line{6}{3}{6}{0};
				\end{scope}
				
				\lefttree{(0,-9)}{1}{2};
				\lefttree{(-3,-6)}{1}{3};
				\lefttree{(0,-6)}{3}{3};
				\lefttree{(3,-6)}{3}{2};
				\lefttree{(-6,-3)}{1}{1};
				\lefttree{(-3,-3)}{3}{1};
				\lefttree{(0,-3)}{2}{1};
				\lefttree{(3,-3)}{2}{3};
				\lefttree{(6,-3)}{2}{2};
				\middletree{(-6,0)}{1};
				\middletree{(0,0)}{3};
				\middletree{(6,0)}{2};
				\righttree{(0,9)}{1}{2};
				\righttree{(-3,6)}{1}{3};
				\righttree{(0,6)}{3}{3};
				\righttree{(3,6)}{3}{2};
				\righttree{(-6,3)}{1}{1};
				\righttree{(-3,3)}{3}{1};
				\righttree{(0,3)}{2}{1};
				\righttree{(3,3)}{2}{3};
				\righttree{(6,3)}{2}{2};
				\end{tikzpicture}
				\]
			\end{figure}
	
	The nerve of this poset is clearly contractible. It is also visible what makes it work. The operadic contractions in this picture are replaced by bimodule contractions.  If we are able to invert those morphisms, we get an internal operad structure on corollas with white vertices together with two operadic maps from two internal operads formed by black corollas. This is the conceptual main point of our theorem.  
\end{remark}
			
	We continue with a formal proof of contractibility of the classifier $NOp_{\bullet\to \bullet \leftarrow \bullet}^{Bimod_{\bullet+\bullet}}.$ 
	
		There is a map of polynomial monads 
	\begin{equation}\label{uuu} u: NOp_{\bullet\to \bullet \leftarrow \bullet} \to NOp\end{equation} 
	constructed as follows.
	On colours  $$\mathbb{N}\coprod\mathbb{N}\coprod\mathbb{N} \to \mathbb{N}$$ 	
	it is an identity on each summand. On operations it forgets about all colours (vertices colours as well as target colours) of trees from $Ptr_{1O2}.$ In other words it only remembers the shape of the tree. This is a cartesian map and the restriction functor along this map applied to an operad $\mathcal{O}$ returns a cospan     
	$\mathcal{O}\stackrel{id}{\to} \mathcal{O} \stackrel{id}{\leftarrow} \mathcal{O}.$ 
	We then have the following commutative square of cartesian maps of polynomial monads:
	\begin{equation}
			\xymatrix{
				Bimod_{\bullet+\bullet}\ar[rr]^{f\circ u} \ar[dd]_{f} && NOp\ar[dd]^{id} \\
				\\
				NOp_{\bullet\to \bullet \leftarrow \bullet} \ar[rr]^{u} && NOp
			}
			\end{equation}	
By Proposition \ref{square}, we then have a map of classifiers
\begin{equation}\label{NOPBIMOD}F: NOp_{\bullet\to \bullet \leftarrow \bullet}^{Bimod_{\bullet+\bullet}}\to u^*(NOp^{NOp}) \ .\end{equation}	
We are going to prove that the nerve of this map is a weak equivalence which, of course, will imply the contractibility of $NOp_{\bullet\to \bullet \leftarrow \bullet}^{Bimod_{\bullet+\bullet}}.$ 
The map (\ref{NOPBIMOD}) on objects has the same effect as (\ref{uuu}) on operations. It also maps contractions to contractions while any morphism which comes from replacement of black vertices by white vertex is mapped to the identity.

We fix a particular $n\in\mathbb{N}$ and consider the restriction  of the map of the classifiers (\ref{NOPBIMOD}) to the component  indexed by $n_w$ (it is clear that on  components with black colours the map (\ref{NOPBIMOD}) is an isomorphism).   To make our notation less heavy we give names to two of our categories as follows:
$$NOp_{\bullet\to \bullet \leftarrow \bullet}^{Bimod_{\bullet+\bullet}}(n_w) =\mathbb{C}$$	
and 
$$u^*(NOp^{NOp})(f(n_w)) = \mathbb{D}.$$
The restriction of (\ref{NOPBIMOD})	on this components will simply be denoted by $F.$  	
	
	For any $d \in \mathbb{D}$, we write $F_d$ for the (strict) fiber of $F$ over $d$; that is, the full subcategory of objects $c \in \mathbb{C}$ such that $F(c) = d$.

	\begin{defin}[\cite{cis06}]
		A functor $F : \mathbb{C} \to \mathbb{D}$ is \emph{smooth} if, for all $d \in \mathbb{D}$, the canonical functor
		\[
		F_d \to d / F
		\]
		induces a weak equivalence of nerves.
		
		Dually a functor $F : \mathbb{C} \to \mathbb{D}$ is \emph{proper} if, for all $d \in \mathbb{D}$, the canonical functor
		\[
		F_d \to F/ d
		\]
		induces a weak equivalence of nerves.		
	\end{defin}
	
	We have the following lemma \cite[Proposition 5.3.4]{cis06} :
	
	\begin{lem}\label{Cisinski}
		A functor $F: \mathbb{C} \to \mathbb{D}$ is smooth if and only if for all maps $f_1 : d_0 \to d_1$ in $\mathbb{D}$ and all objects $c_1$ in $F_{d_1}$, the nerve of the `lifting' category $\mathbb{C}(c_1,f_1)$ of $f_1$ over $c_1$, whose objects are arrows $f: c \to c_1$ such that $F(f) = f_1$ and arrows are commutative triangles :
		
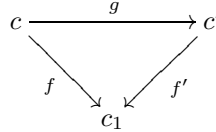
\[
		\xymatrix{
			c \ar[rr]^g \ar[rd]_f && c' \ar[ld]^{f'} \\
			& c_1
		}
		\]
		with $g$ a morphism in $F_{d_0}$, is contractible.
		
		{ There is a} dual characterisation for proper functors.
	\end{lem}

	We call any tree {in $\mathbb{C}$} which has only one vertex a {\em corolla}. We call the unique vertex connected to the {\em root vertex}. The {\em root path} from a vertex is the path from this vertex to the root vertex. Using a description of the category $\mathbb{C}$ from Proposition \ref{genrel}, we have the following lemma:	
	\begin{lem}\label{characterisation}[Characterisation of trees that can be contracted to a corolla]
		Let $c \in \mathbb{C}$ be a corolla. There is a morphism from a tree $a$ to $c$ in $\mathbb{C}$  if and only if the tree $a$ satisfies the following property:
\begin{itemize}\item[$*$] the root path from any vertex consists of a sequence of black vertices of type $1$ { followed by} at most one white vertex { followed by} a sequence of black vertices of type $2$. \end{itemize}			
	\end{lem}
\begin{figure}[H]
		\caption{Tree satisfying the property $*$}
		\[
			\begin{tikzpicture}[scale = .7]
				\draw (0,-1) -- (0,0);
				
				\draw (0,0) -- (-2,1);
				\draw (0,0) -- (2,1);
				
				\draw (-2,1) -- (-3.5,2);
				\draw (-2,1) -- (-2,2);
				\draw (-2,1) -- (-.5,2);
				
				\draw (2,1) -- (1,2);
				\draw (2,1) -- (3,2);
				
				\draw (-3.5,2) -- (-4.5,3);
				\draw (-3.5,2) -- (-3.5,3);
				\draw (-3.5,2) -- (-2.5,3);
				
				\draw (-.5,2) -- (-2,3);
				\draw (-.5,2) -- (-1,3);
				\draw (-.5,2) -- (0,3);
				\draw (-.5,2) -- (1,3);
				
				\draw (3,2) -- (2.3,3);
				\draw (3,2) -- (3.7,3);
				
				\draw[fill] (0,0) circle (3pt) node[right]{$2$};
				\draw[fill] (-2,1) circle (3pt) node[right]{$2$};
				\draw[fill=white] (2,1) circle (3pt);
				\draw[fill=white] (-3.5,2) circle (3pt);
				\draw[fill] (-.5,2) circle (3pt) node[right]{$1$};
				\draw[fill] (3,2) circle (3pt) node[right]{$1$};
			\end{tikzpicture}
		\]
	\end{figure}

	\begin{proof}
		Assume that $a$ satisfies the property $*$. We use induction on the maximal length of a root path in the tree $a.$ If $a$ is a corolla, then the statement is trivial. If $a$ is not a corolla, one can consider $a$ to be a forest of branches joined at the root vertex.  All the branches satisfy the property $*$ and have maximal path root length strictly less than $a.$ By induction, they can therefore be contracted to a corolla. If the root vertex is black of type $2$, the remaining tree can always be contracted. If not, then all the vertices above the root vertex in $a$ are black of type $1$, and the remaining tree can also be contracted.
		
		For the other direction of the equivalence, assume that $a$ does not satisfy the property $*$. Notice that if a tree $a$ does not satisfy the property $*$ and $a \to b$ is a generating morphism, then $b$ does not satisfy the property $*$. Yet the corolla satisfies the property $*$, which is a contradiction.
	\end{proof}
	
	\begin{lem}\label{smooth}
		The functor $F : \mathbb{C} \to \mathbb{D}$ is smooth.
	\end{lem}
	
	\begin{proof}
		Let $f_1 : d_0 \to d_1$ in $\mathbb{D}$ and $c_1$ in $F_{d_1}$. According to the Lemma \ref{Cisinski}, we have to prove that the nerve of the category $\mathbb{C}(c_1,f_1)$ is contractible.
		
		{ Observe that $\mathbb{C}(c_1,f_1)$ is isomorphic to the full subcategory of $\mathbb{C}$ consisting of trees with the same shape as $d_0$ which can be contracted to $c_1$. Moreover, $\mathbb{C}(c_1,f_1)$ is also isomorphic to a product of categories $\mathbb{C}(c^v_1,f^v_1)$ where
			\begin{itemize}
				\item  $v$ runs over the vertices of $d_1$
				\item $d^v_1$ is the corolla whose set of leaves is equal to the set of incoming edges of $v$
				\item $d^v_0$ is the subtree of $d_0$ containing all the vertices that are sent to $v$ through $f_1$
				\item $f^v_1 : d^v_0 \to d^v_1$ is the contraction to the corolla
				\item $c^v_1$ is the corolla with the same color as the vertex $v$ of $c_1$
			\end{itemize}
			
			If the unique vertex of $c^v_1$ is black, then $\mathbb{C}(c^v_1,f^v_1)$ is the terminal category and its nerve is contractible. We can therefore assume that $c^v_1$ is a corolla with a white vertex.
			
			In conclusion, all we have to prove is the contractibility of the nerve of $\mathbb{C}_d$, which is defined as the subcategory of trees in $\mathbb{C}$ with the same shape as $d \in \mathbb{D}$ and that can be contracted to a corolla with a white vertex.}
		
		
		Let us introduce the following notations :
		
		\begin{itemize}
			\item We write $\mathbb{C}^1$ for the full subcategory of 
			{$\mathbb{C}_d$} consisting of trees whose all non root vertices are black of type $1$.
			\item We write $\mathbb{C}_2$ for the full subcategory of 
			{$\mathbb{C}_d$} consisting of trees whose root vertex is black of type $2$.
		\end{itemize}
		
		\def\tree#1#2#3#4{
			\begin{scope}[shift={#1}]
				\begin{scope}[shift={(0,-.15)}]
					\draw (0,0) -- (.25,.25);
					\draw (0,0) -- (-.25,.25);
					\draw (.25,.25) -- (.45,.45);
					\draw (.25,.25) -- (.05,.45);
					\draw (-.25,.25) -- (-.45,.45);
					\draw (-.25,.25) -- (-.05,.45);
					\draw (0,0) -- (0,-.2);
					\ifthenelse
					{
						#2 = 1
					}
					{
						\draw[fill] (0,0) circle (2pt) node[left]{$1$};
					}
					{
						\ifthenelse
						{
							#2 = 2
						}
						{
							\draw[fill] (0,0) circle (2pt) node[left]{$2$};
						}
						{
							\draw[fill=white] (0,0) circle (2pt);
						}
					}
					\ifthenelse
					{
						#3 = 1
					}
					{
						\draw[fill] (-.25,.25) circle (2pt) node[left]{$1$};
					}
					{
						\ifthenelse
						{
							#3 = 2
						}
						{
							\draw[fill] (-.25,.25) circle (2pt) node[left]{$2$};
						}
						{
							\draw[fill=white] (-.25,.25) circle (2pt);
						}
					}
					\ifthenelse
					{
						#4 = 1
					}
					{
						\draw[fill] (.25,.25) circle (2pt) node[right]{$1$};
					}
					{
						\ifthenelse
						{
							#4 = 2
						}
						{
							\draw[fill] (.25,.25) circle (2pt) node[right]{$2$};
						}
						{
							\draw[fill=white] (.25,.25) circle (2pt);
						}
					}
				\end{scope}
			\end{scope}
		}
		
		\def\line#1#2#3#4{
			\draw[->] ({#1 + 1.3*(#3-#1)/sqrt(((#3-#1)*(#3-#1))+((#4-#2)*(#4-#2)))/2},{#2 + 1.3*(#4-#2)/sqrt(((#3-#1)*(#3-#1))+((#4-#2)*(#4-#2)))/2}) -- ({#3 - 1.3*(#3-#1)/sqrt(((#3-#1)*(#3-#1))+((#4-#2)*(#4-#2)))/2},{#4 - 1.3*(#4-#2)/sqrt(((#3-#1)*(#3-#1))+((#4-#2)*(#4-#2)))/2});
		}
		
		\begin{figure}[H]
			\caption{The subcategories $\mathbb{C}^1$ and $\mathbb{C}_2$}
			\[
			\begin{tikzpicture}[scale=.7]
			\begin{scope}
				\draw (0,-1) arc (-90:90:1);
				\draw (-6,1) arc (90:270:1);
				\draw (-6,1) -- (0,1);
				\draw (-6,-1) -- (0,-1);
				\draw (-3,-1) node[below]{$\mathbb{C}^1$};
				
				\draw ({-sqrt(2)/2},{sqrt(2)/2}) -- ({3-sqrt(2)/2},{3+sqrt(2)/2});
				\draw ({3+sqrt(2)/2},{3+sqrt(2)/2}) -- ({6+sqrt(2)/2},{sqrt(2)/2});
				\draw ({6+sqrt(2)/2},{-sqrt(2)/2}) -- ({3+sqrt(2)/2},{-3-sqrt(2)/2});
				\draw ({3-sqrt(2)/2},{-3-sqrt(2)/2}) -- ({-sqrt(2)/2},{-sqrt(2)/2});
				
				\draw ({-sqrt(2)/2},{sqrt(2)/2}) arc (135:225:1);
				\draw ({3-sqrt(2)/2},{-3-sqrt(2)/2}) arc (-135:-45:1);
				\draw ({6+sqrt(2)/2},{-sqrt(2)/2}) arc (-45:45:1);
				\draw ({3+sqrt(2)/2},{3+sqrt(2)/2}) arc (45:135:1);
				\draw (3,-4) node[below]{$\mathbb{C}_2$};
			\end{scope}
			
			\tree{(-6,0)}{1}{1}{1};
			\tree{(-3,0)}{3}{1}{1};
			\tree{(0,0)}{2}{1}{1};
			\tree{(3,0)}{2}{3}{3};
			\tree{(6,0)}{2}{2}{2};
			\tree{(1.5,1.5)}{2}{3}{1};
			\tree{(3,3)}{2}{2}{1};
			\tree{(4.5,1.5)}{2}{2}{3};
			\tree{(1.5,-1.5)}{2}{1}{3};
			\tree{(3,-3)}{2}{1}{2};
			\tree{(4.5,-1.5)}{2}{3}{2};
			
			\line{-6}{0}{-3}{0};
			\line{0}{0}{-3}{0};
			
			\line{0}{0}{3}{0};
			\line{1.5}{-1.5}{3}{0};
			\line{3}{-3}{3}{0};
			\line{4.5}{-1.5}{3}{0};
			\line{6}{0}{3}{0};
			\line{4.5}{1.5}{3}{0};
			\line{3}{3}{3}{0};
			\line{1.5}{1.5}{3}{0};
			
			\line{0}{0}{1.5}{-1.5};
			\line{3}{-3}{1.5}{-1.5};
			\line{3}{-3}{4.5}{-1.5};
			\line{4.5}{-1.5}{6}{0};
			\line{6}{0}{4.5}{1.5};
			\line{3}{3}{4.5}{1.5};
			\line{3}{3}{1.5}{1.5};
			\line{0}{0}{1.5}{1.5};
			\end{tikzpicture}
			\]
		\end{figure}
			
			{ Using lemma \ref{characterisation}, we deduce that $\mathbb{C}_d = \mathbb{C}^1 \cup \mathbb{C}_2$.}
		
			The subcategory $\mathbb{C}^1$ contains only $3$ objects, and one of them is terminal for this subcategory, namely, the tree where the root vertex is white and all other vertices are black of type $1$.
			
			The subcategory $\mathbb{C}_2$ is isomorphic to the product of the subcategories $\mathbb{C}_{d_1},\ldots,\mathbb{C}_{d_k}$, where $d_1,\ldots,d_k$ are the subtrees coming up from the root vertex of $d$. Indeed, an object $c \in \mathbb{C}_2$ can be associated to a sequence of objects $c_1,\ldots,c_k \in \mathbb{C}_{d_1},\ldots,\mathbb{C}_{d_k}$ by taking the subtrees coming up from the root vertex of $d$. This association is obviously a bijection thanks to the characterization of lemma \ref{characterisation}. By induction, the nerves of $\mathbb{C}_{d_1},\ldots,\mathbb{C}_{d_k}$ are contractible. Therefore, the nerve of $\mathbb{C}_2$ is also contractible.
			
			Finally, the intersection $\mathbb{C}^1 \cap \mathbb{C}_2$ contains only one object, namely the tree where the root vertex is black of type $2$ and all the other vertices are black of type $1$.
		
			
			
		
	
	
	
\end{proof}

To prove the first part of Theorem \ref{bullet} it remains to show that the fiber $F_d$ is contractible for any $d\in \mathbb{D}.$ 
This was already observed in the proof of Lemma \ref{smooth}.

For the proof of the second part of Theorem \ref{bullet}, we use a very similar scheme. 
Because of this we give only a brief outline of the proof, mostly focussing on the differences between the two cases.
			
Both polynomial monads $Bimod_{\bullet\to \bullet \leftarrow \bullet}$ and $Wbimod_{\bullet+\bullet}$ have $\mathbb{N}\coprod \mathbb{N}\coprod\mathbb{N} $ as their colours and both involve specifically decorated planar trees. 

For $Bimod_{\bullet\to \bullet \leftarrow \bullet}$, we use trees which now have four different types of vertices: white, black of type $1$ or $2$ and black vertices without any type.  An edge cannot connect two black vertices without type, and the path between any leaf or stump (valency one vertex) and the root should pass through one and only one vertex which is white or black of type $1$ or $2.$ 

\begin{figure}[H]
			\caption{A typical tree  for $Bimod_{\bullet\to \bullet \leftarrow \bullet}$}
			\[
			\begin{tikzpicture}[scale = .7]
			\draw (0,-1) -- (0,0);
			
			\draw (0,0) -- (-6,1);
			\draw (0,0) -- (-2,1);
			\draw (0,0) -- (2,1);
			\draw (0,0) -- (6,1);
			
			\draw (-6,1) -- (-7,2);
			\draw (-6,1) -- (-6,2);
			\draw (-6,1) -- (-5,2);
			
			\draw (-2,1) -- (-3.5,2);
			\draw (-2,1) -- (-2,2);
			\draw (-2,1) -- (-.5,2);
			
			\draw (2,1) -- (1,2);
			\draw (2,1) -- (3,2);
			
			\draw (6,1) -- (4.5,2);
			\draw (6,1) -- (5.5,2);
			\draw (6,1) -- (6.5,2);
			\draw (6,1) -- (7.5,2);
			
			\draw (-3.5,2) -- (-4.5,3);
			\draw (-3.5,2) -- (-3.5,3);
			\draw (-3.5,2) -- (-2.5,3);
			
			\draw (-.5,2) -- (-2,3);
			\draw (-.5,2) -- (-1,3);
			\draw (-.5,2) -- (0,3);
			\draw (-.5,2) -- (1,3);
			
			\draw (3,2) -- (2.3,3);
			\draw (3,2) -- (3.7,3);
			
			\draw[fill] (0,0) circle (4pt);
			\draw[fill=white] (-6,1) circle (4pt);
			\draw[fill] (-2,1) circle (4pt) node[left]{$1$};
			\draw[fill=white] (2,1) circle (4pt);
			\draw[fill] (6,1) circle (4pt) node[right]{$2$};
			\draw[fill] (-3.5,2) circle (4pt);
			\draw[fill] (-.5,2) circle (4pt);
			\draw[fill] (3,2) circle (4pt);
			\end{tikzpicture}
			\]
		\end{figure}

For the monad 	$Wbimod_{\bullet+\bullet}$ we use similar trees subject to the restrictions:
\begin{enumerate}		\item A tree may have at most one white vertex; 			
			\item Black vertices of type $2$ are on the left of the white vertex and black vertices of type $1$ are on the right of the white vertex;
			\item If a tree has all vertices black, the typed vertices must have the same type;
			\item There are exactly three copies of the tree without vertices (a free living edge)  each of them having $0$ as target but of different types: black $1,$ black $2$ or white.  		\end{enumerate} 

\begin{figure}[H]
	\caption{A typical tree for $Wbimod_{\bullet+\bullet}$}
	\[
	\begin{tikzpicture}[scale = .7]
	\draw (0,-1) -- (0,0);
	
	\draw (0,0) -- (-6,1);
	\draw (0,0) -- (-2,1);
	\draw (0,0) -- (2,1);
	\draw (0,0) -- (6,1);
	
	\draw (-6,1) -- (-7,2);
	\draw (-6,1) -- (-6,2);
	\draw (-6,1) -- (-5,2);
	
	\draw (-2,1) -- (-3.5,2);
	\draw (-2,1) -- (-2,2);
	\draw (-2,1) -- (-.5,2);
	
	\draw (2,1) -- (1,2);
	\draw (2,1) -- (3,2);
	
	\draw (6,1) -- (4.5,2);
	\draw (6,1) -- (5.5,2);
	\draw (6,1) -- (6.5,2);
	\draw (6,1) -- (7.5,2);
	
	\draw (-3.5,2) -- (-4.5,3);
	\draw (-3.5,2) -- (-3.5,3);
	\draw (-3.5,2) -- (-2.5,3);
	
	\draw (-.5,2) -- (-2,3);
	\draw (-.5,2) -- (-1,3);
	\draw (-.5,2) -- (0,3);
	\draw (-.5,2) -- (1,3);
	
	\draw (3,2) -- (2.3,3);
	\draw (3,2) -- (3.7,3);
	
	\draw[fill] (0,0) circle (4pt);
	\draw[fill] (-6,1) circle (4pt) node[left]{$2$};
	\draw[fill] (-2,1) circle (4pt) node[left]{$2$};
	\draw[fill=white] (2,1) circle (4pt);
	\draw[fill] (6,1) circle (4pt) node[right]{$1$};
	\draw[fill] (-3.5,2) circle (4pt);
	\draw[fill] (-.5,2) circle (4pt);
	\draw[fill] (3,2) circle (4pt);
	\end{tikzpicture}
	\]
\end{figure}

The map $f: Wbimod_{\bullet + \bullet} \to Bimod_{\bullet\to \bullet \leftarrow \bullet}$ is now obvious. It is the identity on colours and is the natural inclusion on other sets.

The description of the classifier $Bimod_{\bullet\to \bullet \leftarrow \bullet}^{Wbimod_{\bullet+\bullet}}$ follows the usual pattern. The role of constants is crucial as above. In the classifier $Bimod_{\bullet\to \bullet \leftarrow \bullet}^{Wbimod_{\bullet+\bullet}}$, they generate the morphisms of insertion of a stump  on one side of a tree. In particular, we have the following morphisms of replacement:

 \begin{enumerate}\item Morphism replacing a black vertex of type $1$ by a white vertex:
				
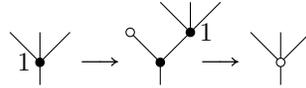
\[
				\begin{tikzpicture}[scale=0.4]
				\draw (0,-.75) -- (0,1);
				\draw (0,0) -- (-1,1);
				\draw (0,0) -- (1,1);
				
				\draw (4,-.75) -- (4,0);
				\draw (4,0) -- (3,1);
				\draw (4,0) -- (5,1);
				
				\draw (5,1) -- (6,2);
				\draw (5,1) -- (5,2);
				\draw (5,1) -- (4,2);
				
				\draw (8,-.75) -- (8,1);
				\draw (8,0) -- (7,1);
				\draw (8,0) -- (9,1);
				
				\draw[fill] (0,0) circle (4pt) node[left]{$1$};
				\draw[fill=white] (3,1) circle (4pt);
				\draw[fill] (5,1) circle (4pt) node[right]{$1$};
				\draw[fill] (4,0) circle (4pt);
				\draw (2,0) node{$\longrightarrow$};
				\draw (6,0) node{$\longrightarrow$};
				\draw[fill=white] (8,0) circle (4pt);
				\end{tikzpicture}
				\]

				\item Morphism replacing a black vertex of type $2$ by a white vertex
				
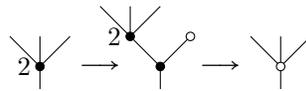
\[
				\begin{tikzpicture}[scale=0.4]
				\draw (0,-.75) -- (0,1);
				\draw (0,0) -- (-1,1);
				\draw (0,0) -- (1,1);
				
				\draw (4,-.75) -- (4,0);
				\draw (4,0) -- (3,1);
				\draw (4,0) -- (5,1);
				
				\draw (3,1) -- (2,2);
				\draw (3,1) -- (3,2);
				\draw (3,1) -- (4,2);
				
				\draw (8,-.75) -- (8,1);
				\draw (8,0) -- (7,1);
				\draw (8,0) -- (9,1);
				
				\draw[fill] (0,0) circle (4pt) node[left]{$2$};
				\draw[fill=white] (5,1) circle (4pt);
				\draw[fill] (3,1) circle (4pt) node[left]{$2$};
				\draw[fill] (4,0) circle (4pt);
				\draw (2,0) node{$\longrightarrow$};
				\draw (6,0) node{$\longrightarrow$};
				\draw[fill=white] (8,0) circle (4pt);
				\end{tikzpicture}
				\]
				
			\end{enumerate}
These two morphism classes are enough to proceed with a proof of contractibility of the classifier $Bimod_{\bullet\to \bullet \leftarrow \bullet}^{Wbimod_{\bullet+\bullet}}$
in exact analogy with the proof for $NOp_{\bullet\to \bullet \leftarrow \bullet}^{Bimod_{\bullet+\bullet}}.$
			
\end{proof}

\begin{remark} \label{RBimod}One can again consider a reduced version of the theory which in this context means that bimodules and weak bimodules we consider have unique operations in degrees $0$ and $1$ as in Turchin's paper \cite{T}[Part 1]. In particular, such weak bimodules are $0$-pointed automatically.  The corresponding classifiers do not contains vertices of valency $1$ and $2$ and, hence, are finite posets.  The following is a picture of the category
			$RBimod_{\bullet\to \bullet \leftarrow \bullet}^{RWbimod_{\bullet+\bullet}}(2_w):$ 
			
			\def\lefttree#1#2#3{
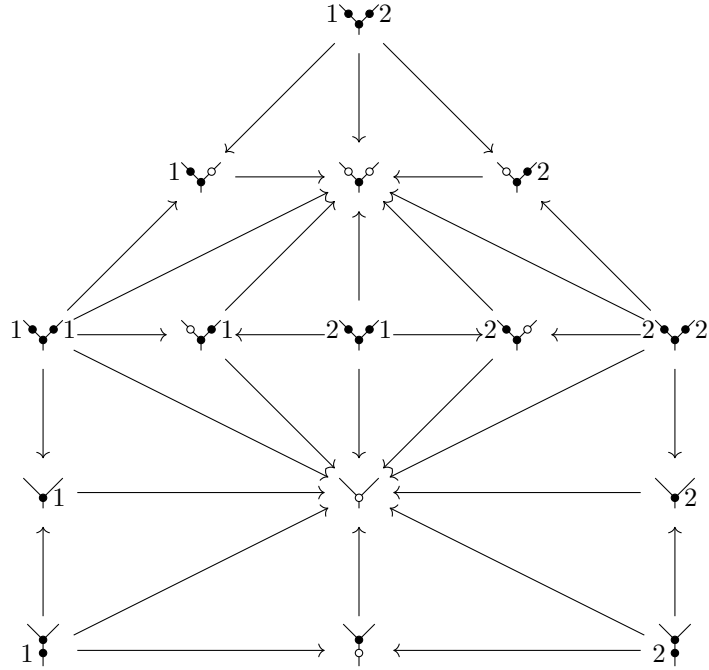
				\begin{scope}[shift={#1}]
					\begin{scope}[shift={(0,0)}]
						\draw (0,.3) -- (0,-.2);
						\draw (0,.3) -- (-.3,.6);
						\draw (0,.3) -- (.3,.6);
						\draw[fill] (0,.3) circle (2pt);
						\ifthenelse
						{
							#2 = 1
						}
						{
							\draw[fill] (0,.05) circle (2pt) node[left]{$1$};
						}
						{
							\ifthenelse
							{
								#2 = 2
							}
							{
								\draw[fill] (0,.05) circle (2pt) node[left]{$2$};
							}
							{
								\draw[fill=white] (0,.05) circle (2pt);
							}
						}
					\end{scope}
				\end{scope}
			}
			
			\def\righttree#1#2#3{
				\begin{scope}[shift={#1}]
					\begin{scope}[shift={(0,0)}]
						\draw (0,0) -- (-.375,.375);
						\draw (0,0) -- (.375,.375);
						\draw (0,0) -- (0,-.2);
						\draw[fill] (0,0) circle (2pt);
						\ifthenelse
						{
							#2 = 1
						}
						{
							\draw[fill] (-.2,.2) circle (2pt) node[left]{$1$};
						}
						{
							\ifthenelse
							{
								#2 = 2
							}
							{
								\draw[fill] (-.2,.2) circle (2pt) node[left]{$2$};
							}
							{
								\draw[fill=white] (-.2,.2) circle (2pt);
							}
						}
						\ifthenelse
						{
							#3 = 1
						}
						{
							\draw[fill] (.2,.2) circle (2pt) node[right]{$1$};
						}
						{
							\ifthenelse
							{
								#3 = 2
							}
							{
								\draw[fill] (.2,.2) circle (2pt) node[right]{$2$};
							}
							{
								\draw[fill=white] (.2,.2) circle (2pt);
							}
						}
					\end{scope}
				\end{scope}
			}
			
			\def\middletree#1#2{
				\begin{scope}[shift={#1}]
					\draw (0,0) -- (-.375,.375);
					\draw (0,0) -- (.375,.375);
					\draw (0,0) -- (0,-.2);
					\ifthenelse
					{
						#2 = 1
					}
					{
						\draw[fill] (0,0) circle (2pt) node[right]{$1$};
					}
					{
						\ifthenelse
						{
							#2 = 2
						}
						{
							\draw[fill] (0,0) circle (2pt) node[right]{$2$};
						}
						{
							\draw[fill=white] (0,0) circle (2pt);
						}
					}
				\end{scope}
			}
			
			\def\line#1#2#3#4{
				\draw[->] ({#1 + 1.3*(#3-#1)/sqrt(((#3-#1)*(#3-#1))+((#4-#2)*(#4-#2)))/2},{#2 + 1.3*(#4-#2)/sqrt(((#3-#1)*(#3-#1))+((#4-#2)*(#4-#2)))/2}) -- ({#3 - 1.3*(#3-#1)/sqrt(((#3-#1)*(#3-#1))+((#4-#2)*(#4-#2)))/2},{#4 - 1.3*(#4-#2)/sqrt(((#3-#1)*(#3-#1))+((#4-#2)*(#4-#2)))/2});
			}

			\begin{figure}[H]\label{fullcategory}
				\[
				\begin{tikzpicture}[scale = 0.7]
				\begin{scope}[shift={(0,.1)}]
				\line{0}{-3}{0}{0};
				
				\line{-6}{-3}{0}{-3};
				\line{6}{-3}{0}{-3};
				\line{-6}{-3}{0}{0};
				\line{6}{-3}{0}{0};
				\line{-6}{-3}{-6}{0};
				\line{6}{-3}{6}{0};
				
				
				\line{-6}{0}{0}{0};
				\line{6}{0}{0}{0};
				
				\line{-6}{3}{0}{0};
				\line{-3}{3}{0}{0};
				\line{0}{3}{0}{0};
				\line{3}{3}{0}{0};
				\line{6}{3}{0}{0};
				
				\line{-6}{3}{-3}{3};
				\line{0}{3}{-3}{3};
				\line{0}{3}{3}{3};
				\line{6}{3}{3}{3};
				
				
				\line{-6}{3}{0}{6};
				\line{-3}{3}{0}{6};
				\line{0}{3}{0}{6};
				\line{3}{3}{0}{6};
				\line{6}{3}{0}{6};
				\line{-3}{6}{0}{6};
				\line{0}{9}{0}{6};
				\line{3}{6}{0}{6};
				
				
				\line{0}{9}{-3}{6};
				\line{0}{9}{3}{6};
				\line{-6}{3}{-3}{6};
				\line{6}{3}{3}{6};
				\line{-6}{3}{-6}{0};
				\line{6}{3}{6}{0};
				\end{scope}
				
				\lefttree{(-6,-3)}{1}{1};
				\lefttree{(0,-3)}{3}{1};
				\lefttree{(6,-3)}{2}{1};
				\middletree{(-6,0)}{1};
				\middletree{(0,0)}{3};
				\middletree{(6,0)}{2};
				\righttree{(0,9)}{1}{2};
				\righttree{(-3,6)}{1}{3};
				\righttree{(0,6)}{3}{3};
				\righttree{(3,6)}{3}{2};
				\righttree{(-6,3)}{1}{1};
				\righttree{(-3,3)}{3}{1};
				\righttree{(0,3)}{2}{1};
				\righttree{(3,3)}{2}{3};
				\righttree{(6,3)}{2}{2};
				\end{tikzpicture}
				\]
			\end{figure}
			
			The nerve of this poset is clearly contractible.
			
		\end{remark}

\section{Second cofinality theorem}

Let $\mathrm{Bimod}_+$ be the category of $1$-pointed $Ass$-bimodules.  Analogously, let  $\mathrm{WBimod}_+$ be the category of $0$-pointed weak $Ass$-bimodules.  These categories are categories of algebras of polynomial monads $Bimod_+$ and $WBimod_+$ respectively obtained by the pushouts described in the Theorem \ref{+point}. 

\begin{theorem}\label{cofinalmaps}
There are homotopically cofinal maps of polynomial monads
$$f: Bimod_+\to NOp_{**}$$
and
$$g:WBimod_+\to Bimod_{**}$$
\end{theorem}

\begin{proof} As before we give the details of the proof of only the first part of the theorem. 
The proof of the part concerning bimodules versus weak bimodules is very similar and we leave it as an exercise. 

Firstly we have to describe the monads  $NOp_{**}$ and $Bimod_+$ explicitly. 
The first monad is represented by a polynomial 
  \[
			\xymatrix{
				\mathbb{N}  && Ptr_{*\mathrm{O}*}^{*} \ar[ll]_-s \ar[rr]^-{p} && Ptr_{*\mathrm{O}*} \ar[rr]^-t && \mathbb{N} \ .			}
			\] 
 Here,   $Ptr_{*\mathrm{O}*}$ is the subset of  $Ptr_{1\mathrm{O}2}$ consisting of trees with a condition that no edge can connect two black vertices of the same type. Similarly, $Ptr_{*\mathrm{O}*}$ is the subset of  $Ptr_{1\mathrm{O}2}^*$ with one white vertex marked (so we do not allow insertion to black vertices). 
 The rest of the description of the monad is very similar to the monad 
 $NOp_{\bullet\to\bullet\leftarrow\bullet}.$   

Analogously, the monad $Bimod_+$ is represented by a polynomial 
 \[
			\xymatrix{
				\mathbb{N}  && Ptr_{*\mathrm{B}*}^{*} \ar[ll]_-s \ar[rr]^-{p} && Ptr_{*\mathrm{B}*} \ar[rr]^-t && \mathbb{N}
			}
			\] 
where 			
$Ptr_{*\mathrm{B}*}\subset  Ptr_{1\mathrm{B}2} $  whose elements are subject to the condition as above. The rest of the structure is also similar to the monad $Bimod_{\bullet+\bullet}.$ It is also clear how to construct a map  			
	\begin{equation}\label{map2} f: Bimod_+\to NOp_{**}.\end{equation} 
We are going to prove that  this map of polynomial monads
is homotopically cofinal. 

We will need another map of polynomial monads intermediate between (\ref{map1}) and (\ref{map2}). 
Namely, let $NOp_{\circ\to \bullet \leftarrow \circ}$ be the monad whose algebras are cospans of non-symmetric operads 
{ $$\mathcal{A}\to\mathcal{C}\leftarrow \mathcal{B}$$ 
where $\mathcal{A}$ and $\mathcal{B}$ are such that $\mathcal{A}(1)\cong \mathcal{B}(1) \cong 1.$} 
\begin{remark} We will call such non-symmetric operads {\it semireduced} because they do not have nontrivial unary operations but may have nontrivial operations of arity $0.$ 
\end{remark}
 Explicitly, such a monad given by a polynomial similar to  (\ref{polyforspan}) but corresponding trees are {\it semireduced}; that is, they do not contain black vertices of valency $2.$  
Also, the colours for black types are now $\mathbb{N}\setminus \{1\}.$ 

Also, let $Bimod_{\circ+\circ}$ be the monad whose algebras are given by triples of two semireduced operads and a $1$-pointed bimodule over them. We have a map of the monads $$f:Bimod_{\circ+\circ}\to  NOp_{\circ\to \bullet \leftarrow \circ} .$$  

All these monads are included in the following commutative diagram of cartesian maps:

\begin{equation}\label{reduction}
			\xymatrix{
				Bimod_{\bullet+\bullet}\ar[rr]^{} \ar[dd]_{}&& Bimod_{\circ+\circ}\ar[rr]^{} \ar[dd]_{}           && Bimod_+\ar[dd]^{} \\
				\\
				NOp_{\bullet\to \bullet \leftarrow \bullet} \ar[rr]^{h} &&   NOp_{\circ\to \bullet \leftarrow \circ} \ar[rr]^{k}       &&NOp_{**}
			}
			\end{equation}	
where the horizontal map $h$ acts on colours as follows: on each $\mathbb{N}_{b_i}, i=1,2$ it is defined by $h(n) = n, n\ne 1$ and $h(1) = 0.$ On $\mathbb{N}_w$ it acts identically.  
On a tree it erases all valency $2$ black points.   
The horizontal map $k$ on colours sends each element from black summand $\mathbb{N}_{b_i}$ to $0$ and it is an identity on $\mathbb{N}_w.$ On operations it sends a tree to the maximal possible contraction of this tree with respect to operadic contractions of black vertices. In other words, it contracts each edge which connects two black vertices of the same type to a black vertex of this type.  The top horizontal maps act similarly.  

The square (\ref{reduction})  induces the following maps of classifiers
$$G: NOp_{\bullet\to \bullet \leftarrow \bullet}^{Bimod_{\bullet+\bullet}}\to h^*(NOp_{\circ\to \bullet \leftarrow \circ}^{Bimod_{\circ+\circ}})$$
$$E: NOp_{\circ\to \bullet \leftarrow \circ}^{Bimod_{\circ+\circ}}\to k^*( NOp_{**}^{Bimod_+})$$

\begin{lem} The underlying functor of  $G$ has a left adjoint $K.$
\end{lem}
\begin{proof}
Let $n$ be one of the colours of the monad $NOp_{\bullet\to \bullet \leftarrow \bullet}.$ We will denote the underlying functor of $G$ restricted to $n$ by the same letter $G$ to simplify the language. Then $G$ acts on objects of $NOp_{\bullet\to \bullet \leftarrow \bullet}^{Bimod_{\bullet+\bullet}}$ by erasing valency $2$ black points. We define $K$ on objects to be an inclusion of semireduced trees to the set of all trees. We claim that this inclusion can be extended to a functor.

Indeed, the set of generators for morphisms in the classifier $NOp_{\circ\to \bullet \leftarrow \circ}^{Bimod_{\circ+\circ}}$ is the same as in Proposition \ref{genrel} except that in (d) we do not allow the insertion of black points on an edge. The relations are also the same as in Proposition \ref{genrel}. 
To define $K$ on morphisms is to define it on generators and we simply can do it by mapping a generator to the corresponding generator. Since relations are the same we get a functor.  

Finally, the unit of the adjunction is the identity and the counit is a morphism $KG(a)\to a$ which inserts all black points back to $a$ after $G$ erases them.  
It is trivial to check that unit and counit satisfy the two triangle relations. 
\end{proof}

\begin{lem} The underlying functor of $E$ is proper and has contractible fibres. 
\end{lem} 
\begin{proof} To see this we first want to prove that the underlying categories of $\mathbb{C} =NOp_{\circ\to \bullet \leftarrow \circ}^{Bimod_{\circ+\circ}}$  and  $\mathbb{D} =NOp_{**}^{Bimod_+}$  are  posets. 

 For an object $a\in \mathbb{C}$, let us denote by $(\mathbb{C},a)$ the full subcategory of $\mathbb{C}$ spanned by the objects $b$ for which there exists a morphism $b\to a.$  Observe that by construction of classifiers the category $(\mathbb{C},a)$ is a product of categories $(\mathbb{C}, { a^v})$ where $v$  runs over the vertices of $a$ and ${ a^v}$ is the corolla determined by $v$ (that is, a corolla whose unique vertex has the same colour as $v$ and whose set of leaves is equal to the set of incoming edges of $v$). Hence, it is enough to prove that each  $(\mathbb{C},{ a^v})$ is a poset. Moreover, it is enough to prove that ${ a^v}$ is the terminal object in this category.
 
 If $v$ is a black vertex of a type $i =1,2$, it is clear that $(\mathbb{C} {a^v})$ is isomorphic to the category of trees and their contractions whose all vertices are black of the same type $i.$ This is a poset with the terminal object ${a^v}$ as easily follows from axioms of semireduced non-symmetric operad.
 
 If $v$ is a white vertex, we can use Lemma \ref{characterisation} to describe the objects of $(\mathbb{C},{ a^v}).$ All morphisms are operadic contractions between black vertices and bimodule contractions for white and black vertices. The fact that ${a^v}$ is terminal here follows from the bimodule relations by a standard combinatorial technique involving the diamond lemma and induction on the number of black vertices.   
 
A similar kind  of proof works  for $\mathbb{D}.$

The fiber $E_d$ of the functor $E$ over any tree $d\in \mathbb{D}$  contains $d$ as well. Moreover, it is very obvious that $d$ is the terminal object of this fiber. So, given $f:d_0\to d_1$ in $\mathbb{C}$ and $c\in E_{d_0}$, an object of the lifting category of $f$ under $c_0$ is just any morphism  $g:c_0\to c_1$ where $c_1\in E_{d_1}.$ Since $d_1$ is the terminal object in $E_{d_1}$, we have a unique morphism from $u:c_1\to d_1$ and hence, a morphism in the lifting category from $g\to g\cdot u.$ Since we are in a poset, the morphism $g\cdot u$ does not depend on $g$. Hence, it is the terminal object of the lifting category. This completes the proof. 
\end{proof} 

\begin{remark} To illustrate the idea of the proof above  one can consider a reduced version of the classifier $NOp_{**}^{Bimod_+}$.
The following picture represents the category $d^*RNOp_{**}^{RBimod_+}(3_w)$ which is the target of the functor 
$G(3_w).$ The source of this functor is represented by the category shown at Remark \ref{RNOP}.

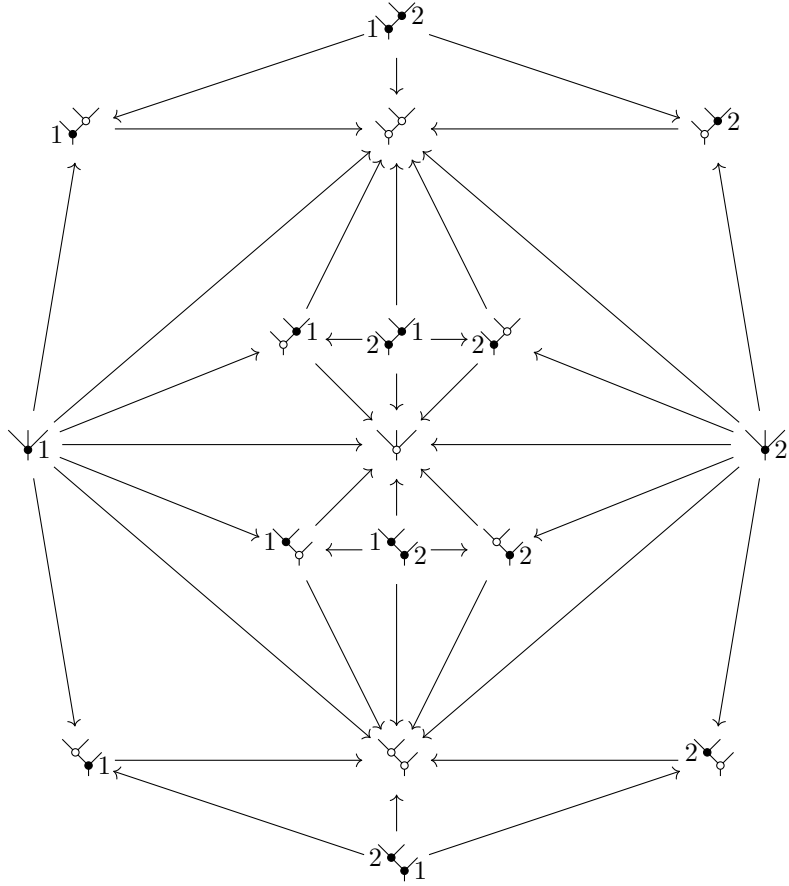
			\def\lefttree#1#2#3{
				\begin{scope}[shift={#1}]
					\begin{scope}[shift={(.15,0)}]
						\draw (0,0) -- (-.5,.5);
						\draw (0,0) -- (.25,.25);
						\draw (-.25,.25) -- (0,.5);
						\draw (0,0) -- (0,-.2);
						\ifthenelse
						{
							#2 = 1
						}
						{
							\draw[fill] (0,0) circle (2pt) node[right]{$1$};
						}
						{
							\ifthenelse
							{
								#2 = 2
							}
							{
								\draw[fill] (0,0) circle (2pt) node[right]{$2$};
							}
							{
								\draw[fill=white] (0,0) circle (2pt);
							}
						}
						\ifthenelse
						{
							#3 = 1
						}
						{
							\draw[fill] (-.25,.25) circle (2pt) node[left]{$1$};
						}
						{
							\ifthenelse
							{
								#3 = 2
							}
							{
								\draw[fill] (-.25,.25) circle (2pt) node[left]{$2$};
							}
							{
								\draw[fill=white] (-.25,.25) circle (2pt);
							}
						}
					\end{scope}
				\end{scope}
			}
			
			\def\righttree#1#2#3{
				\begin{scope}[shift={#1}]
					\begin{scope}[shift={(-.15,0)}]
						\draw (0,0) -- (.5,.5);
						\draw (0,0) -- (-.25,.25);
						\draw (.25,.25) -- (0,.5);
						\draw (0,0) -- (0,-.2);
						\ifthenelse
						{
							#2 = 1
						}
						{
							\draw[fill] (0,0) circle (2pt) node[left]{$1$};
						}
						{
							\ifthenelse
							{
								#2 = 2
							}
							{
								\draw[fill] (0,0) circle (2pt) node[left]{$2$};
							}
							{
								\draw[fill=white] (0,0) circle (2pt);
							}
						}
						\ifthenelse
						{
							#3 = 1
						}
						{
							\draw[fill] (.25,.25) circle (2pt) node[right]{$1$};
						}
						{
							\ifthenelse
							{
								#3 = 2
							}
							{
								\draw[fill] (.25,.25) circle (2pt) node[right]{$2$};
							}
							{
								\draw[fill=white] (.25,.25) circle (2pt);
							}
						}
					\end{scope}
				\end{scope}
			}
			
			\def\middletree#1#2{
				\begin{scope}[shift={#1}]
					\draw (0,0) -- (-.375,.375);
					\draw (0,0) -- (0,.375);
					\draw (0,0) -- (.375,.375);
					\draw (0,0) -- (0,-.2);
					\ifthenelse
					{
						#2 = 1
					}
					{
						\draw[fill] (0,0) circle (2pt) node[right]{$1$};
					}
					{
						\ifthenelse
						{
							#2 = 2
						}
						{
							\draw[fill] (0,0) circle (2pt) node[right]{$2$};
						}
						{
							\draw[fill=white] (0,0) circle (2pt);
						}
					}
				\end{scope}
			}
			
			\def\line#1#2#3#4{
				\draw[->] ({#1 + 1.3*(#3-#1)/sqrt(((#3-#1)*(#3-#1))+((#4-#2)*(#4-#2)))/2},{#2 + 1.3*(#4-#2)/sqrt(((#3-#1)*(#3-#1))+((#4-#2)*(#4-#2)))/2}) -- ({#3 - 1.3*(#3-#1)/sqrt(((#3-#1)*(#3-#1))+((#4-#2)*(#4-#2)))/2},{#4 - 1.3*(#4-#2)/sqrt(((#3-#1)*(#3-#1))+((#4-#2)*(#4-#2)))/2});
			}

		\begin{figure}[H]\label{fullcategory}
		\[
		\begin{tikzpicture}[scale = 0.7]
			\begin{scope}[shift={(0,.1)}]
				\line{0}{8}{-6}{6};
				\line{0}{8}{0}{6};
				\line{0}{8}{6}{6};
				
				\line{-6}{6}{0}{6};
				\line{6}{6}{0}{6};
				
				\line{-2}{2}{0}{6};
				\line{-2}{2}{0}{0};
				\line{0}{2}{-2}{2};
				\line{0}{2}{0}{6};
				\line{0}{2}{0}{0};
				\line{0}{2}{2}{2};
				\line{2}{2}{0}{6};
				\line{2}{2}{0}{0};
				
				\line{-7}{0}{-6}{6};
				\line{-7}{0}{0}{6};
				\line{-7}{0}{-2}{2};
				\line{-7}{0}{0}{0};
				\line{-7}{0}{-2}{-2};
				\line{-7}{0}{-6}{-6};
				\line{-7}{0}{0}{-6};
				
				\line{7}{0}{6}{6};
				\line{7}{0}{0}{6};
				\line{7}{0}{2}{2};
				\line{7}{0}{0}{0};
				\line{7}{0}{2}{-2};
				\line{7}{0}{6}{-6};
				\line{7}{0}{0}{-6};
				
				\line{-2}{-2}{0}{-6};
				\line{-2}{-2}{0}{0};
				\line{0}{-2}{-2}{-2};
				\line{0}{-2}{0}{-6};
				\line{0}{-2}{0}{0};
				\line{0}{-2}{2}{-2};
				\line{2}{-2}{0}{-6};
				\line{2}{-2}{0}{0};
				
				\line{-6}{-6}{0}{-6};
				\line{6}{-6}{0}{-6};
				
				\line{0}{-8}{-6}{-6};
				\line{0}{-8}{0}{-6};
				\line{0}{-8}{6}{-6};
			\end{scope}
		
		\lefttree{(0,-8)}{1}{2};
		
		\lefttree{(-6,-6)}{1}{3};
		\lefttree{(0,-6)}{3}{3};
		\lefttree{(6,-6)}{3}{2};
		
		\lefttree{(0,-2)}{2}{1};
		
		\lefttree{(-2,-2)}{3}{1};
		\lefttree{(2,-2)}{2}{3};
		
		\middletree{(-7,0)}{1};
		\middletree{(0,0)}{3};
		\middletree{(7,0)}{2};
		
		\righttree{(-2,2)}{3}{1};
		\righttree{(2,2)}{2}{3};
		
		\righttree{(0,2)}{2}{1};
		
		\righttree{(-6,6)}{1}{3};
		\righttree{(0,6)}{3}{3};
		\righttree{(6,6)}{3}{2};
		
		\righttree{(0,8)}{1}{2};
		\end{tikzpicture}
		\]
	\end{figure}
		
The functor $G$ simply `shrinks' certain faces of the polytope from \ref{RNOP}. This corresponds exactly to the contraction discribed by Turchin in \cite{T}[p. 34].  

\end{remark}

We can finish the proof of Theorem \ref{cofinalmaps} by observing that $N(G)$ is a weak equivalence (since $G$ is a right adjoint) and $N(E)$ is a weak equivalence (since it is proper with contractible fiber).  
\end{proof}

\section{Dwyer-Hess-Turchin's delooping 	theorems}
	
	Our goal is to prove the following theorem first established independently in \cite{DH} and \cite{T}.
	\begin{theorem}\label{DHT1} For any simplicial multiplicative operad $\mathcal{O}$, there exists a fibration sequence of simplicial sets
	$$\Omega Map_{\mathrm{NOp}}({ u^*}Ass,u^*\mathcal{O})\to Map_{\mathrm{Bimod}}({ v^*} Ass,{v}^*\mathcal{O})\to fib(\mathcal{O}_1 ).$$
	
	 For any simplicial pointed bimodule  $\mathcal{B}$, there exists a fibration sequence of simplicial sets
	$$\Omega Map_{\mathrm{Bimod}}({{ v^*} }Ass,{b}^*\mathcal{B})\to Map_{\mathrm{WBimod}}({{w^*v^*} } Ass,{w^* b^*} \mathcal{B})\to fib(\mathcal{B}_0 ).$$	

	\end{theorem} 
	
This theorem implies the explicit double delooping formula of Dwyer-Hess-Turchin:
	\begin{corol}\label{DHT3}Let a simplicial multiplicative operad  $\mathcal{O}$ be such that  $\mathcal{O}_0$ and $\mathcal{O}_1$ are contractible. Then 	there is a weak equivalence of simplicial spaces
	$$\holim_{\Delta}(\mathcal{O}^*) \sim \Omega^2 Map_{\mathrm{NOp}}(u^*Ass,u^*\mathcal{O}),$$
	where $\mathcal{O}^*$ is the cosimplicial object associated to multiplicative operad $\mathcal{O}.$	\end{corol}
	
\begin{proof}[Proof of Theorem \ref{DHT1}]  Since  $f: Bimod_+\to NOp_{**}$ is homotopically cofinal the result will follow from Theorems \ref{pointedalgebras} and \ref{+point} if we know that 
the squares \ref{squarenop} and \ref{squarebimod} are  homotopically cofinal. This is the content of two lemmas below.

The proof for the second fibration sequence is similar.
\end{proof}




	\begin{lemma}
	The square
	\begin{equation}\label{squarenop}
		\xymatrix{
				NOp \ar[rr] \ar[dd] && NOp_* \ar[dd] \\
				\\
				NOp_* \ar[rr] && NOp_{**}
		}
	\end{equation}
	is homotopically cofinal.
\end{lemma}

\begin{proof}
	According to the Theorem \ref{recognition}, we have to prove that the nerve of $NOp_{**}^{\oint F}$ is contractible, where $F$ is the presheaf of polynomial monads given by the span $$NOp_* \leftarrow NOp \to NOp_*.$$
	
	Using Proposition \ref{dddd} and Theorem \ref{recognition}, the nerve of $NOp^{\oint NOp}$ is contractible, where $NOp$ is the constant presheaf $$NOp \leftarrow NOp \to NOp.$$
	

	It will be enough  to prove that $NOp_{**}^{\oint F}$ and $NOp^{\oint NOp}$ as categories over $\mathbb{N}$ are connected by a sequence of adjoint functors.

	According to the Remark \ref{desc} the objects of the category  $NOp_{**}^{\oint F}$ are given by trees with two types of black vertices $1,2$ and three types of white vertices $0,1,2$, with the condition that no edge can connect two black vertices of the same type.
	
	The morphisms are generated by:
	\begin{itemize}
		\item transformations of a black vertex to a white vertex of the same type or transformations of a white vertex of type $0$ to a white vertex of type $1$ or to a white vertex of type $2$ :
		
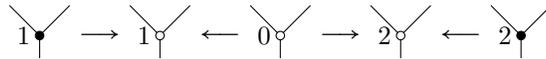
\[
		\begin{tikzpicture}[scale=0.4]
		\draw (-8,-.75) -- (-8,0);
		\draw (-8,0) -- (-7,1);
		\draw (-8,0) -- (-9,1);

		\draw (-4,-.75) -- (-4,0);
		\draw (-4,0) -- (-3,1);
		\draw (-4,0) -- (-5,1);

		\draw (0,-.75) -- (0,0);
		\draw (0,0) -- (-1,1);
		\draw (0,0) -- (1,1);
		
		\draw (4,-.75) -- (4,0);
		\draw (4,0) -- (3,1);
		\draw (4,0) -- (5,1);
				
		\draw (8,-.75) -- (8,0);
		\draw (8,0) -- (7,1);
		\draw (8,0) -- (9,1);
		
		\draw[fill] (-8,0) circle (4pt) node[left]{$1$};
		\draw[fill=white] (-4,0) circle (4pt) node[left]{$1$};
		\draw[fill=white] (0,0) circle (4pt) node[left]{$0$};
		\draw[fill=white] (4,0) circle (4pt) node[left]{$2$};
		\draw[fill] (8,0) circle (4pt) node[left]{$2$};
		\draw (-6,0) node{$\longrightarrow$};
		\draw (-2,0) node{$\longleftarrow$};
		\draw (2,0) node{$\longrightarrow$};
		\draw (6,0) node{$\longleftarrow$};
		\end{tikzpicture}
		\]
		\item operadic contractions of edges connecting white vertices of the same type
	\end{itemize}
	
	The category $NOp^{\oint NOp}$ is the full subcategory of $NOp_{**}^{\oint F}$ containing trees with only white vertices.
	
	We define $alt\left( NOp^{\oint NOp} \right)$ as the full subcategory of $NOp^{\oint NOp}$ containing trees with the condition that no edge can connect two white vertices of type $1$ or two white vertices of type $2$. Similarly, we define $alt\left( NOp_{**}^{\oint F} \right)$.
	
	Then the inclusion
	\[
		alt\left( NOp^{\oint NOp} \right) \hookrightarrow NOp^{\oint NOp}
	\]
	has a left adjoint that sends a tree to the same tree but where the edges connecting two whites vertices of type $1$ or two white vertices of type $2$.
	
	Similarly, there is an adjunction between $alt\left( NOp_{**}^{\oint F} \right)$ and $NOp_{**}^{\oint F}$.
	
	Finally, the inclusion
	\[
		alt\left( NOp^{\oint NOp} \right) \hookrightarrow alt\left( NOp_{**}^{\oint F} \right)
	\]
	has a left adjoint that sends a tree to the same tree but where the black vertices have been turned into whites vertices of the same type.
	
	In summary, we have a sequence of adjunctions
	\[
		\begin{tikzcd}
		NOp^{\oint NOp} \ar[r,bend left,"",""{name=A, below}] & alt\left( NOp^{\oint NOp} \right) \ar[l,bend left,"",""{name=B,above}] \ar[from=A, to=B, symbol=\dashv] \ar[r,bend left,"",""{name=A, below}] & alt\left( NOp_{**}^{\oint F} \right) \ar[l,bend left,"",""{name=B,above}] \ar[from=A, to=B, symbol=\vdash] \ar[r,bend left,"",""{name=A, below}] & NOp_{**}^{\oint F} \ar[l,bend left,"",""{name=B,above}] \ar[from=A, to=B, symbol=\vdash]
		\end{tikzcd}
	\]
	which concludes the proof.
\end{proof}

	\begin{lemma}
	The square
	\begin{equation}\label{squarebimod}
	\xymatrix{
		Id \ar[rr] \ar[dd] && Bimod \ar[dd] \\
		\\
		Id_+ \ar[rr] && Bimod_+
	}
	\end{equation}
	is homotopically cofinal.
\end{lemma}

\begin{proof}
This time we have to prove that the nerve of $Bimod_+^{\oint F}$ is contractible, where $F$ is the presheaf of polynomial monads given by the span
	\[
		Id_+ \leftarrow Id \to Bimod.
	\]
	
As in the previous Lemma we will exhibit a string of adjunctions  connecting the category $Bimod_+^{\oint F}$ (over $\mathbb{N}$) and a subcategory of $Bimod_+^{\oint F}$ whose nerve is contractible.
	
	The objects in $Bimod_+^{\oint F}$ are trees with two types of black vertices $1,2$ and three types of white vertices $0,1,2$. The conditions are
	\begin{itemize}
		\item the white vertices of type $0$ and $1$ can occur only if the vertex has valency 2, that is only one incoming edge
		\item there can be no edge between two black vertices of the same type
		\item any path from a leaf and the root meets possibly a black vertex of type $1$, then possibly a white vertex of any type, then possibly a black vertex of type $2$
	\end{itemize}

	Remark that these conditions imply that if a vertex is black of type $2$, it can only be the root vertex.
	
	The morphisms in $Bimod_+^{\oint F}$ are generated by:
	\begin{itemize}
		\item transformations of a white vertex of type $0$ to a white vertex of type $1$ or to a white vertex of type $2$
		\[
\begin{tikzpicture}[scale=0.4]
\draw (-4,-1) -- (-4,1);

\draw (0,-1) -- (0,1);

\draw (4,-1) -- (4,1);

\draw[fill=white] (-4,0) circle (4pt) node[left]{$1$};
\draw[fill=white] (0,0) circle (4pt) node[left]{$0$};
\draw[fill=white] (4,0) circle (4pt) node[left]{$2$};
\draw (-2,0) node{$\longleftarrow$};
\draw (2,0) node{$\longrightarrow$};
\end{tikzpicture}
\]
		\item addition of an unary white vertex of type $1$ above a black vertex of type $2$ or below a black vertex of type $1$, as long as the tree obtained is still in the set of objects, for example
		
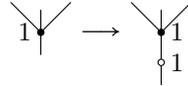
\[
		\begin{tikzpicture}[scale=0.4]
		\draw (0,-.75) -- (0,.75);
		\draw (0,0) -- (-1,1);
		\draw (0,0) -- (1,1);
		
		\draw (4,-1.75) -- (4,.75);
		\draw (4,0) -- (3,1);
		\draw (4,0) -- (5,1);
		
		\draw[fill] (0,0) circle (3pt) node[left]{$1$};
		\draw[fill=white] (4,-1) circle (3pt) node[right]{$1$};
		\draw[fill] (4,0) circle (3pt) node[right]{$1$};
		\draw (2,0) node{$\longrightarrow$};
		\end{tikzpicture}
		\]
%
%
		\item bimodules operations when there are black vertices of type $1$ above white vertices of type $2$ or when there are white vertices of type $2$ above a black vertex of type $2$
				
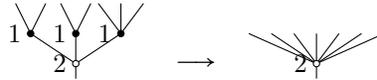
\[
		\begin{tikzpicture}[scale=0.4]
		\draw (0,-.5) -- (0,0);
		\draw (0,0) -- (-1.5,1);
		\draw (0,0) -- (0,1);
		\draw (0,0) -- (1.5,1);
		
		\draw (-1.5,1) -- (-2,2);
		\draw (-1.5,1) -- (-1,2);
		
		\draw (0,1) -- (-.5,2);
		\draw (0,1) -- (.5,2);
		
		\draw (1.5,1) -- (.75,2);
		\draw (1.5,1) -- (1.5,2);
		\draw (1.5,1) -- (2.25,2);
		
		\begin{scope}[shift={(8,0)}]
		\draw (0,0) -- (0,-.5);
		\draw (0,0) -- (-2.25,1);
		\draw (0,0) -- (-1.5,1);
		\draw (0,0) -- (-.75,1);
		\draw (0,0) -- (0,1);
		\draw (0,0) -- (.75,1);
		\draw (0,0) -- (1.5,1);
		\draw (0,0) -- (2.25,1);
		\draw[fill=white] (0,0) circle (3pt) node[left]{$2$};
		\end{scope}
		
		\draw[fill=white] (0,0) circle (3pt) node[left]{$2$};
		\draw[fill] (-1.5,1) circle (3pt) node[left]{$1$};
		\draw[fill] (0,1) circle (3pt) node[left]{$1$};
		\draw[fill] (1.5,1) circle (3pt) node[left]{$1$};
		\draw (4,0) node{$\longrightarrow$};
		\end{tikzpicture}
		\]
		or
				
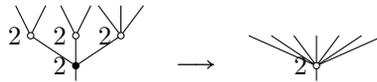
\[
		\begin{tikzpicture}[scale=0.4]
		\draw (0,-.5) -- (0,0);
		\draw (0,0) -- (-1.5,1);
		\draw (0,0) -- (0,1);
		\draw (0,0) -- (1.5,1);
		
		\draw (-1.5,1) -- (-2,2);
		\draw (-1.5,1) -- (-1,2);
		
		\draw (0,1) -- (-.5,2);
		\draw (0,1) -- (.5,2);
		
		\draw (1.5,1) -- (.75,2);
		\draw (1.5,1) -- (1.5,2);
		\draw (1.5,1) -- (2.25,2);
		
		\draw[fill] (0,0) circle (3pt) node[left]{$2$};
		\draw[fill=white] (-1.5,1) circle (3pt) node[left]{$2$};
		\draw[fill=white] (0,1) circle (3pt) node[left]{$2$};
		\draw[fill=white] (1.5,1) circle (3pt) node[left]{$2$};
		
		\draw (4,0) node{$\longrightarrow$};
		
		\begin{scope}[shift={(8,0)}]
		\draw (0,0) -- (0,-.5);
		\draw (0,0) -- (-2.25,1);
		\draw (0,0) -- (-1.5,1);
		\draw (0,0) -- (-.75,1);
		\draw (0,0) -- (0,1);
		\draw (0,0) -- (.75,1);
		\draw (0,0) -- (1.5,1);
		\draw (0,0) -- (2.25,1);
		\draw[fill=white] (0,0) circle (3pt) node[left]{$2$};
		\end{scope}
		\end{tikzpicture}
		\]
	\end{itemize}
	
	First, let us define $W_{012}$ as the full subcategory of $Bimod_+^{\oint F}$ containing the trees for which the path from any leaf to the root vertex contains exactly one white vertex.
	
	The inclusion
	\[
		W_{012} \hookrightarrow Bimod_+^{\oint F}
	\]
	has a left adjoint given which sends a tree to same tree but where we add a unary white vertex of type $1$ on all the paths from a leaf to the root vertex which do not contain a white vertex.
	
	Now, we define $W_2$ as the full subcategory of $W_{012}$ containing trees where white vertices are only of type $2$. It is obvious that $W_2$ is isomorphic to $Bimod^{Bimod}$.
	
We then have a sequence of adjunctions
	\[
			\begin{tikzcd}
	W_2 \ar[r,bend left,"",""{name=A, below}] & W_{02} \ar[l,bend left,"",""{name=B,above}] \ar[from=A, to=B, symbol=\vdash] \ar[r,bend left,"",""{name=A, below}] & W_{012} \ar[l,bend left,"",""{name=B,above}] \ar[from=A, to=B, symbol=\dashv]
	\end{tikzcd}
	\]
	where $W_{02}$ is the full subcategory of $W_{012}$ containing trees where white vertices are only of type $0$ or $2$ and the functors $W_{02} \to W_2$ and $W_{012} \to W_{02}$ turn white vertices of type $0$ to white vertices of type $2$ and white vertices of type $1$ to white vertices of type $0$ respectively.
	
	This concludes the proof.
\end{proof}

\noindent {\bf Acknowledgements.}    We  wish to express our  gratitude to
 B. Fresse, A.Lazarev, R. Street,  S. Lack, R. Garner, B. Shoikhet, V. Turchin, for many useful discussions.

The  first author also  gratefully acknowledges  the financial
support  of Max Planck
Institut f\"{u}r Mathematik. 
\renewcommand{\refname}{Bibliography.}

\

\

\end{document}